\documentclass[10pt]{amsart}

\usepackage[utf8]{inputenc}
\usepackage[english]{babel}
\mathchardef\mhyphen="2D

\usepackage{amsthm, amsmath, amssymb, stmaryrd, mathrsfs, tikz, MnSymbol, bbm, mathtools, amsfonts}
\usepackage{enumerate}
\usepackage{relsize}
\usepackage{blindtext}

\usepackage{hyperref}
\hypersetup{
    colorlinks=true,
    linkcolor=blue,
    citecolor=red
}
\usepackage{cleveref}

\providecommand*\showkeyslabelformat[1]

\usepackage[notref,notcite]{showkeys}

\usetikzlibrary{decorations.pathreplacing,angles,quotes}

\usetikzlibrary{matrix, arrows}

\usepackage[letterpaper, margin=1.1in]{geometry}

\usepackage{lipsum}
\usepackage[titletoc,title]{appendix}
  
\theoremstyle{definition}
\newtheorem{theorem}{Theorem}
\newtheorem{definition}[theorem]{Definition}
\newtheorem{definitions}[theorem]{Definitions}
\newtheorem{lemma}[theorem]{Lemma}

\newtheorem{proposition}[theorem]{Proposition}
\newtheorem{examples}[theorem]{Examples}
\newtheorem{example}[theorem]{Example}

\newtheorem{remark}[theorem]{Remark}

\newtheorem{notation}[theorem]{Notation}
\newtheorem*{unpackdef}{Unpacked Definition}

\newtheorem*{compositionc}{The composition condition for $\omega$-categories}
\newtheorem*{coherencec}{The coherence condition for $\omega$-categories}

\newtheorem*{coherencen}{The coherence condition for $n$-categories}
\newtheorem*{coher+comp}{The $n$-cell composition coherence condition for $n$-categories}
\newtheorem{coherencethm}[theorem]{The Coherence Theorem}

\numberwithin{theorem}{section}

\def\bs{\ensuremath\boldsymbol}

\title{Presentations for globular operads}

\author{Rhiannon Griffiths}

\begin{document}

\begin{abstract} In this paper we develop the theory of presentations for globular operads and construct presentations for the globular operads corresponding to several key theories of $n$-category for $n \leqslant 4$.
\end{abstract}

\maketitle

\tableofcontents

\section{Introduction}

Operads are tools that have been used to describe a wide variety of algebraic structures. They first arose as a way to understand operations on $k$-fold loop spaces in homotopy theory \cite{PM}, but the idea has since been adapted and applied across many areas of mathematics; see, for example, \cite{VH}, \cite{MSS}, or \cite{BV}. Each type of operad describes a class of algebraic theories, and a specified operad of said type concisely encodes one such theory into a single object.

Globular operads are a kind of operad whose algebras share a strong formal similarity with higher categories. This approach to higher categories has been worked on extensively by Michael Batanin and by Tom Leinster, who defined fully weak $n$-categories as algebras for a specified globular operad \cite{MB}, \cite{TL}. Yet the current literature offers no way to find the globular operad corresponding to a given notion of higher category, nor a proof that such an operad must exist.

In this paper we define presentations for globular operads and demonstrate how these presentations provide a way to explicitly construct the globular operad corresponding to any algebraic notion of higher category. The underlying idea is the same as in presentations for simpler objects like groups or rings; we describe an algebraic structure, in this case a globular operad, by specifying a set of generators and a set of relations between them. In particular, we show how to construct a presentation for a globular operad in such a way that generators correspond to the kinds of composition operations and coherence cells present in the associated higher categories, and the relations correspond to the axioms.

Additionally, we show that a presentation can be built in such a way that the coherence theorem for the associated higher categories is satisfied automatically; see Section \ref{contractibility section} for a discussion of the coherence theorem. The highest dimension for which there exists a hands-on definition of a fully weak algebraic $n$-category together with a proven coherence theorem is $n=3$; these are the tricategories of Nick Gurski \cite{NG}. In the final section of this paper, we give a presentation for the globular operad for fully weak 4-categories satisfying the coherence theorem.

While this method of using operads to build concrete models for algebraic higher categories has many potential applications, this work is done with a specific application in mind. In a preprint of Michael Batanin \cite{batanin}, he conjectures that it should possible to take `slices' of globular operads. The $k^{th}$ slice was said to be the symmetric operad obtained by considering only the $k$-dimensional data. Thus, given a globular operad equivalent to some notion of higher category, the slices would isolate the algebraic structure of those higher categories in each dimension.

However, due to the gaps in knowledge surrounding globular operads at the time, it was not possible to formulate a definition of slices. As a first application of presentations, we will show in the follow up paper \cite{RG} that given a presentation $P$ for a globular operad $\bs{G}$, there exists a symmetric operad determined by the $k$-dimensional data of $P$; this symmetric operad is the $k^{th}$ slice of $\bs{G}$. Following this, we use slices to formally construct the string and surface diagrams arising from several key theories of higher category and show that, up to isomorphism, the slices do not depend on the choice of presentation.

Batanin also hypothesised that slices could tell us when one theory of higher category is equivalent to another. This is significant because fully weak higher categories are often the most useful for applications to areas such as algebraic topology and homotopy theory, but become too complex for practical use in dimensions greater than 2. A solution is to find a notion of semi-strict higher category that is just weak enough to be equivalent to the fully weak variety, while still being tractable enough to work with directly. In Section \ref{examples section} of this paper, we construct presentations for the globular operads for two different theories of semi-strict 4-category. The first of these are $4$-categories with weak units in low dimensions, and the second are $4$-categories with weak interchange laws. In \cite{RG} we show that both are equivalent to fully weak $4$-categories by studying the geometric properties and graphical calculi of the associated surface diagrams. It is likely that using the language of presentations, these results can be generalised to dimensions greater than 4.

\subsection{Organisation of this paper} In Section \ref{free higher cat monad section} we describe the free strict higher category monads used to define globular operads, which are covered in detail in Section \ref{globop section}. In Section \ref{algebras section} we define the category of algebras for a globular operad, and highlight the similarites between these algebras and higher categories. In Section \ref{contractibility section} we discuss contractiblity, and show that contractible globular operads are precisely those operads whose algebras satisfy the required properties of composition and coherence for higher categories. Presentations are defined in Section \ref{presenations section}, followed by an example of a presentation for the globular operad for strict 2-categories. In Section \ref{weak unbiased section} we construct presentations for the globular operads for weak unbiased higher categories, which are used in Section \ref{glob ops for higher cats section} to identify precisely when a globular operad is equivalent to some theory of higher category. The final section is devoted to constructing presentations for the globular operads for several key theories of $n$-category for $n \leqslant 4$; Section \ref{weaksection} covers fully weak 4-categories, and Section \ref{semi-strict section} covers $n$-categories with weak units in low dimensions and $n$-categories with weak interchange laws, respectively.

\subsection*{Acknowledgements} I would like to thank my advisor Nick Gurski for his exceptional guidance and support throughout the writing of this paper. I am also grateful to Michael Batanin and Richard Garner for conversations related to this project. Finally, the material in this paper also appears in my PhD thesis \cite{RG1}.

\section{The free strict higher category monads}\label{free higher cat monad section}

This paper concerns algebraic\footnote{By an algebraic notion of higher category we mean a notion of higher category for which composition and its associated coherence are given by specified operations satisfying equational axioms.} notions of higher category for which the underlying graph data is a globular set (in the case of $\omega$-categories), or an $n$-globular set (in the case of $n$-categories). There are other definitions of higher category for which the underlying data is given by a more complex structure, opetopic and simplicial definitions for example, but we will not study these here. We begin by defining strict higher categories in terms of monads on categories of globular sets.

\begin{definition} A \textit{globular set} $\bs{G}$ is a diagram

\vspace{1.5mm}

\begin{center}
\begin{tikzpicture}[node distance=2cm, auto]

\node (A) {$G_0$};
\node (B) [left of=A] {$G_1$};
\node (C) [left of=B] {$G_2$};
\node (D) [left of=C] {$...$};

\draw[transform canvas={yshift=0.5ex},->] (B) to node {$s$} (A);
\draw[transform canvas={yshift=-0.5ex},->] (B) to node [swap] {$t$} (A);

\draw[transform canvas={yshift=0.5ex},->] (C) to node {$s$} (B);
\draw[transform canvas={yshift=-0.5ex},->] (C) to node [swap] {$t$} (B);

\draw[transform canvas={yshift=0.5ex},->] (D) to node {$s$} (C);
\draw[transform canvas={yshift=-0.5ex},->] (D) to node [swap] {$t$} (C);

\end{tikzpicture}
\end{center}
in $\mathbf{Set}$ satisfying $ss=st$ and $ts=tt:G_{n+2} \longrightarrow G_n$ for all $n \in \mathbb{N}$.
\end{definition}

We refer to the elements of $G_n$ as the $n$-cells of $\bs{G}$, and to $s$ and $t$ as the source and target maps, respectively. An $n$-cell of a globular set may be represented diagrammatically using its source and target $k$-cells for all $k<n$. For example, a 2-cell $\chi$ with $s(\chi) = x$, $t(\chi) = x'$, $ss(\chi) = st(\chi) =X$ and $ts(\chi) = tt(\chi) = Y$ is represented

\begin{center}
\begin{tikzpicture}[node distance=2.5cm, auto]

\node (A) {$X$};
\node (B) [right of=A] {$Y$.};
\node (X) [node distance=1.25cm, right of=A] {$\Downarrow \chi$};

\draw[->, bend left=40] (A) to node {$x$} (B);
\draw[->, bend right=40] (A) to node [swap] {$x'$} (B);

\end{tikzpicture}
\end{center}

\begin{definition}
A \textit{morphism $f : \bs{G} \longrightarrow \bs{H}$ of globular sets} is a collection $\{f_n : G_n \longrightarrow H_n\}_{n \in \mathbb{N}}$ of functions preserving the sources and targets.
\end{definition}

\begin{remark}\label{GSet is presheaf} Note that the category $\mathbf{GSet}$ of globular sets is the presheaf category $[\mathbb{G}^{\textrm{op}}, \mathbf{Set}]$, where $\mathbb{G}$ is generated by the graph

\begin{center}
\begin{tikzpicture}[node distance=2cm, auto]

\node (A) {$...$};
\node (B) [left of=A] {$2$};
\node (C) [left of=B] {$1$};
\node (D) [left of=C] {$0$};

\draw[transform canvas={yshift=0.5ex},->] (B) to node {$s$} (A);
\draw[transform canvas={yshift=-0.5ex},->] (B) to node [swap] {$t$} (A);

\draw[transform canvas={yshift=0.5ex},->] (C) to node {$s$} (B);
\draw[transform canvas={yshift=-0.5ex},->] (C) to node [swap] {$t$} (B);

\draw[transform canvas={yshift=0.5ex},->] (D) to node {$s$} (C);
\draw[transform canvas={yshift=-0.5ex},->] (D) to node [swap] {$t$} (C);

\end{tikzpicture}
\end{center}
subject to the equations $ss=ts$ and $st=tt$.
\end{remark}

\begin{definition} The \textit{free strict $\omega$-category functor} $(-)^*:\mathbf{GSet} \longrightarrow \mathbf{Str} \ \omega \mhyphen \mathbf{Cat}$ is the left adjoint to the canonical forgetful functor $U{:} \ \mathbf{Str} \ \omega \mhyphen \mathbf{Cat} \longrightarrow \mathbf{GSet}$ from the category of strict $\omega$-categories; see \cite[Appendix F]{TL}.
\end{definition}

The free strict $\omega$-category $\bs{G}^*$ on a globular set $\bs{G}$ is the $\omega$-category whose $n$-cells are $n$-pasting diagrams in $\bs{G}$ and whose composition is concatenation along matching boundary cells. For instance, the 2-pasting diagrams

\begin{center}
\begin{tikzpicture}[node distance=2.5cm, auto]

\node (A) {$X$};
\node (B) [right of=A] {$Y$};
\node (C) [right of=B] {$Z$};
\node (D) [right of=C] {$U$};

\draw[->, bend left=80] (A) to node {$x$} (B);
\draw[->, bend left=25] (A) to node {} (B);
\draw[->, bend right=25] (A) to node [swap] {} (B);
\draw[->, bend right=80] (A) to node [swap] {$x'''$} (B);

\node (Y) [node distance=1.25cm, right of=A] {$\Downarrow  \chi' $};
\node (X) [node distance=0.63cm, above of=Y] {$\Downarrow  \chi $};
\node (Z) [node distance=0.63cm, below of=Y] {$\Downarrow  \chi'' $};

\draw[->] (B) to node {$y$} (C);

\draw[->, bend left=60] (C) to node {$z$} (D);
\draw[->] (C) to node {} (D);
\draw[->, bend right=60] (C) to node [swap] {$z''$} (D);

\node (W) [node distance=1.25cm, right of=C] {};
\node (U) [node distance=0.4cm, above of=W] {$\Downarrow  \zeta $};
\node (V) [node distance=0.4cm, below of=W] {$\Downarrow  \zeta' $};

\end{tikzpicture}

\begin{tikzpicture}[node distance=2.5cm, auto]

\node (A) {$X$};
\node (B) [right of=A] {$Y$};
\node (C) [right of=B] {$Z$};
\node (D) [right of=C] {$U$};

\draw[->] (A) to node {$x'''$} (B);

\draw[->, bend left=60] (B) to node {$y$} (C);
\draw[->] (B) to node {} (C);
\draw[->, bend right=60] (B) to node [swap] {$y''$} (C);

\node (W) [node distance=1.25cm, right of=B] {};
\node (X) [node distance=0.4cm, above of=W] {$\Downarrow  \psi$};
\node (Y) [node distance=0.4cm, below of=W] {$\Downarrow  \psi'$};

\draw[->, bend left=40] (C) to node {$z''$} (D);
\draw[->, bend right=40] (C) to node [swap] {$z'''$} (D);

\node (Z) [node distance=1.3cm, right of=C] {$\Downarrow  \zeta''$};

\end{tikzpicture}
\end{center}
in $\bs{G}$ are both 2-cells of $\bs{G}^*$ whose vertical composite is the 2-cell below.

\vspace{1mm}

\begin{center}
\begin{tikzpicture}[node distance=2.5cm, auto]

\node (A) {$X$};
\node (B) [right of=A] {$Y$};
\node (C) [right of=B] {$Z$};
\node (D) [right of=C] {$U$};

\draw[->, bend left=80] (A) to node {$x$} (B);
\draw[->, bend left=25] (A) to node {} (B);
\draw[->, bend right=25] (A) to node [swap] {} (B);
\draw[->, bend right=80] (A) to node [swap] {$x'''$} (B);

\node (Y) [node distance=1.25cm, right of=A] {$\Downarrow  \chi'$};
\node (X) [node distance=0.63cm, above of=Y] {$\Downarrow  \chi$};
\node (Z) [node distance=0.63cm, below of=Y] {$\Downarrow  \chi''$};

\draw[->, bend left=60] (B) to node {$y$} (C);
\draw[->] (B) to node {} (C);
\draw[->, bend right=60] (B) to node [swap] {$y''$} (C);

\node (W) [node distance=1.25cm, right of=B] {};
\node (U) [node distance=0.4cm, above of=W] {$\Downarrow  \psi$};
\node (V) [node distance=0.4cm, below of=W] {$\Downarrow  \psi'$};

\draw[->, bend left=80] (C) to node {$z$} (D);
\draw[->, bend left=25] (C) to node {} (D);
\draw[->, bend right=25] (C) to node [swap] {} (D);
\draw[->, bend right=80] (C) to node [swap] {$z'''$} (D);

\node (M) [node distance=1.25cm, right of=C] {$\Downarrow  \zeta'$};
\node (L) [node distance=0.63cm, above of=M] {$\Downarrow  \zeta$};
\node (N) [node distance=0.63cm, below of=M] {$\Downarrow  \zeta''$};

\end{tikzpicture}
\end{center}
Here it is understood that $\chi, \chi', \chi'', \psi, \psi', \zeta, \zeta'$ and $\zeta''$ are all 2-cells of $\bs{G}$.

\begin{definition} \label{degenerate pd}
A \textit{degenerate} $n$-pasting diagram in a globular set $\bs{G}$ is an $n$-pasting diagram which does not contain any $n$-cells of $\bs{G}$.
\end{definition}

\begin{examples}
The following 2-pasting diagrams are both degenerate.

\vspace{1mm}

\begin{center}
\begin{tikzpicture}[node distance=2cm, auto]

\node (A) {$X$};
\node (B) [right of=A] {$Y$};
\node (C) [right of=B] {$Z$};
\node (D) [right of=C] {$U$};
\node (X) [node distance=2cm, left of=A] {$X$};

\draw[->] (A) to node {$x$} (B);
\draw[->] (B) to node {$y$} (C);
\draw[->] (C) to node {$z$} (D);

\end{tikzpicture}
\end{center}
\end{examples}

\begin{remark}
Each $k$-pasting diagram in a globular set $\bs{G}$ may be viewed as a degenerate $n$-pasting diagram for any $n > k$. However, the sets $G^*_k$ and $G^*_n$ of $k$-cells of $\bs{G}^*$ and $n$-cells of $\bs{G}^*$, respectively, are disjoint. In other words, a degenerate $n$-pasting diagram is distinct from its corresponding ($n-1)$-pasting diagram. Note that the degenerate $n$-pasting diagrams in $\bs{G}$ are the identity $n$-cells of $\bs{G}^*$.
\end{remark}

\begin{definition} Let $n$ be natural number. The $n$-\textit{ball} $\bs{B_n}$ is the globular set

\vspace{1mm}

\begin{center}
\begin{tikzpicture}[node distance=2.3cm, auto]

\node (A) {$\{0,1\}$};
\node (C) [left of=A] {$. . .$};
\node (D) [left of=C] {$\{0,1\}$};
\node (E) [left of=D] {$ \{ \filledstar \}$};
\node (F) [node distance=2.1cm, left of=E] {$\emptyset$};
\node (G) [node distance=2cm, left of=F] {$\emptyset$};
\node (H) [node distance=2cm, left of=G] {$. . .$};

\draw[transform canvas={yshift=0.5ex},->] (C) to node {\small ${0}$} (A);
\draw[transform canvas={yshift=-0.5ex},->] (C) to node [swap] {\small ${1}$} (A);

\draw[transform canvas={yshift=0.5ex},->] (D) to node {\small ${0}$} (C);
\draw[transform canvas={yshift=-0.5ex},->] (D) to node [swap] {\small ${1}$} (C);

\draw[transform canvas={yshift=0.5ex},->] (E) to node {\small ${0}$} (D);
\draw[transform canvas={yshift=-0.5ex},->] (E) to node [swap] {\small ${1}$} (D);

\draw[transform canvas={yshift=0.5ex},->] (F) to node {} (E);
\draw[transform canvas={yshift=-0.5ex},->] (F) to node [swap] {} (E);

\draw[transform canvas={yshift=0.5ex},->] (G) to node {} (F);
\draw[transform canvas={yshift=-0.5ex},->] (G) to node [swap] {} (F);

\draw[transform canvas={yshift=0.5ex},->] (H) to node {} (G);
\draw[transform canvas={yshift=-0.5ex},->] (H) to node [swap] {} (G);

\end{tikzpicture}
\end{center}
consisting of a single $n$-cell. Here the arrows labelled $0$ and $1$ represent the constant functions.
\end{definition}

\begin{definition} \label{trivlal pd}
A \textit{simple} $n$-pasting diagram in a globular set $\bs{G}$ is the image of a morphism $\bs{B_n} \longrightarrow \bs{G}$ of globular sets. 
\end{definition}

\begin{examples}
Every 0-pasting diagram is simple, and the pasting diagrams below are simple 1-, 2- and 3-pasting diagrams, respectively.
\begin{center}
\begin{tikzpicture}[node distance=2.5cm, auto]

\node (A) {$X$};
\node (B) [node distance=2cm, right of=A] {$Y$};

\draw[->] (A) to node {$x$} (B);

\node (A') [node distance=2cm, right of=B] {$X$};
\node (B') [right of=A'] {$Y$};
\node (X') [node distance=1.25cm, right of=A'] {$\Downarrow \chi$};

\draw[->, bend left=40] (A') to node {$x$} (B');
\draw[->, bend right=40] (A') to node [swap] {$x'$} (B');

\node (C) [node distance=2cm, right of=B'] {$X$};
\node (D) [node distance=3cm, right of=C] {$Y$};

\draw[->, bend left=48] (C) to node {$x$} (D);
\draw[->, bend right=48] (C) to node [swap] {$x'$} (D);

\node (W) [node distance=1.2cm, right of=C] {};
\node (U) [node distance=0.7cm, above of=W] {};
\node (V) [node distance=0.7cm, below of=W] {};
\node (W') [node distance=1.8cm, right of=C] {};
\node (U') [node distance=0.7cm, above of=W'] {};
\node (V') [node distance=0.7cm, below of=W'] {};

\draw[->, bend right] (U) to node [swap] {$\chi$} (V);
\draw[->, bend left] (U') to node {$\tilde{\chi}$} (V');

\node(c) [node distance=0.3cm, right of=W] {};
\node() [node distance=0.15cm, above of=c] {$\Lambda$};
\node() [node distance=0.15cm, below of=c] {$\Rrightarrow$};

\end{tikzpicture}
\end{center}
\end{examples}

\begin{definition} The \textit{free strict $\omega$-category monad} $(-)^* = \big((-)^*,\eta, \mu \big)$ on $\mathbf{GSet}$ is the monad arising from the adjunction $(-)^* \dashv U \colon \ \mathbf{Str} \ \omega \mhyphen \mathbf{Cat} \longrightarrow \mathbf{GSet}$.
\end{definition}

\begin{notation}
We denote by $(-)^* \mhyphen \mathbf{Alg}$ the category of algebras for the monad $(-)^*$.
\end{notation}

It is shown in \cite[Appendix F]{TL} that the forgetful functor $U{:} \ \mathbf{Str} \ \omega \mhyphen \mathbf{Cat} \longrightarrow \mathbf{GSet}$ is monadic, so the free-forgetful adjunctions 

\begin{center}
\begin{tikzpicture}[node distance=3cm, auto]

\node (A) {$\mathbf{GSet}$};
\node (B) [ right of=A] {$\mathbf{Str} \ \omega \mhyphen \mathbf{Cat}$};
\node () [node distance=1.5cm, right of=A] {$\perp$};

\draw[->, bend left=40] (A) to node {} (B);
\draw[->, bend left=40] (B) to node {} (A);

\node (A') [node distance=3cm, right of=B] {$\mathbf{GSet}$};
\node (B') [right of=A'] {$(-)^* \mhyphen \mathbf{Alg}$};
\node () [node distance=1.5cm, right of=A'] {$\perp$};

\draw[->, bend left=40] (A') to node {} (B');
\draw[->, bend left=40] (B') to node {} (A');

\end{tikzpicture}
\end{center}
are equivalent. In fact, a $(-)^{\ast}$-algebra structure on a globular set $\bs{G}$ is precisely a strict $\omega$-category with underlying globular set $\bs{G}$, and a morphism of $(-)^{\ast}$-algebras is precisely a strict $\omega$-functor. This can be seen by unpacking the definitions:

\begin{unpackdef}
An algebra for $(-)^{\ast}$ is a morphism $\theta: \bs{G}^{\ast} \longrightarrow \bs{G}$ of globular sets satisfying unit and multiplication axioms. Each $\theta_n$ composes $n$-pasting diagrams in $\bs{G}$ into single $n$-cells of $\bs{G}$ in a way that is consistent with the sources and targets. 

\vspace{1mm}

\begin{center}
\begin{tikzpicture}[node distance=2.5cm, auto]

\node (A) {$X$};
\node (B) [right of=A] {$Y$};
\node (C) [right of=B] {$Z$};
\node (D) [right of=C] {$U$};

\draw[->, bend left=80] (A) to node {$x$} (B);
\draw[->, bend left=25] (A) to node {} (B);
\draw[->, bend right=25] (A) to node [swap] {} (B);
\draw[->, bend right=80] (A) to node [swap] {$x'''$} (B);

\node (Y) [node distance=1.25cm, right of=A] {$\Downarrow  \chi' $};
\node (X) [node distance=0.63cm, above of=Y] {$\Downarrow  \chi $};
\node (Z) [node distance=0.63cm, below of=Y] {$\Downarrow  \chi'' $};

\draw[->] (B) to node {$y$} (C);

\draw[->, bend left=60] (C) to node {$z$} (D);
\draw[->] (C) to node {} (D);
\draw[->, bend right=60] (C) to node [swap] {$z''$} (D);

\node (W) [node distance=1.25cm, right of=C] {};
\node (U) [node distance=0.4cm, above of=W] {$\Downarrow  \zeta $};
\node (V) [node distance=0.4cm, below of=W] {$\Downarrow  \zeta' $};

\node (A') [node distance=3.3cm, right of=D] {$X$};
\node (B') [node distance=4.4cm, right of=A'] {$U$};
\node (X') [node distance=2.2cm, right of=A'] {\small $\Downarrow \theta_2(\chi, \chi',\chi'', y, \zeta, \zeta')$};

\draw[->, bend left=40] (A') to node {$\theta_1(x,y,z)$} (B');
\draw[->, bend right=40] (A') to node [swap] {$\theta_1(x''',y, z'')$} (B');

\node (x) [node distance=0.5cm, right of=D] {};
\node (y) [node distance=0.5cm, left of=A'] {};
\draw[|->, dashed] (x) to node {$\theta_2$} (y);

\end{tikzpicture}
\end{center}
Note that $\theta_0$ is the identity function on 0-cells; this follows from the unit axiom, which says that $\theta_n$ sends simple $n$-pasting diagrams in $\bs{G}$ to the corresponding $n$-cell of $\bs{G}$. The binary composition operations on $n$-cells present in an $\omega$-category are the result of applying $\theta_n$ to $n$-pasting diagrams consisting of a single pair of $n$-cells sharing a $k$-cell boundary. This is demonstrated by the examples below.

\vspace{1mm}

\begin{center}
\begin{tikzpicture}[node distance=2.5cm, auto]

\node (A) {$X$};
\node (B) [node distance=2cm, right of=A] {$Y$};
\node (C) [node distance=2cm, right of=B] {$Z$};
\node (x) [node distance=0.5cm, right of=C] {};

\draw[->] (A) to node {$x$} (B);
\draw[->] (B) to node {$y$} (C);

\node (A') [node distance=3.3cm, right of=C] {$X$};
\node (B') [right of=A'] {$Z$};
\node (y) [node distance=0.5cm, left of=A'] {};

\draw[->] (A') to node {$xy$} (B');
\draw[|->, dashed] (x) to node {$\theta_1$} (y);

\node (F) [node distance=2cm, below of=C] {$Z$};
\node (E) [left of=F] {$Y$};
\node (D) [left of=E] {$X$};
\node () [node distance=1.25cm, right of=D] {$\Downarrow \chi$};
\node () [node distance=1.25cm, right of=E] {$\Downarrow \psi$};
\node (p) [node distance=0.5cm, right of=F] {};

\draw[->, bend left=40] (D) to node {$x$} (E);
\draw[->, bend right=40] (D) to node [swap] {$x'$} (E);
\draw[->, bend left=40] (E) to node {$y$} (F);
\draw[->, bend right=40] (E) to node [swap] {$y'$} (F);

\node (D') [node distance=3.3cm, right of=F] {$X$};
\node (E') [node distance=2.8cm, right of=D'] {$Z$};
\node () [node distance=1.4cm, right of=D'] {$\Downarrow \chi \ast \psi$};
\node (q) [node distance=0.5cm, left of=D'] {};

\draw[->, bend left=40] (D') to node {$xy$} (E');
\draw[->, bend right=40] (D') to node [swap] {$x'y'$} (E');
\draw[|->, dashed] (p) to node {$\theta_2$} (q);

\node (H') [node distance=2.7cm, below of=F] {$Y$};
\node (G') [left of=H'] {$X$};
\node (u) [node distance=1.25cm, right of=G'] {};
\node () [node distance=0.4cm, above of=u] {$\Downarrow \chi$};
\node () [node distance=0.4cm, below of=u] {$\Downarrow \chi'$};
\node (j) [node distance=0.5cm, right of=H'] {};

\draw[->, bend left=60] (G') to node {$x$} (H');
\draw[->] (G') to node {} (H');
\draw[->, bend right=60] (G') to node [swap] {$x''$} (H');

\node (G) [node distance=3.3cm, right of=H'] {$X$};
\node (H) [node distance=2.8cm, right of=G] {$Y$};
\node () [node distance=1.4cm, right of=G] {$\Downarrow \chi \cdot \chi'$};
\node (k) [node distance=0.5cm, left of=G] {};

\draw[->, bend left=40] (G) to node {$x$} (H);
\draw[->, bend right=40] (G) to node [swap] {$x''$} (H);
\draw[|->, dashed] (j) to node {$\theta_2$} (k);

\end{tikzpicture}
\end{center}

\noindent The images of degenerate $n$-pasting diagrams (Definition \ref{degenerate pd}) under $\theta_n$ are the identity $n$-cells.

\vspace{1mm}

\begin{center}
\begin{tikzpicture}[node distance=2cm, auto]

\node (A) {$X$};

\node (X) [node distance=3.3cm, right of=A] {$X$};
\node (Y) [node distance=2.25cm, right of=X] {$X$};

\draw[->] (X) to node {$1_X$} (Y);

\node (x) [node distance=0.5cm, right of=A] {};
\node (y) [node distance=0.5cm, left of=X] {};
\draw[|->, dashed] (x) to node {$\theta_1$} (y);

\node (A') [node distance=1.7cm, below of=A] {$X$};

\node (X') [node distance=3.3cm, right of=A'] {$X$};
\node (Y') [node distance=2.5cm, right of=X'] {$X$};
\node () [node distance=1.25cm, right of=X'] {$\Downarrow 1_{1_X}$};

\draw[->, bend left=40] (X') to node {$1_X$} (Y');
\draw[->, bend right=40] (X') to node [swap] {$1_X$} (Y');

\node (x') [node distance=0.5cm, right of=A'] {};
\node (y') [node distance=0.5cm, left of=X'] {};
\draw[|->, dashed] (x') to node {$\theta_2$} (y');

\node (C'') [node distance=2.4cm, below of=A'] {$Z$};
\node (B'') [left of=C''] {$Y$};
\node (A'') [left of=B''] {$X$};

\draw[->] (A'') to node {$x$} (B'');
\draw[->] (B'') to node {$y$} (C'');

\node (X'') [node distance=3.3cm, right of=C''] {$X$};
\node (Y'') [node distance=2.5cm, right of=X''] {$Z$};
\node () [node distance=1.25cm, right of=X''] {$\Downarrow 1_{xy}$};

\draw[->, bend left=40] (X'') to node {$xy$} (Y'');
\draw[->, bend right=40] (X'') to node [swap] {$xy$} (Y'');

\node (x'') [node distance=0.5cm, right of=C''] {};
\node (y'') [node distance=0.5cm, left of=X''] {};
\draw[|->, dashed] (x'') to node {$\theta_2$} (y'');

\end{tikzpicture}
\end{center}

\noindent The algebra structure map $\theta$ therefore equips $\bs{G}$ with the data of a strict $\omega$-category. The multiplication axiom is equivalent to the axioms for a strict $\omega$-category; it says that given any $n$-pasting diagram in $\bs{G}$ there is exactly one way to compose it into a single $n$-cell. For instance, given a 1-pasting diagram $(x,y,z)$ we have
\vspace{1mm}
$$ x(yz) = \theta_1\big( x, \theta_1(y,z) \big) = \theta_1\big( \theta_1(x), \theta_1(y,z) \big) = \theta_1(x,y,z) = \theta_1 \big( \theta_1(x,y), \theta_1(z) \big) = \theta_1 \big( \theta_1(x,y),z \big) = (xy)z$$
\vspace{-3mm}

\noindent so 1-cell composition satisfies the associativity axiom. The third and fourth equalities in the expression above are instances of the multiplication axiom, and the second and fifth equations are given by the unit axiom. As another example, given a simple 1-pasting diagram $(x)$ in $\bs{G}$ the multiplication axiom yields the following equalities,
$$1_Xx = \theta_1(1_X, x) = \theta_1\big(\theta_1(X), \theta_1(x)\big) = \theta_1(x) = x$$ 
$$x1_Y = \theta_1(x, 1_Y) = \theta_1\big(\theta_1(x), \theta_1(Y)\big) = \theta_1(x) = x$$
\vspace{-3mm}

\noindent so composition of 1-cells satisfies the identity axioms. 
\end{unpackdef}

\begin{unpackdef}
A morphism $f:(\bs{G}, \theta) \longrightarrow (\bs{H}, \sigma)$ of algebras for $(-)^*$ is a morphism $f:\bs{G} \longrightarrow \bs{H}$ of the underlying globular sets such that the following diagram commutes.

\vspace{1mm}

\begin{center}
\begin{tikzpicture}[node distance=2.4cm, auto]

\node (X){$\bs{G}^*$};
\node (Y) [right of=X] {$\bs{H}^*$};
\node (Z) [node distance=2cm, below of=X] {$\bs{G}$};
\node (W) [right of=Z] {$\bs{H}$};

\draw[->] (X) to node {$f^*$} (Y);
\draw[->] (X) to node [swap] {$\theta$} (Z);
\draw[->] (Z) to node [swap] {$f$} (W);
\draw[->] (Y) to node {$\sigma$} (W);

\end{tikzpicture}
\end{center}

\noindent This means that a morphism of $(-)^*$-algebras is a collection of functions $\{ f_n:G_n \longrightarrow H_n\}_{n \in \mathbb{N}}$ strictly preserving sources, targets, and the composition of pasting diagrams. Such a morphism is precisely a strict $\omega$-functor.
\end{unpackdef}

The free strict $n$-category monad, which by abuse of notation we also denote by $(-)^*$, is defined analogoulsy; we just replace globular sets with $n$-globular sets:

\begin{definition}\label{GSet_n}
The \textit{category $\mathbf{GSet}_{\bs{n}}$ of $n$-globular sets} is the presheaf category $[\mathbb{G}_n, \mathbf{Set}]$, where $\mathbb{G}_n$ is generated by the graph
\begin{center}
\begin{tikzpicture}[node distance=2cm, auto]

\node (A) {$n$};
\node (B) [left of=A] {$...$};
\node (C) [left of=B] {$1$};
\node (D) [left of=C] {$0$};

\draw[transform canvas={yshift=0.5ex},->] (B) to node {$s$} (A);
\draw[transform canvas={yshift=-0.5ex},->] (B) to node [swap] {$t$} (A);

\draw[transform canvas={yshift=0.5ex},->] (C) to node {$s$} (B);
\draw[transform canvas={yshift=-0.5ex},->] (C) to node [swap] {$t$} (B);

\draw[transform canvas={yshift=0.5ex},->] (D) to node {$s$} (C);
\draw[transform canvas={yshift=-0.5ex},->] (D) to node [swap] {$t$} (C);

\end{tikzpicture}
\end{center}
subject to the equations $ss=ts$ and $st=tt$.
\end{definition}

In keeping with the infinite dimensional case, a $(-)^*$-algebra on an $n$-globular set $\bs{G_n}$ is precisely a strict $n$-category with underlying $n$-globular set $\bs{G_n}$, and a morphism of $(-)^*$-algebras is precisely a strict $n$-functor. 

\section{Globular operads}\label{globop section}

Having seen how to define strict higher categories in terms of monads on $\mathbf{GSet}$ and $\mathbf{GSet}_{\bs{n}}$, we would like a similar way to define weaker varieties of higher category. This can be done using ($n$-)globular operads, which are defined using the free strict higher category monads.

\begin{notation} We denote by $\mathbf{1}$ the terminal globular set given by the following diagram in $\mathbf{Set}$.

\vspace{1.5mm}

\begin{center}
\begin{tikzpicture}[node distance=2cm, auto]

\node (A) {$\{ \filledstar \}$};
\node (B) [left of=A] {$\{ \filledstar \}$};
\node (C) [left of=B] {$\{ \filledstar \}$};
\node (D) [left of=C] {$...$};

\draw[transform canvas={yshift=0.5ex},->] (B) to node {} (A);
\draw[transform canvas={yshift=-0.5ex},->] (B) to node [swap] {} (A);

\draw[transform canvas={yshift=0.5ex},->] (C) to node {} (B);
\draw[transform canvas={yshift=-0.5ex},->] (C) to node [swap] {} (B);

\draw[transform canvas={yshift=0.5ex},->] (D) to node {} (C);
\draw[transform canvas={yshift=-0.5ex},->] (D) to node [swap] {} (C);

\end{tikzpicture}
\end{center}
\end{notation}

The $n$-cells of $\bs{1}^*$ are $n$-pasting diagrams in $\bs{1}$; observe that there is no need to label the individual cells within these pasting diagrams since each $k$-cell represents the single $k$-cell of $\bs{1}$. A typical 2-cell $\pi$ of $\bs{1}^*$ is illustrated below.
\begin{center}
\begin{tikzpicture}[node distance=1.5cm, auto]

\node (A) {$\cdot$};
\node (B) [node distance=2.5cm, right of=A] {$\cdot$};
\node (C) [node distance=2.5cm, right of=B] {$\cdot$};
\node (D) [node distance=2.5cm, right of=C] {$\cdot$};
\node (D') [node distance=0.7cm, left of=A] {$\pi \ =$};

\draw[->, bend left=85] (A) to node {} (B);
\draw[->, bend left=25] (A) to node {} (B);
\draw[->, bend right=25] (A) to node [swap] {} (B);
\draw[->, bend right=85] (A) to node [swap] {} (B);

\node (XB) [node distance=1.25cm, right of=A] {$\Downarrow $};
\node (XA) [node distance=0.6cm, above of=XB] {$\Downarrow $};
\node (XC) [node distance=0.6cm, below of=XB] {$\Downarrow $};

\draw[->] (B) to node {} (C);

\draw[->, bend left=60] (C) to node {} (D);
\draw[->] (C) to node {} (D);
\draw[->, bend right=60] (C) to node [swap] {} (D);

\node (W) [node distance=1.25cm, right of=C] {};
\node (YA) [node distance=0.38cm, above of=W] {$\Downarrow $};
\node (YC) [node distance=0.38cm, below of=W] {$\Downarrow $};

\end{tikzpicture}
\end{center}

\noindent We may think of the $n$-cells of $\bs{1}^*$ as the collection of possible shapes of $n$-pasting diagrams in a globular set. 

\begin{notation}\label{boundary notation}
Since the source and target $(n{-}1)$-cells of each $n$-cell of $\bs{1}^*$ are equal, we write ${\partial}$ rather than $s$ or $t$ when we want to refer to the source or target of a cell. For example, given the 2-cell $\pi$ of $\bs{1}^*$ depicted above, ${\partial} \pi$ denotes the following 1-cell of $\bs{1}^*$.

\vspace{1.5mm}

\begin{center}
\begin{tikzpicture}[node distance=1.5cm, auto]

\node (A) {$\cdot$};
\node (B) [node distance=2.2cm, right of=A] {$\cdot$};
\node (C) [node distance=2.2cm, right of=B] {$\cdot$};
\node (D) [node distance=2.2cm, right of=C] {$\cdot$};
\node (D') [node distance=0.7cm, left of=A] {${\partial} \pi \ =$};

\draw[->] (A) to node {} (B);
\draw[->] (B) to node {} (C);
\draw[->] (C) to node {} (D);

\end{tikzpicture}
\end{center}

\end{notation}

\begin{definition} 
The \textit{category $\mathbf{GColl}$ of globular collections} is the slice category $\mathbf{GSet} / \bs{1}^*$. 
\end{definition}

\begin{definition} The functor $\circ:\mathbf{GColl} \times \mathbf{GColl} \longrightarrow \mathbf{GColl}$ sends a pair $(g:\bs{G} \longrightarrow \bs{1}^*, \ h:\bs{H} \longrightarrow \bs{1}^*)$ of globular collections to the composite of the left hand diagonals in the diagram

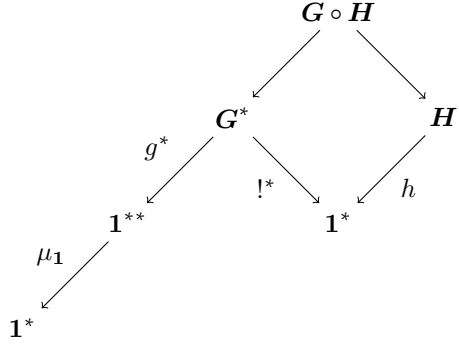
\vspace{1.5mm}

\begin{center}
\begin{tikzpicture}[node distance=1.4cm, auto]

\node (A) {$\bs{G} \circ \bs{H}$};
\node (TA) [left of=A, below of=A] {$\bs{G}^*$};
\node (X) [left of=TA, below of=TA] {$\bs{1}^{**}$};
\node (Y) [left of=X, below of=X] {$\bs{1}^*$};
\node (B) [right of=A, below of=A] {$\bs{H}$};
\node (C) [node distance=2.8cm, below of=A] {$\bs{1}^*$};

\draw[->] (A) to node [swap] {} (TA);
\draw[->] (A) to node {} (B);
\draw[->] (TA) to node [swap] {$!^*$} (C);
\draw[->] (B) to node {$h$} (C);
\draw[->] (TA) to node [swap] {$g^*$} (X);
\draw[->] (X) to node [swap] {$\mu_{\bs{1}}$} (Y);

\end{tikzpicture}
\end{center}

\noindent where the upper square is a pullback square, and is defined on pairs of morphisms using the universal property of pullbacks.
\end{definition}

Recall that a natural transformation is \textit{cartesian} if all of its naturality squares are pullback squares.

\begin{definition}
A monad $T=(T,\eta, \mu)$ on a category $\mathcal{S}$ is \textit{cartesian} if
\begin{enumerate}[i)]
\item $\mathcal{S}$ has pullbacks,
\item $T$ preserves pullbacks, and
\item $\eta$ and $\mu$ are cartesian.
\end{enumerate}
\end{definition}

The free strict $\omega$-category monad $(-)^*$ on $\mathbf{GSet}$ is cartesian \cite[Appendix F]{TL}. Consequently, $(\bf{GColl}, \circ)$ is a monoidal category with unit $\eta_{\bs{1}}:\bs{1} \longrightarrow \bs{1}^*$. The coherence isomorphisms are defined using the universal property of pullbacks. The same property is then used to verify the axioms for a monoidal category. The construction of the left and right unit isomorphisms is straightforward, and uses the fact that $\eta$ is cartesian. The construction of the associativity isomorphisms is slighty more complex and uses that $(-)^*$ preserves pullbacks and $\mu$ is cartesian.

\begin{definition} A \textit{globular operad} is a monoid in the monoidal category $(\bf{GColl}, \circ)$.
\end{definition}

\begin{unpackdef}
A globular operad $\bs{G} = (\bs{G}, g, ids, comp)$ is a morphism $g:\bs{G} \longrightarrow \bs{1}^*$ of globular sets together with an identity map $ids:\bs{1} \longrightarrow \bs{G}$ and a composition map $comp:\bs{G} \circ \bs{G} \longrightarrow \bs{G}$ making the following diagrams commute

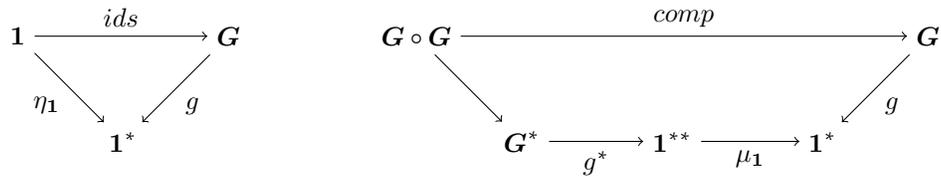
\begin{center}
\begin{tikzpicture}[node distance=2.8cm, auto]

\node (X) {$\bs{1}$};
\node (Y) [right of=X] {$\bs{G}$};
\node (A) [node distance=1.4cm, below of=X, right of=X] {$\bs{1}^*$};

\node (U) [node distance=2.5cm, right of=Y] {$\bs{G} \circ \bs{G}$};
\node (B) [node distance=1.4cm, below of=U, right of=U] {$\bs{G}^*$};
\node (C) [node distance=2cm, right of=B] {$\bs{1}^{**}$};
\node (D) [node distance=2cm, right of=C] {$\bs{1}^*$};
\node (V) [node distance=1.4cm, above of=D, right of=D] {$\bs{G}$};

\draw[->] (X) to node {$ids$} (Y);
\draw[->] (X) to node [swap] {$\eta_{\bs{1}}$} (A);
\draw[->] (Y) to node {$g$} (A);

\draw[->] (U) to node {$comp$} (V);
\draw[->] (U) to node [swap] {} (B);
\draw[->] (B) to node [swap] {$g^*$} (C);
\draw[->] (C) to node [swap] {$\mu_{\bs{1}}$} (D);
\draw[->] (V) to node {$g$} (D);

\end{tikzpicture}
\end{center}

\noindent and satisfying identity and associativity axioms. 
We think of each $n$-cell $\Lambda$ of $\bs{G}$ as an abstract operation composing $n$-pasting diagrams of shape $g(\Lambda)$ into single $n$-cells. For example, if $\chi$ is a 2-cell of $\bs{G}$ and $g(\chi) = \tau$,

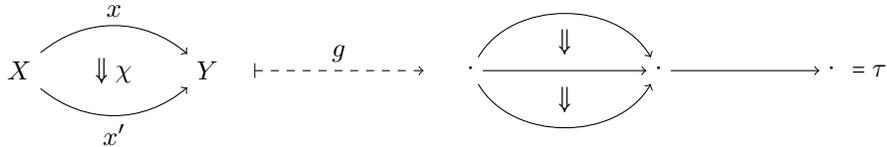
\begin{center}\label{tau}
\begin{tikzpicture}[node distance=2.5cm, auto]

\node (B) {$\cdot$};
\node (C) [right of=B] {$\cdot$};
\node (D) [node distance=2.3cm, right of=C] {$\cdot$};
\node (D') [node distance=0.5cm, right of=D] {$= \tau$}; 

\draw[->] (C) to node {} (D);

\draw[->, bend left=60] (B) to node {} (C);
\draw[->] (B) to node {} (C);
\draw[->, bend right=60] (B) to node [swap] {} (C);

\node (W) [node distance=1.25cm, right of=B] {};
\node (YA) [node distance=0.38cm, above of=W] {$\Downarrow $};
\node (YC) [node distance=0.38cm, below of=W] {$\Downarrow $};

\node (X) [node distance=3.5cm, left of=B] {$Y$};
\node (Y) [left of=X] {$X$};
\draw[->, bend left=40] (Y) to node {$x$} (X);
\draw[->, bend right=40] (Y) to node [swap] {$x'$} (X);
\node (Z) [node distance=1.25cm, right of=Y] {$\Downarrow \chi$};

\node (P) [node distance=0.5cm, right of=X] {};
\node (Q) [node distance=0.5cm, left of=B] {};
\draw[|->, dashed] (P) to node {$g$} (Q);

\end{tikzpicture}
\end{center}

\noindent then we think of $\chi$ as an abstract operation composing 2-pasting diagrams of shape $\tau$ into single 2-cells (and of $x$ and $x'$ as operations composing 1-pasting diagrams of shape ${\partial}\tau$ into single 1-cells). 
The identity map $ids:\bs{1} \longrightarrow \bs{G}$ picks out an $n$-cell of $\bs{G}$ for each $n \in \mathbb{N}$. We denote this $n$-cell by $\text{id}_n$ and refer to it as the \textit{identity} $n$-cell. The image of $\text{id}_n$ under $g$ is the simple $n$-pasting diagram in $\bs{1}$; see Definition \ref{trivlal pd}.

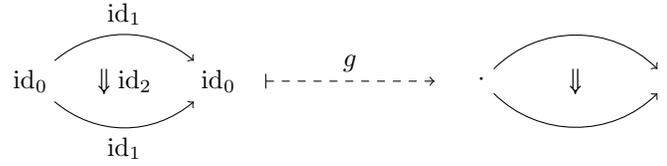
\begin{center}
\begin{tikzpicture}[node distance=2.5cm, auto]

\node (C) {$\cdot$};
\node (D) [node distance=2.5cm, right of=C] {$\cdot$};

\draw[->, bend left=45] (C) to node {} (D);
\draw[->, bend right=45] (C) to node [swap] {} (D);

\node (W) [node distance=1.25cm, right of=C] {$\Downarrow$};

\node (X) [node distance=3.5cm, left of=C] {$\text{id}_0$};
\node (Y) [left of=X] {$\text{id}_0$};
\draw[->, bend left=40] (Y) to node {$\text{id}_1$} (X);
\draw[->, bend right=40] (Y) to node [swap] {$\text{id}_1$} (X);
\node (Z) [node distance=1.25cm, right of=Y] {$\Downarrow \text{id}_2$};

\node (P) [node distance=0.5cm, right of=X] {};
\node (Q) [node distance=0.5cm, left of=C] {};
\draw[|->, dashed] (P) to node {$g$} (Q);

\end{tikzpicture}
\end{center}

The $n$-cells of the globular set $\bs{G} \circ \bs{G}$ are pairs consisting of an $n$-cell $\Lambda$ of $\bs{G}$ together with an $n$-cell of $\bs{G}^*$ (or an $n$-pasting diagram in $\bs{G}$) of shape $g(\Lambda)$. A typical 2-cell $\big((\nu, \nu', v), \chi \big)$ of $\bs{G} \circ \bs{G}$ is

\vspace{-4mm}
\begin{center}
\[ \left(
\begin{tikzpicture}[node distance=2cm, auto, baseline=-1.5mm]

\node (A) {$U$};
\node (B) [node distance=2.5cm, right of=A] {$V$};
\node (C) [node distance=2.5cm, right of=B] {$W$};
\node (p) [node distance=0.44cm, right of=C] {};
\node () [node distance=0.1cm, below of=p] {$,$};
\node (C') [node distance=1cm, right of=C] {$X$};
\node (D) [node distance=2.5cm, right of=C'] {$Y$};

\draw[->] (B) to node {$v$} (C);

\draw[->, bend left=60] (A) to node {$u$} (B);
\draw[->] (A) to node {} (B);
\draw[->, bend right=60] (A) to node [swap] {$u''$} (B);

\node (W) [node distance=1.25cm, right of=A] {};
\node (YA) [node distance=0.4cm, above of=W] {$\Downarrow  \nu $};
\node (YC) [node distance=0.4cm, below of=W] {$\Downarrow  \nu' $};

\draw[->, bend left=40] (C') to node {$x$} (D);
\draw[->, bend right=40] (C') to node [swap] {$x'$} (D);

\node (XB) [node distance=1.3cm, right of=C'] {$\Downarrow  \chi $};

\end{tikzpicture}
\right) \]
\end{center}
where the left hand side is a 2-pasting diagram in $\bs{G}$ and $\chi$ is the 2-cell above  satisfying $g(\chi)=\tau$. We use this example to describe the composition map $comp: \bs{G} \circ \bs{G} \longrightarrow \bs{G}$. Let $g(\nu) = \tau_1$, $g(\nu') = \tau_2$, and $g(v) = \tau_3$,

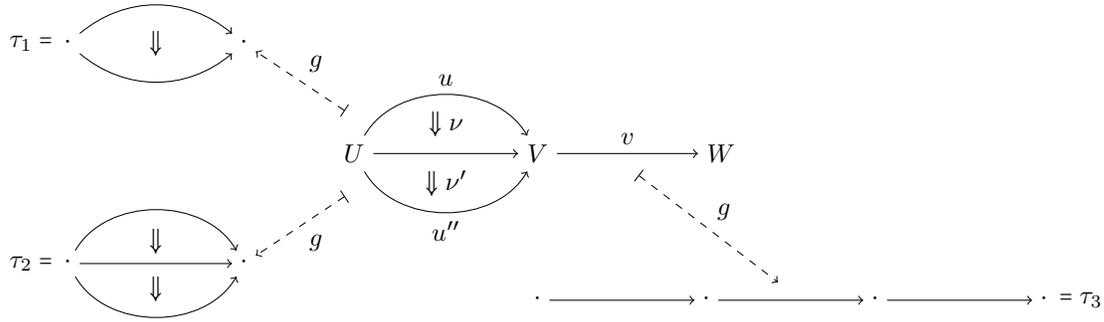
\begin{center}
\resizebox{5.9in}{!}{
\begin{tikzpicture}[node distance=2.5cm, auto]

\node (A) {$U$};
\node (B) [node distance=2.5cm, right of=A] {$V$};
\node (C) [node distance=2.5cm, right of=B] {$W$};
\draw[->] (B) to node {$v$} (C);
\draw[->, bend left=60] (A) to node {$u$} (B);
\draw[->] (A) to node {} (B);
\draw[->, bend right=60] (A) to node [swap] {$u''$} (B);
\node (X) [node distance=1.25cm, right of=A] {};
\node (Y) [node distance=0.4cm, above of=X] {$\Downarrow \nu $};
\node (Z) [node distance=0.4cm, below of=X] {$\Downarrow \nu' $};

\node (D) [node distance=2cm, below of=B] {$\cdot$};
\node (E) [node distance=2.3cm, right of=D] {$\cdot$};
\node (F) [node distance=2.3cm, right of=E] {$\cdot$};
\node (G) [node distance=2.3cm, right of=F] {$\cdot$};
\node (G') [node distance=0.5cm, right of=G] {$= \tau_3$};
\draw[->] (F) to node {} (G);
\draw[->] (E) to node {} (F);
\draw[->] (D) to node {} (E);
\node (a) [node distance=1.25cm, right of=B] {};
\node (a') [node distance=0.2cm, below of=a] {};
\node (b) [node distance=1.2cm, left of=F] {};
\node (b') [node distance=0.15cm, above of=b] {};
\draw[|->, dashed] (a') to node {$g$} (b');

\node (H) [node distance=1.5cm, left of=A, above of=A] {$\cdot$};
\node (I) [node distance=2.4cm, left of=H] {$\cdot$};
\node (I') [node distance=0.5cm, left of=I] {$\tau_1 = $};
\draw[->, bend left=40] (I) to node {} (H);
\draw[->, bend right=40] (I) to node {} (H);
\node (X') [node distance=1.2cm, right of=I] {$\Downarrow$};;
\node (y) [node distance=0.5cm,above of=A] {};
\draw[|->, dashed] (y) to node [swap] {$g$} (H);

\node (K) [node distance=1.5cm, left of=A, below of=A] {$\cdot$};
\node (L) [node distance=2.4cm, left of=K] {$\cdot$};
\node (L') [node distance=0.5cm, left of=L] {$\tau_2 =$};
\draw[->, bend left=60] (L) to node {} (K);
\draw[->] (L) to node {} (K);
\draw[->, bend right=60] (L) to node {} (K);
\node (T) [node distance=1.2cm, right of=L] {};
\node (X'') [node distance=0.3cm, above of=T] {$\Downarrow$};
\node (Y'') [node distance=0.35cm, below of=T] {$\Downarrow$};
\node (z) [node distance=0.5cm, below of=A] {};
\draw[|->, dashed] (z) to node {$g$} (K);

\end{tikzpicture}\label{tau_i}
}
\end{center}
and denote by $ \tau \circ (\tau_1, \tau_2, \tau_3)$ the 2-cell of $\bs{1}^*$ obtained by replacing the individual cells in $\tau$ with $\tau_1, \tau_2$ and $\tau_3$, respectively.
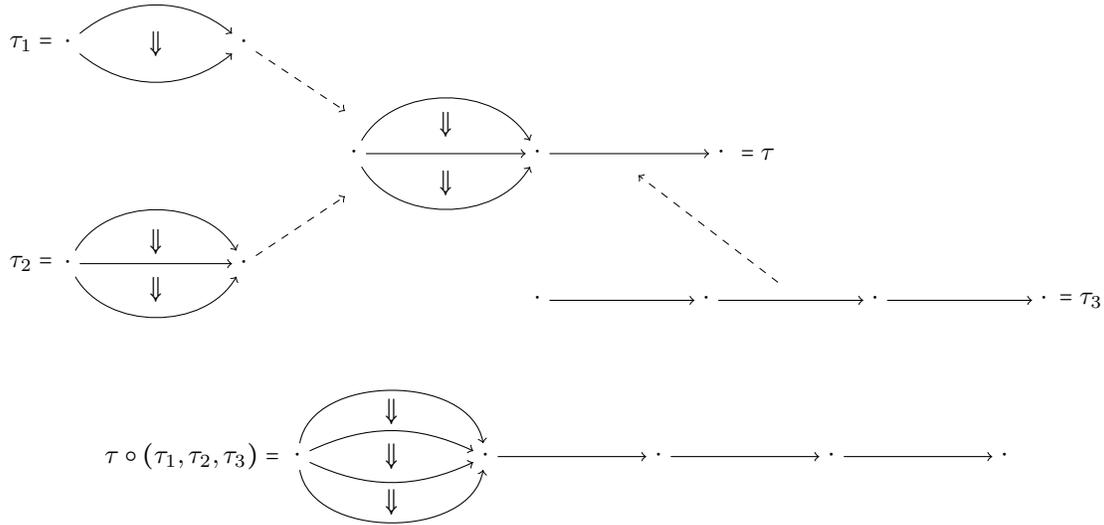
\begin{center}
\resizebox{5.9in}{!}{
\begin{tikzpicture}[node distance=2.5cm, auto]

\node (A) {$\cdot$};
\node (B) [node distance=2.5cm, right of=A] {$\cdot$};
\node (C) [node distance=2.5cm, right of=B] {$\cdot$};
\node (C') [node distance=0.5cm, right of=C] {$= \tau$};
\draw[->] (B) to node {} (C);
\draw[->, bend left=60] (A) to node {} (B);
\draw[->] (A) to node {} (B);
\draw[->, bend right=60] (A) to node [swap] {} (B);
\node (X) [node distance=1.25cm, right of=A] {};
\node (Y) [node distance=0.4cm, above of=X] {$\Downarrow$};
\node (Z) [node distance=0.4cm, below of=X] {$\Downarrow$};

\node (D) [node distance=2cm, below of=B] {$\cdot$};
\node (E) [node distance=2.3cm, right of=D] {$\cdot$};
\node (F) [node distance=2.3cm, right of=E] {$\cdot$};
\node (G) [node distance=2.3cm, right of=F] {$\cdot$};
\node (G') [node distance=0.5cm, right of=G] {$= \tau_3$};
\draw[->] (F) to node {} (G);
\draw[->] (E) to node {} (F);
\draw[->] (D) to node {} (E);
\node (a) [node distance=1.25cm, right of=B] {};
\node (a') [node distance=0.2cm, below of=a] {};
\node (b) [node distance=1.2cm, left of=F] {};
\node (b') [node distance=0.15cm, above of=b] {};
\draw[->, dashed] (b') to node [swap] {} (a');

\node (H) [node distance=1.5cm, left of=A, above of=A] {$\cdot$};
\node (I) [node distance=2.4cm, left of=H] {$\cdot$};
\node (I') [node distance=0.5cm, left of=I] {$\tau_1 = $};
\draw[->, bend left=40] (I) to node {} (H);
\draw[->, bend right=40] (I) to node {} (H);
\node (X') [node distance=1.2cm, right of=I] {$\Downarrow$};
\node (y) [node distance=0.5cm, above of=A] {};
\draw[->, dashed] (H) to node [swap] {} (y);

\node (K) [node distance=1.5cm, left of=A, below of=A] {$\cdot$};
\node (L) [node distance=2.4cm, left of=K] {$\cdot$};
\node (L') [node distance=0.5cm, left of=L] {$\tau_2 =$};
\draw[->, bend left=60] (L) to node {} (K);
\draw[->] (L) to node {} (K);
\draw[->, bend right=60] (L) to node {} (K);
\node (T) [node distance=1.2cm, right of=L] {};
\node (X'') [node distance=0.3cm, above of=T] {$\Downarrow$};
\node (Y'') [node distance=0.35cm, below of=T] {$\Downarrow$};
\node (z) [node distance=0.5cm, below of=A] {};
\draw[->, dashed] (K) to node {} (z);

\end{tikzpicture}
}
\end{center}

\vspace{1mm}

\begin{center}
\begin{tikzpicture}[node distance=2.5cm, auto]

\node (A) {$\cdot$};
\node (B) [node distance=2.5cm, right of=A] {$\cdot$};
\node (C) [node distance=2.3cm, right of=B] {$\cdot$};
\node (D) [node distance=2.3cm, right of=C] {$\cdot$};
\node (E) [node distance=2.3cm, right of=D] {$\cdot$};

\draw[->, bend left=80] (A) to node {} (B);
\draw[->, bend left=25] (A) to node {} (B);
\draw[->, bend right=25] (A) to node [swap] {} (B);
\draw[->, bend right=80] (A) to node [swap] {} (B);
\draw[->] (B) to node {} (C);
\draw[->] (C) to node {} (D);
\draw[->] (D) to node {} (E);

\node (C') [node distance=1.4cm, left of=A] {$\tau  \circ (\tau_1, \tau_2, \tau_3) = $};

\node (X) [node distance=1.25cm, right of=A] {$\Downarrow$};
\node (Y) [node distance=0.6cm, above of=X] {$\Downarrow$};
\node (Z) [node distance=0.6cm, below of=X] {$\Downarrow$};

\end{tikzpicture}
\end{center}
Then $comp ((\nu, \nu', v), \chi) = \chi \circ (\nu, \nu', v)$
is a 2-cell of $\bs{G}$ satisfying $g(\chi \circ (\nu, \nu', v)) = \tau \circ (\tau_1, \tau_2, \tau_3)$. 
\vspace{-2mm}
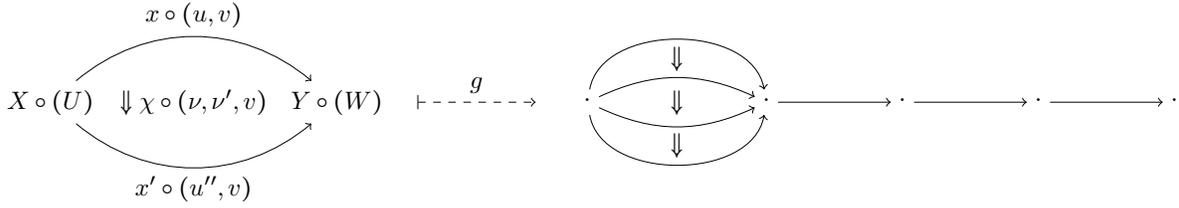
\begin{center}
\resizebox{1\textwidth}{!}{
\begin{tikzpicture}[node distance=1.9cm, auto]

\node (B) {$\cdot$};
\node (C) [node distance=2.5cm, right of=B] {$\cdot$};
\node (D) [right of=C] {$\cdot$};
\node (E) [right of=D] {$\cdot$};
\node (F) [right of=E] {$\cdot$};

\draw[->, bend left=25] (B) to node {} (C);
\draw[->, bend right=25] (B) to node {} (C);
\draw[->, bend left=80] (B) to node {} (C);
\draw[->, bend right=80] (B) to node {} (C);
\draw[->] (C) to node {} (D);
\draw[->] (D) to node {} (E);
\draw[->] (E) to node {} (F);

\node (X) [node distance=1.25cm, right of=B] {$\Downarrow$};
\node (Y) [node distance=0.6cm, above of=X] {$\Downarrow$};
\node (Z) [node distance=0.6cm, below of=X] {$\Downarrow$};

\node (R) [node distance=3.5cm, left of=B] {$Y \circ (W)$};
\node (Q) [node distance=4cm, left of=R] {$X \circ (U)$};
\draw[->, bend left=40] (Q) to node {$x \circ (u,v)$} (R);
\draw[->, bend right=40] (Q) to node [swap] {$x' \circ (u'', v)$} (R);
\node [node distance=2cm, right of=Q] {$\Downarrow \chi \circ (\nu, \nu', v)$};

\node (r) [node distance=1cm, right of=R] {};
\node (a) [node distance=0.6cm, left of=B] {};
\draw[|->, dashed] (r) to node {$g$} (a);

\end{tikzpicture}
}
\end{center}
We think of $\chi \circ (\nu, \nu', v)$ as the abstract operation composing 2-pasting diagrams of shape $\tau \circ (\tau_1, \tau_2, \tau_3)$ given by first using $\nu, \nu'$ and $v$ to compose smaller components and then applying $\chi$ to the result. 

More generally, for any $n$-cell $((\Lambda_1,...,\Lambda_m), \Lambda )$ of $\bs{G} \circ \bs{G}$ there is an $n$-cell 
$$comp((\Lambda_1,...,\Lambda_m), \Lambda) = \Lambda \circ (\Lambda_1,...,\Lambda_m)$$ 
of $\bs{G}$ for which $g( \Lambda \circ (\Lambda_1,...,\Lambda_m))$ is the $n$-pasting diagram in $\bs{1}$ obtained by 
replacing the individual cells in $g(\Lambda)$ with the $g(\Lambda_i)$s.
We think of $\Lambda \circ (\Lambda_1,...,\Lambda_m)$ as the abstract operation composing $n$-pasting diagrams of shape $g( \Lambda \circ (\Lambda_1,...,\Lambda_m) )$ given by first using the $\Lambda_i$s to compose smaller components and then applying $\Lambda$ to the result. 
The identity and associativity axioms for globular operads are expressed by the following equalities.
$$\text{id}_n \circ (\Lambda) = \Lambda = \Lambda \circ (\text{id}_{n_1},...,\text{id}_{n_m})$$
\vspace{-4mm}
$$\Lambda \circ \big( \Lambda_1 \circ (\Lambda_{11},...,\Lambda_{1k_1}),...,\Lambda_m \circ (\Lambda_{m1},...,\Lambda_{mk_m}) \big) = \big( \Lambda \circ (\Lambda_1,...,\Lambda_m) \big) \circ (\Lambda_{11},...,\Lambda_{1k_1},...,\Lambda_{m1},...,\Lambda_{mk_m})$$
\end{unpackdef}

\begin{definition}
A \textit{morphism of globular operads} is a morphism of monoids in $(\bf{GColl}, \circ)$, i.e., a morphism of the underlying globular collections preserving composition and identities.
\end{definition}

\begin{notation} We denote by $\mathbf{GOp}$ the category of globular operads and their morphisms. 
\end{notation}

Truncating everything in this section to $n$-dimensions by replacing the free strict $\omega$-category monad on $\mathbf{GSet}$ with the free strict $n$-category monad on $\mathbf{GSet}_{\bs{n}}$ will give an account of $n$-globular operads.

\begin{notation} We denote by $\mathbf{GColl}_{\bs{n}}$ and $\mathbf{GOp_{\bs{n}}}$ the categories of $n$-globular collections and $n$-globular operads, respectively.
\end{notation}

\section{Algebras for globular operads}\label{algebras section}

In this section we define the category of algebras associated to an ($n$-)globular operad. In section \ref{presenations section} we show that for any algebraic notion of higher category, we can construct an ($n$-)globular operad whose algebras are precisely those higher categories. 

\begin{definition} Each globular operad $\bs{G} = (\bs{G}, g, ids, comp)$ induces a monad $(-)_{\bs{G}}$ on $\bf{GSet}$. The underlying endofunctor sends a globular set $\bs{A}$ to the pullback object

\begin{center}
\begin{tikzpicture}[node distance=1.4cm, auto]

\node (X) {$\bs{A_G}$};
\node (TX) [left of=X, below of=X] {$\bs{A}^*$};
\node (C) [right of=X, below of=X] {$\bs{G}$};

\draw[->] (X) to node [swap] {} (TX);
\draw[->] (X) to node {} (C);

\node (TC') [node distance=2.8cm, below of=X] {$\bs{1}^*$};

\draw[->] (TX) to node [swap] {$!^*$} (TC');
\draw[->] (C) to node {$g$} (TC');

\end{tikzpicture}
\end{center}
and is defined on morphisms using the universal property of pullbacks. The same property yields a canonical morphism $j:(\bs{A_G})_{\bs{G}} \longrightarrow \bs{G} \circ\bs{G}$. The unit and multiplication of the monad at $\bs{A}$ are then the unique morphisms satisfying the commutativity of the following diagrams.

\vspace{1mm}

\begin{center}
\begin{tikzpicture}[node distance=1.4cm, auto]

\node (B'A') {$\bs{A_G}$};
\node (TA') [left of=B'A', below of=B'A'] {$\bs{A}^*$};
\node (B') [right of=B'A', below of=B'A'] {$\bs{G}$};
\node (TY) [right of=TA', below of=TA'] {$\bs{1}^*$};

\node (K) [node distance=0.85cm, above of=B'A'] {};
\node (K') [node distance=0.65cm, right of=K] {$mult_{\bs{A}}$};

\draw[->] (B'A') to node {} (TA');
\draw[->] (B'A') to node [swap] {} (B');
\draw[->] (TA') to node [swap] {$!^*$} (TY);
\draw[->] (B') to node {$g$} (TY);

\node (BA) [node distance=2.8cm, above of=B'A'] {$(\bs{A_G})_{\bs{G}}$};
\node (B) [right of=BA, below of=BA] {$\bs{G} \circ \bs{G}$};
\node (X) [left of=BA, below of=BA] {$\bs{A_G}^*$};
\node (TB) [node distance=1.6cm, above of=TA'] {$\bs{A}^{**}$};

\draw[->] (TB) to node [swap] {$\mu_{\bs{A}}$} (TA');
\draw[->] (X) to node {} (TB);
\draw[->] (BA) to node {} (X);
\draw[->] (BA) to node {$j$} (B);
\draw[->] (B) to node {$comp$} (B');
\draw[->, dashed] (BA) to node [swap] {} (B'A');

\node (BAq) [node distance=6cm, left of=B'A'] {$\bs{A_G}$};
\node (TAq) [left of=BAq, below of=BAq] {$\bs{A}^*$};
\node (Bq) [right of=BAq, below of=BAq] {$\bs{G}$};
\node (TYq) [node distance=2.8cm, below of=BAq] {$\bs{1}^*$};

\draw[->] (BAq) to node [swap] {} (TAq);
\draw[->] (BAq) to node {} (Bq);
\draw[->] (TAq) to node [swap] {$!^*$} (TYq);
\draw[->] (Bq) to node {$g$} (TYq);

\node (Xq) [node distance=2.5cm, above of=BAq] {$\bs{A}$};
\node (xq) [node distance=0.5cm, below of=Xq] {};
\node (Cq) [right of=Xq, below of=Xq] {$\bs{1}$}; 

\draw[->, bend right=35] (Xq) to node [swap] {$\eta_{\bs{A}}$} (TAq);
\draw[->] (Xq) to node {$!$} (Cq);
\draw[->] (Cq) to node {$ids$} (Bq);
\draw[->, dashed] (Xq) to node [swap] {} (BAq);
\draw[->, dashed] (xq) to node [swap] {$unit_{\bs{A}}$} (BAq);

\end{tikzpicture}
\end{center}

\end{definition}

\begin{remark}
The monad axioms for $(-)_{\bs{G}}$ are satisfied by the identity and associativity axioms for $\bs{G}$. 
\end{remark}

The $n$-cells of $\bs{A_G}$ are pairs consisting of an $n$-cell $\Lambda$ of $\bs{G}$ together with an $n$-pasting diagram in $\bs{A}$ of shape $g(\Lambda)$. A typical 2-cell $( (\alpha, \alpha', b), \chi)$ of $\bs{A_G}$ is

\vspace{-6mm}

\begin{center}
\[ \left( 
\begin{tikzpicture}[node distance=2cm, auto, baseline=-1.5mm]

\node (A) {$A$};
\node (B) [node distance=2.5cm, right of=A] {$B$};
\node (C) [node distance=2.4cm, right of=B] {$C$};
\node (P) [node distance=0.45cm, right of=C] {};
\node () [node distance=0.1cm, below of=P] {\large ,};
\node (C') [node distance=1cm, right of=C] {$X$};
\node (D) [node distance=2.5cm, right of=C'] {$Y$};

\draw[->, bend left=60] (A) to node {$a$} (B);
\draw[->] (A) to node {} (B);
\draw[->, bend right=60] (A) to node [swap] {$a''$} (B);

\draw[->] (B) to node {$b$} (C);

\node (W) [node distance=1.25cm, right of=A] {};
\node (YA) [node distance=0.4cm, above of=W] {$\Downarrow  \alpha$};
\node (YC) [node distance=0.4cm, below of=W] {$\Downarrow  \alpha'$};

\draw[->, bend left=40] (C') to node {$x$} (D);
\draw[->, bend right=40] (C') to node [swap] {$x'$} (D);

\node (XB) [node distance=1.3cm, right of=C'] {$\Downarrow  \chi$};

\end{tikzpicture}
\right) \]
\end{center}

\vspace{1mm}

\noindent where the left hand side is a 2-pasting diagram in $\bs{A}$ and $\chi$ is a 2-cell of $\bs{G}$ satisfying $g(\chi) = \tau$; see page \pageref{tau}. We think of each $n$-cell of $\bs{A_G}$ as an $n$-pasting diagram in $\bs{A}$ together with an operation composing it into a single $n$-cell. For each $n$-cell $\alpha$ of $\bs{G}$ we have 
$$unit_{\bs{A}}(\alpha) = ( (\alpha), \, \text{id}_n),$$
and the image of an $n$-cell $\big( ( (\alpha_{11},...,\alpha_{1k_1}), \hspace{0.5mm} \Lambda_1 ) \hspace{0.3mm} ,..., \hspace{0.3mm} ( (\alpha_{m1},...,\alpha_{mk_m}), \hspace{0.5mm} \Lambda_m ), \hspace{0.5mm} \Lambda \big)$ of  $(\bs{A_G})_{\bs{G}}$ under $mult_{\bs{A}}$ is the $n$-cell $$\big( (\alpha_{11},...,\alpha_{1k_1},...,\alpha_{m1},...,\alpha_{mk_m}), \hspace{0.5mm} \Lambda \circ (\Lambda_1,...,\Lambda_m)\big)$$

\vspace{1mm}

\noindent of $\bs{A}$. For instance, $mult_{\bs{A}}$ sends a 2-cell 

\vspace{-4mm}
\begin{center}
\[ \left(
\resizebox{0.95\textwidth}{!}{
\begin{tikzpicture}[node distance=1.8cm, auto, baseline=-12mm]
\node (B) {$A$};
\node () [node distance=0.3cm, left of=B] {\Large $\Bigg($};
\node (C) [right of=B] {$B$};
\node (t) [node distance=0.4cm, right of=C] {};
\node () [node distance=0.1cm, below of=t] {$,$};
\draw[->, bend left=50] (B) to node {$a$} (C);
\draw[->, bend right=50] (B) to node [swap] {} (C);
\node (y) [node distance=0.9cm, right of=B] {$\Downarrow \alpha$};
\node (Z) [node distance=0.8cm, right of=C] {$U$};
\node (U) [right of=Z] {$V$};
\draw[->, bend left=50] (Z) to node {$u$} (U);
\draw[->, bend right=50] (Z) to node [swap] {} (U);
\node (X') [node distance=0.9cm, right of=Z] {$\Downarrow \nu$};
\node () [node distance=0.3cm, right of=U] {\Large $\Bigg)$};
\node (Q) [node distance=0.7cm, below of=A] {};

\node (B') [node distance=2cm, below of=B] {\small $A$};
\node () [node distance=0.3cm, left of=B'] {\Large $\Bigg($};
\node (C') [right of=B'] {$B$};
\node (r) [node distance=0.4cm, right of=C'] {};
\node () [node distance=0.1cm, below of=r] {$,$};
\draw[->, bend left=70] (B') to node {} (C');
\draw[->] (B') to node {} (C');
\draw[->, bend right=70] (B') to node [swap] {$a'''$} (C');
\node (y') [node distance=0.9cm, right of=B'] {};
\node (y'') [node distance=0.3cm, above of=y'] {$\Downarrow \alpha'$};
\node (y''') [node distance=0.3cm, below of=y'] {$\Downarrow \alpha'$};
\node (Z') [node distance=0.8cm, right of=C'] {$U$};
\node (U') [right of=Z'] {$V$};
\draw[->, bend left=50] (Z') to node {} (U');
\draw[->, bend right=50] (Z') to node [swap] {$u''$} (U');
\node (X') [node distance=0.9cm, right of=Z'] {$\Downarrow \nu'$};
\node () [node distance=0.3cm, right of=U'] {\Large $\Bigg)$};

\node (R) [node distance=1cm, below of=U] {};
\node (C'') [node distance=1cm, right of=R] {$B$};
\node (D) [node distance=1.5cm, right of=C''] {$C$};
\node (E) [node distance=1.5cm, right of=D] {$D$};
\node (F) [node distance=1.5cm, right of=E] {$E$};
\node (q) [node distance=0.4cm, right of=F] {};
\node () [node distance=0.1cm, below of=q] {$,$};
\node (U'') [node distance=0.8cm, right of=F] {$V$};
\node (V) [node distance=1.5cm, right of=U''] {$W$};
\draw[->] (C'') to node {$b$} (D);
\draw[->] (D) to node {$c$} (E);
\draw[->] (E) to node {$d$} (F);
\draw[->] (U'') to node {$v$} (V);
\node () [node distance=0.3cm, left of=C''] {$\Big($};
\node () [node distance=0.3cm, right of=V] {$\Big)$};

\node (X) [node distance=1.1cm,right of=V] {$X$};
\node (p) [node distance=0.65cm,right of=V] {};
\node () [node distance=0.1cm, below of=p] {,};
\node (Y) [right of=X] {$Y$};
\draw[->, bend left=50] (X) to node {$x$} (Y);
\draw[->, bend right=50] (X) to node [swap] {$x'$} (Y);
\node (z) [node distance=0.9cm, right of=X] {$\Downarrow \chi$};
\end{tikzpicture} 
}
\right) \]
\end{center}

\vspace{2mm}

\noindent of $(\bs{A_G})_{\bs{G}}$ to the 2-cell of $\bs{A_G}$ below.

\vspace{-3mm}
\begin{center}
\[ \left(
\begin{tikzpicture}[node distance=2cm, auto, baseline=-1.5mm]
\node (B) {$A$};
\node (C) [node distance=2.5cm, right of=B] {$B$};
\node (D) [right of=C] {$C$};
\node (E) [right of=D] {$D$};
\node (F) [right of=E] {$E$};
\node (P) [node distance=0.45cm, right of=F] {};
\node () [node distance=0.1cm, below of=P] {$,$};

\draw[->, bend left=80] (B) to node {$a$} (C);
\draw[->, bend right=80] (B) to node [swap] {$a'''$} (C);
\draw[->, bend left=25] (B) to node {} (C);
\draw[->, bend right=25] (B) to node {} (C);
\draw[->] (C) to node {$b$} (D);
\draw[->] (D) to node {$c$} (E);
\draw[->] (E) to node {$d$} (F);

\node (Y) [node distance=1.25cm, right of=B] {$\Downarrow  \alpha'$};
\node (Y') [node distance=0.6cm, above of=Y] {$\Downarrow  \alpha$};
\node (Y'') [node distance=0.6cm, below of=Y] {$\Downarrow  \alpha''$};

\node (A') [node distance=1.4cm, right of=F] {$X \circ (U)$};
\node (B') [node distance=4cm, right of=A'] {$Y \circ (W)$};

\draw[->, bend left=40] (A') to node {$x \circ (u,v)$} (B');
\draw[->, bend right=40] (A') to node [swap] {$x' \circ (u'', v)$} (B');

\node (R) [node distance=2cm, right of=A'] {$\Downarrow \chi \circ (\nu, \nu', v)$};

\end{tikzpicture}
\right) \]
\end{center}

\vspace{2mm}

\noindent Here the images of $\nu$, $\nu'$ and $v$ under the underlying collection map $g$ are as on page \pageref{tau_i}. 

\begin{definition}
Let $\bs{G}$ be a globular operad. The \textit{category $\bs{G} \mhyphen \bf{Alg}$ of algebras for $\bs{G}$} (or $\bs{G}$-\textit{algebras}) is the category of algebras for the monad $(-)_{\bs{G}}$.
\end{definition}

\begin{unpackdef}
An algebra for $\bs{G}$ on a globular set $\bs{A}$ is a morphism $\theta:\bs{A_G} \longrightarrow \bs{A}$ of globular sets satisfying unit and multiplication axioms. Given an $n$-cell $( (\alpha_1,...,\alpha_m),  \, \Lambda )$ of $\bs{A_G}$ we write
$$\theta_n ( (\alpha_1,...,\alpha_m), \, \Lambda ) = \Lambda(\alpha_1,...,\alpha_m)$$
and think of $\Lambda(\alpha_1,...,\alpha_m)$ as a composition of the $n$-pasting diagram $(\alpha_1,...,\alpha_m)$ in $\bs{A}$ into single $n$-cell of $\bs{A}$. For example, $\theta$ sends the 2-cell $((\alpha, \alpha', b), \chi)$ of $\bs{A_G}$ above to a 2-cell

\begin{center}
\begin{tikzpicture}[node distance=2cm, auto]
\node (R) {$X(A)$};
\node (Q) [node distance=3.7cm, right of=R] {$Y(C)$};
\draw[->, bend left=37] (R) to node {$x(a,b)$} (Q);
\draw[->, bend right=37] (R) to node [swap] {$x'(a'',b)$} (Q);
\node (U) [node distance=1.85cm, right of=R] {$\Downarrow \chi(\alpha, \alpha',b)$};
\end{tikzpicture}
\end{center}

\noindent of $\bs{A}$, thought of as a composite of $(\alpha, \alpha', b)$. 
The unit and multiplication axioms are expressed by the following equalities.
$$\text{id}_n(\alpha) = \alpha$$
\vspace{-1mm}
$$\Lambda \hspace{0.5mm} \big( \Lambda_1 (\alpha_{11},...,\alpha_{1k_1}),..., \hspace{0.5mm} \Lambda_m(\alpha_{m1},...,\alpha_{mk_m}) \big) = \big( \Lambda \circ (\Lambda_1,...,\Lambda_m) \big) (\alpha_{11},...,\alpha_{1k_1},...,\alpha_{m1},...,\alpha_{mk_m})$$
\end{unpackdef}

\vspace{0.5mm}

\begin{unpackdef}
A morphism $F:(\bs{A},\theta) \longrightarrow (\bs{B},\sigma)$ of algebras for $\bs{G}$ is a morphism $F:\bs{A} \longrightarrow \bs{B}$ of the underlying globular sets \textit{strictly} preserving the composition of pasting diagrams defined by the algebra structures.
\end{unpackdef}

Observe the similarities between the free strict higher category monad  $(-)^*$ on $\bf{GSet}$ and the monad $(-)_{\bs{G}}$ induced by a globular operad $\bs{G}$. Given a globular set $\bs{A}$, both $\bs{A}^*$ and $\bs{A}_{\bs{G}}$ are globular sets whose $n$-cells are $n$-pasting diagrams in $\bs{A}$, together with some extra data in the case of $\bs{A}_{\bs{G}}$, and an algebra for either monad on $\bs{A}$ is a way of composing these pasting diagrams in a coherent way. 
However, not every globular operad gives rise to something that we could call a theory of $\omega$-category. Take for example the initial globular operad given by equipping the globular collection $\eta_{\bs{1}}:\bs{1} \longrightarrow \bs{1}^*$ with its unique operad structure. The monad induced on $\mathbf{GSet}$ by this operad is isomorphic to the identity monad, so the category of algebras is isomorphic to $\mathbf{GSet}$ rather than to some category of $\omega$-categories and the strict functors between them. In Section \ref{glob ops for higher cats section}, we determine precisely when a globular operad \textit{does} give rise to some sensible theory of $\omega$-category. 

\begin{example}[Strict $\omega$-categories]\label{GlobOp T for strict w cats}
The terminal globular operad $\bs{T}$ is given by equipping the collection $1:\bs{1}^* \longrightarrow \bs{1}^*$ with its unique operad structure. The monad $(-)_{\bs{T}}$ induced on $\mathbf{GSet}$ is isomorphic to the free strict $\omega$-category monad $(-)^*$, so the category $\bs{T} \mhyphen \bf{Alg}$ of algebras for $\bs{T}$ is equivalent to the category $\bf{Str} \ \omega \mhyphen \bf{Cat}$ of strict $\omega$-categories. 
We refer to $\bs{T}$ as \textit{the globular operad for strict $\omega$-categories}.
\end{example}

An algebra for an $n$-globular operad $\bs{G_n}$ is defined analogousy to an algebra for a globular operad $\bs{G}$.

\begin{example}[Strict $n$-categories]\label{GlobOp T for strict n cats}
The monad induced on $\mathbf{GSet}_{\bs{n}}$ by the terminal $n$-globular operad $\bs{T_n}$ is isomorphic to the free strict $n$-category monad $(-)^*$,
so we refer to $\bs{T_n}$ as \textit{the globular operad for strict $n$-categories}.
\end{example}

\section{Contractibility}\label{contractibility section}

In this section, we demonstrate how a globular operad being contractible is precisely the property needed to guarantee that its algebras satisfy the required conditions on composition and coherence for $\omega$-categories. We then adjust the coherence condition to suit $n$-categories, and show that a similar statement is true for $n$-globular operads. In Section \ref{glob ops for higher cats section}, we use the results obtained here to define subcategories of $\mathbf{GOp}$ and $\mathbf{GOp}_{\bs{n}}$, respectively, consisting of just those ($n$-)globular operads whose algebras are some variety of higher category. We begin with a brief explanation of the composition and coherence conditions.

\begin{compositionc} 
For any $n$-pasting diagram in an $\omega$-category together with a composition of the $(n-1)$-dimensional boundary, there should be an operation composing it into a single $n$-cell that is consistent with this composition of the boundary.
\end{compositionc}

\begin{example} Let 

\begin{center}
\begin{tikzpicture}[node distance=2.5cm, auto]

\node (A) {$A$};
\node (B) [right of=A] {$B$};
\node (C) [right of=B] {$C$};
\node (D) [right of=C] {$D$};

\draw[->, bend left=40] (A) to node {$a$} (B);
\draw[->, bend left=40] (B) to node {$b$} (C);
\draw[->, bend left=40] (C) to node {$c$} (D);
\draw[->, bend right=40] (A) to node [swap] {$a'$} (B);
\draw[->, bend right=40] (B) to node [swap] {$b'$} (C);
\draw[->, bend right=40] (C) to node [swap] {$c'$} (D);

\node (W) [node distance=1.25cm, right of=A] {$\Downarrow  \alpha$};
\node (U) [node distance=1.25cm, right of=B] {$\Downarrow  \beta$};
\node (V) [node distance=1.25cm, right of=C] {$\Downarrow  \gamma$};

\end{tikzpicture}
\end{center}

\noindent be a 2-pasting diagram in some $\omega$-category. Given any way of composing the 1-cell boundary, say $a(bc)$ and $(a'b')c'$, there should be a way to compose the pasting diagram into a single 2-cell of the form

\vspace{0.5mm}

\begin{center}
\begin{tikzpicture}[node distance=2.5cm, auto]

\node (A) {$A$};
\node (B) [right of=A] {$D$.};

\draw[->, bend left=40] (A) to node {$a(bc)$} (B);
\draw[->, bend right=40] (A) to node [swap] {$(a'b')c'$} (B);

\node () [node distance=1.25cm, right of=A] {$\Downarrow$};

\end{tikzpicture}
\end{center}

\end{example}

\begin{coherencec}
For any two ways of composing the same $n$-pasting diagram in an $\omega$-category that agree on the composition of its $(n-1)$-dimensional boundary, there should be a coherence $(n+1)$-cell between the two composites. 
\end{coherencec}

\begin{example} The compositions $a(bc)$ and $(ab)c$ of a 1-pasting diagram

\vspace{0.5mm}

\begin{center}
\begin{tikzpicture}[node distance=2cm, auto]

\node (A) {$A$};
\node (B) [right of=A] {$B$};
\node (C) [right of=B] {$C$};
\node (D) [right of=C] {$D$};

\draw[->] (A) to node {$a$} (B);
\draw[->] (B) to node {$b$} (C);
\draw[->] (C) to node {$c$} (D);

\end{tikzpicture}
\end{center}

\vspace{0.5mm}

\noindent in an $\omega$-category both use the same way of composing the boundary 0-cells $A$ and $D$ - the `do nothing' composition, so there should be a coherence 2-cell between them,

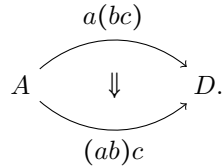
\vspace{0.5mm}

\begin{center}
\begin{tikzpicture}[node distance=2.5cm, auto]

\node (A) {$A$};
\node (B) [right of=A] {$D$.};

\draw[->, bend left=40] (A) to node {$a(bc)$} (B);
\draw[->, bend right=40] (A) to node [swap] {$(ab)c$} (B);

\node (W) [node distance=1.25cm, right of=A] {$\Downarrow$};

\end{tikzpicture}
\end{center}

\end{example}

We now give Leinster's definition of contractibility for globular operads \cite[Chapter 9.1]{TL}, and show how contractibility combines composition and coherence into a single property. It should be noted that this definition differs from Batanin's original definition, which can be found in \cite[Section 8]{MB}.

\begin{definition}
Let $\bs{G}$ be a globular set. For $n \geqslant 1$ we say that a pair $(x,x')$ of $n$-cells of $\bs{G}$ are \textit{parallel} if they share the same source and target, so $s(x)=s(x')$ and $t(x)=t(x')$. All pairs of 0-cells are considered parallel.
\end{definition}

\begin{definition}  A \textit{contraction function} on a globular collection $g:\bs{G} \longrightarrow \bs{1}^*$ is a function assigning to each triple $(x, x', \tau)$, where $(x,x')$ is a parallel pair of $n$-cells of $\bs{G}$ satisfying $g(x) = g(x') = {\partial}\tau$ (see Notation \ref{boundary notation}), an $(n{+}1)$-cell $\chi:x \longrightarrow x'$ of $\bs{G}$ satisfying $g(\chi)=\tau$.
\end{definition}

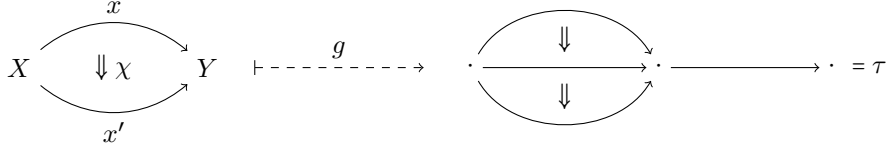
\begin{center}
\begin{tikzpicture}[node distance=2.5cm, auto]

\node (B) {$\cdot$};
\node (C) [right of=B] {$\cdot$};
\node (D) [node distance=2.3cm, right of=C] {$\cdot$};
\node (D') [node distance=0.5cm, right of=D] {$= \tau$}; 

\draw[->] (C) to node {} (D);

\draw[->, bend left=60] (B) to node {} (C);
\draw[->] (B) to node {} (C);
\draw[->, bend right=60] (B) to node [swap] {} (C);

\node (W) [node distance=1.25cm, right of=B] {};
\node (YA) [node distance=0.38cm, above of=W] {$\Downarrow $};
\node (YC) [node distance=0.38cm, below of=W] {$\Downarrow $};

\node (X) [node distance=3.5cm, left of=B] {$Y$};
\node (Y) [left of=X] {$X$};
\draw[->, bend left=40] (Y) to node {$x$} (X);
\draw[->, bend right=40] (Y) to node [swap] {$x'$} (X);
\node (Z) [node distance=1.25cm, right of=Y] {$\Downarrow \chi$};

\node (P) [node distance=0.5cm, right of=X] {};
\node (Q) [node distance=0.5cm, left of=B] {};
\draw[|->, dashed] (P) to node {$g$} (Q);

\end{tikzpicture}
\end{center}

Given a contraction function on a globular collection $g:\bs{G} \longrightarrow \bs{1}^*$, we refer to the cells of $\bs{G}$ in the image of the contraction function as \textit{contraction cells}.

\begin{definition}
A globular operad $\bs{G}$ is \textit{contractible} if there exists a contraction function on its underlying globular collection.
\end{definition}

\begin{example} There is a unique contraction on the globular operad $\bs{T}$ for strict $\omega$-categories; see Example \ref{GlobOp T for strict w cats}. When equipped with this contraction, every $n$-cell of $\bs{T}$ for $n>0$ is a contraction cell.
\end{example}

\begin{notation}\label{the category CG-Op} We denote by $\mathbf{C \mhyphen GOp}$ the category whose objects are globular operads carrying a specified contraction and whose arrows are contraction preserving morphisms of globular operads.
\end{notation}

We demonstrate how the contractibilty of a globular operad ensures that its algebras satisfy the required composition and coherence conditions for $\omega$-categories with some low dimensional examples. In these examples we assume that $\bs{G}$ is a contractible globular operad with underlying globular collection $g:\bs{G} \longrightarrow \bs{1}^*$, and that $\theta:\bs{A}_{\bs{G}} \longrightarrow \bs{A}$
is an algebra for $\bs{G}$ on a globular set $\bs{A}$. Note that we only assume that $\bs{G}$ is contractible, we do not choose a \textit{specific} contraction function.

\begin{example}[\textit{Composition}] Let 

\vspace{-1mm}

\begin{center}
\begin{tikzpicture}[node distance=2cm, auto]

\node (A) {$A$};
\node (B) [right of =A] {$B$};
\node (C) [right of =B] {$C$};

\draw[->] (A) to node {$a$} (B);
\draw[->] (B) to node {$b$} (C);

\end{tikzpicture}
\end{center}

\noindent be a 1-pasting diagram in $\bs{A}$, and denote by $\sigma$ the following 1-pasting diagram in $\bs{1}$.

\vspace{1mm}

\begin{center}
\begin{tikzpicture}[node distance=2cm, auto]

\node (A) {$\cdot$};
\node (B) [right of =A] {$\cdot$};
\node (C) [right of =B] {$\cdot$};
\node () [node distance=0.5cm, left of =A] {$\sigma = $};

\draw[->] (A) to node {} (B);
\draw[->] (B) to node {} (C);

\end{tikzpicture}
\end{center}
Since $\bs{G}$ is contractible and $(\text{id}_0,\text{id}_0,\sigma)$ is a triple satisfying $g(\text{id}_0)=g(\text{id}_0)=\partial \sigma$, there must exist a 1-cell $x:\text{id}_0 \longrightarrow \text{id}_0$ of $\bs{G}$ satisfying $g(x) = \sigma$. This means that

\vspace{-1em}
\begin{center}
\[ \left(
\begin{tikzpicture}[node distance=2cm, auto, baseline=-1.2mm]

\node (A) {$A$};
\node (B) [right of =A] {$B$};
\node (C) [right of =B] {$C$};
\node (X) [node distance=0.8cm, right of=C] {$\text{id}_0$};
\node (Y) [right of=X] {$\text{id}_0$};
\node (X') [node distance=0.35cm, right of=C] {};
\node () [node distance=0.1cm, below of=X'] {$,$};

\draw[->] (A) to node {$a$} (B);
\draw[->] (B) to node {$b$} (C);
\draw[->] (X) to node {$x$} (Y);

\end{tikzpicture}
\right) \]
\end{center}
is a 1-cell of $\bs{A}_{\bs{G}}$, denoted $( (a,b), x )$. We think of the image $\theta_1((a,b), x) = x(a,b)$ of this 1-cell under the algebra structure map $\theta:\bs{A}_{\bs{G}} \longrightarrow \bs{A}$ as a composite of $a$ and $b$ in $\bs{A}$. 

\begin{center}
\begin{tikzpicture}[node distance=2cm, auto]

\node (P) {$A$};
\node (Q) [node distance=2.5cm, right of =P] {$C$};

\draw[->] (P) to node {$x(a,b)$} (Q);

\end{tikzpicture}
\end{center}

\end{example}

\begin{example}[\textit{Coherence}]

Let $x$ be the 1-cell of $\bs{G}$ described in the previous example and consider a 1-pasting diagram 

\vspace{-1mm}

\begin{center}
\begin{tikzpicture}[node distance=2cm, auto]

\node (A) {$A$};
\node (B) [right of =A] {$B$};
\node (C) [right of =B] {$C$};
\node (D) [right of =C] {$D$};

\draw[->] (A) to node {$a$} (B);
\draw[->] (B) to node {$b$} (C);
\draw[->] (C) to node {$c$} (D);

\end{tikzpicture}
\end{center}
in $\bs{A}$. The operadic composites $x \circ (\text{id}_1,x)$ and $x \circ (x, \text{id}_1)$ are parallel 1-cells of $\bs{G}$ whose images under $g$ are  shown below.

\begin{center}
\begin{tikzpicture}[node distance=2cm, auto]

\node (A) {$\cdot$};
\node (A') [node distance=2.6cm, left of =A] {$g(x \circ (\text{id}_1,x)) = g(x \circ (x, \text{id}_1)) =$};
\node (B) [right of =A] {$\cdot$};
\node (C) [right of =B] {$\cdot$};
\node (D) [right of =C] {$\cdot$};

\draw[->] (A) to node [swap] {} (B);
\draw[->] (B) to node [swap] {} (C);
\draw[->] (C) to node [swap] {} (D);

\end{tikzpicture}
\end{center}

\noindent By the contractibility of $\bs{G}$ there must exists a 2-cell $\varsigma:x \circ (\text{id}_1,x) \longrightarrow x \circ (x, \text{id}_1)$ of $\bs{G}$ satisfying

\vspace{0.5mm}

\begin{center}
\begin{tikzpicture}[node distance=2cm, auto]

\node (A) {$\cdot$};
\node (A') [node distance=0.7cm, left of =A] {$g(\varsigma)=$};
\node (B) [right of =A] {$\cdot$};
\node (C) [right of =B] {$\cdot$};
\node (D) [right of =C] {$\cdot$};

\draw[->] (A) to node [swap] {} (B);
\draw[->] (B) to node [swap] {} (C);
\draw[->] (C) to node [swap] {} (D);

\end{tikzpicture}
\end{center}
where $g(\varsigma)$ is a degenerate 2-pasting diagram in $\bs{1}$; see Definition \ref{degenerate pd}. It follows that 

\vspace{-0.8em}
\begin{center}
\[ \left(
\begin{tikzpicture}[node distance=2cm, auto, baseline=-1.4mm]

\node (A) {$A$};
\node (B) [right of =A] {$B$};
\node (C) [right of =B] {$C$};
\node (D) [right of =C] {$D$};
\node (X) [node distance=0.9cm, right of=D] {$\text{id}_0$};
\node (P'') [node distance=1.5cm, right of =X] {$\Downarrow \varsigma$};
\node (Y) [node distance=3cm, right of=X] {$\text{id}_0$};
\node (X') [node distance=0.4cm, right of=D] {};
\node () [node distance=0.1cm, below of=X'] {$,$};

\draw[->] (A) to node {$a$} (B);
\draw[->] (B) to node {$b$} (C);
\draw[->] (C) to node {$c$} (D);
\draw[->, bend left=35] (X) to node {$x \circ (\text{id}_1,x)$} (Y);
\draw[->, bend right=35] (X) to node [swap] {$x \circ (x, \text{id}_1)$} (Y);

\end{tikzpicture}
\right) \]
\end{center}
is a 2-cell of $\bs{A}_{\bs{G}}$, denoted $( (a,b,c), \varsigma)$. Note that we are now viewing $(a,b,c)$ as a degenerate 2-pasting diagram in $\bs{A}$ rather than a 1-pasting diagram. Using the notation $x(a,b) = ab$, the unit and multiplication axioms for $\bs{G}$-algebras yield the following equalities.

\vspace{-1mm}

$$\big( x \circ (\text{id}_1,x)\big) \ (a,b,c) = x\big(\text{id}_1(a), \ x(b,c) \big) = a(bc)$$

\vspace{-1mm}

$$\big( x \circ (x, \text{id}_1)\big) \ (a,b,c) = x\big(x(a,b), \ \text{id}_1(c) \big) = (ab)c$$

\vspace{1.5mm}

\noindent The image $\theta_2((a,b,c), \varsigma) = \varsigma(a,b,c)$ of the 2-cell of $\bs{A}_{\bs{G}}$ above under the algebra structure map $\theta$ can therefore be thought of as a coherence 2-cell $a(bc) \longrightarrow (ab)c$ in $\bs{A}$.

\begin{center}
\begin{tikzpicture}[node distance=2cm, auto]

\node (P) {$A$};
\node (P') [node distance=1.5cm, right of =P] {$\Downarrow \varsigma (a,b,c)$};
\node (Q) [node distance=3cm, right of =P] {$D$};

\draw[->, bend left=40] (P) to node {$a(bc)$} (Q);
\draw[->, bend right=40] (P) to node [swap] {$(ab)c$} (Q);

\end{tikzpicture}
\end{center}

\end{example}

\begin{example}[\textit{Composition}]

Continuing from the previous example, let

\begin{center}
\begin{tikzpicture}[node distance=2.5cm, auto]

\node (A) {$A$};
\node (B) [right of=A] {$B$};
\node (C) [right of=B] {$C$};
\node (D) [right of=C] {$D$};

\draw[->, bend left=40] (A) to node {$a$} (B);
\draw[->, bend left=40] (B) to node {$b$} (C);
\draw[->, bend left=40] (C) to node {$c$} (D);
\draw[->, bend right=40] (A) to node [swap] {$a'$} (B);
\draw[->, bend right=40] (B) to node [swap] {$b'$} (C);
\draw[->, bend right=40] (C) to node [swap] {$c'$} (D);

\node (W) [node distance=1.25cm, right of=A] {$\Downarrow  \alpha$};
\node (U) [node distance=1.25cm, right of=B] {$\Downarrow  \beta$};
\node (V) [node distance=1.25cm, right of=C] {$\Downarrow  \gamma$};

\end{tikzpicture}
\end{center}

\noindent be a 2-pasting diagram in $\bs{A}$. Since $\bs{G}$ is contractible there must exist a 2-cell $\psi: x \circ (\text{id}_1,x) \longrightarrow x \circ (x, \text{id}_1)$ of $\bs{G}$ such that 

\vspace{-3mm}
\begin{center}
\[ \left(
\begin{tikzpicture}[node distance=2.5cm, auto, baseline=-1.5mm]

\node (A) {$A$};
\node (B) [right of =A] {$B$};
\node (C) [right of =B] {$C$};
\node (D) [right of =C] {$D$};
\node (X) [node distance=0.9cm, right of=D] {$\text{id}_0$};
\node (P'') [node distance=1.5cm, right of =X] {$\Downarrow \psi$};
\node (Y) [node distance=3cm, right of=X] {$\text{id}_0$};
\node (X') [node distance=0.4cm, right of=D] {};
\node () [node distance=0.1cm, below of=X'] {$,$};

\draw[->, bend left=40] (A) to node {$a$} (B);
\draw[->, bend left=40] (B) to node {$b$} (C);
\draw[->, bend left=40] (C) to node {$c$} (D);
\draw[->, bend right=40] (A) to node [swap] {$a'$} (B);
\draw[->, bend right=40] (B) to node [swap] {$b'$} (C);
\draw[->, bend right=40] (C) to node [swap] {$c'$} (D);

\node (W) [node distance=1.25cm, right of=A] {$\Downarrow  \alpha$};
\node (U) [node distance=1.25cm, right of=B] {$\Downarrow  \beta$};
\node (V) [node distance=1.25cm, right of=C] {$\Downarrow  \gamma$};
\draw[->, bend left=35] (X) to node {$x \circ (\text{id}_1,x)$} (Y);
\draw[->, bend right=35] (X) to node [swap] {$x \circ (x, \text{id}_1)$} (Y);

\end{tikzpicture}
\right) \]
\end{center}

\vspace{2mm}

\noindent is a 2-cell of $\bs{A}_{\bs{G}}$, denoted $( (\alpha, \beta, \gamma), \psi)$. The image of this 2-cell under the algebra structure map $\theta$ is a 2-cell $\psi(\alpha, \beta, \gamma):a(bc) \longrightarrow (a'b')c'$ of $\bs{A}$, thought of as a composite of $\alpha$, $\beta$ and $\gamma$.

\end{example}

The coherence condition for $n$-categories is stronger than the coherence condition for $\omega$-categories; in $n$-categories the $n$-cells require special attention.
This is clear when we consider that algebraic definitions of higher category can be broken down into \textit{data} and \textit{axioms}. 
For $n$-categories, the data consists of $k$-cells ($0 \leqslant k \leqslant n$), including specified coherence cells, and composition operations, while the axioms impose various constraints on composition. 
It is thought that for any $n$-pasting diagram in an $n$-category together with a way to compose its $(n-1)$-dimensional boundary, the axioms should ensure that there is a \textit{unique} choice of composite $n$-cell that is consistent with the composition of the boundary.
Moreover, for fully weak $n$-categories the axioms should actually be equivalent to this condition on $n$-cell composition; this is known as the \textit{coherence theorem}.

\begin{coherencethm}\label{the coherence theorem}
The axioms for a fully weak algebraic $n$-category are equivalent to the condition that for any $n$-pasting diagram together with a composition of the $(n-1)$-dimensional boundary, there is a \textit{unique} choice of composite $n$-cell with respect to this composition of the boundary.
\end{coherencethm}

\begin{examples}
The associativity and identity axioms for an ordinary (1-)category are equivalent to the condition that for any sequence $A \longrightarrow \hdots \longrightarrow A'$ of composable 1-cells there is exactly one way to compose it into a single 1-cell $A \longrightarrow A'$. Similarly, the axioms for a bicategory are equivalent to stating that given any 2-pasting diagram together with a way of composing its 1-cell boundary, there is a unique way to compose the pasting diagram into a single 2-cell that is consistent with this composition of the boundary \cite[Section 2.4]{TL4}. Take for example a 2-pasting diagram

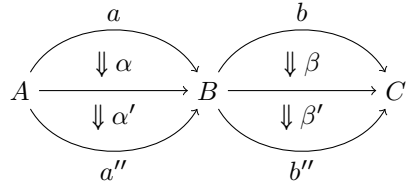
\begin{center}
\begin{tikzpicture}[node distance=2.5cm, auto]

\node (A) {$A$};
\node (B) [right of=A] {$B$};
\node (C) [right of=B] {$C$};

\draw[->, bend left=60] (A) to node {$a$} (B);
\draw[->, bend left=60] (B) to node {$b$} (C);

\draw[->] (A) to node {} (B);
\draw[->] (B) to node {} (C);

\draw[->, bend right=60] (A) to node [swap] {$a''$} (B);
\draw[->, bend right=60] (B) to node [swap] {$b''$} (C);

\node (W) [node distance=1.25cm, right of=A] {};
\node (U) [node distance=1.25cm, right of=B] {};

\node (W') [node distance=0.35cm, above of=W] {$\Downarrow  \alpha$};
\node (U') [node distance=0.35cm, above of=U] {$\Downarrow  \beta$};

\node (W'') [node distance=0.35cm, below of=W] {$\Downarrow  \alpha'$};
\node (U'') [node distance=0.35cm, below of=U] {$\Downarrow  \beta'$};

\end{tikzpicture}
\end{center} 

\noindent in a bicategory. The composites $(\alpha \cdot \alpha) \ast (\beta \cdot \beta')$ and $(\alpha' \ast \beta) \cdot (\alpha' \ast \beta')$ both use the same composition of its boundary, and by the interchange law these composites are equal.
\end{examples}

The highest dimension for which there exists a hands-on definition of fully weak algebraic $n$-category is $n=4$; these are tetracategories of Alexander Hoffnung \cite{AH}. The highest dimension for which there exists a hands-on definition of fully weak algebraic $n$-category \textit{together} with a proven coherence theorem is $n=3$; these are the tricategories of Nick Gurski \cite{NG}.

The composition condition for $n$-categories is the same as for $\omega$-categories, we just truncate to $n$-dimensions. The coherence condition, however, is adjusted to include an extra requirement, reflecting the desired properties of $n$-cell composition and its associated coherence discussed above. 

\begin{coherencen}
\leavevmode
\begin{enumerate}[i)]
\item For any two ways of composing the same $k$-pasting diagram ($k < n$) in an $n$-category that agree on the composition of its $(k-1)$-dimensional boundary, there should be a coherence $(k+1)$-cell between the two composites; and
\item any two ways of composing the same $n$-pasting diagram that agree on the composition of its $(n-1)$-dimensional boundary should be equal.
\end{enumerate}
\end{coherencen}

A contraction on an $n$-globular collection is defined analogously to a contraction on a globular collection. However, in keeping with the new coherence condition, there is an extra requirement that an $n$-globular operad must meet in order to be contractible. 

\begin{definition}\label{contractible n-glob ops}
An $n$-globular operad $\bs{G_n}$ with underlying collection $g:\bs{G_n} \longrightarrow \bs{1}^*$ is \textit{contractible} if
\begin{enumerate}[i)]
\item there exists a contraction function on its underlying $n$-globular collection, and
\item any parallel pair $(\chi,\chi')$ of $n$-cells of $\bs{G_n}$ satisfying $g(\chi) = g(\chi')$ are equal.
\end{enumerate}
\end{definition} 

\begin{example} There is a unique contraction on the $n$-globular operad $\bs{T_n}$ for strict $n$-categories; see Example \ref{GlobOp T for strict n cats}. When equipped with this contraction every $k$-cell of $\bs{T_n}$ for $0 < k \leqslant n$ is a contraction cell.
\end{example}

\begin{notation}\label{the category CG-Opn} We denote by $\mathbf{C \mhyphen GOp}_{\bs{n}}$ the category of $n$-globular operads carrying a specified contraction and contraction preserving morphisms.
\end{notation}

\begin{lemma}\label{contractible n-cells} A contractible $n$-globular operad $\bs{G_n}$ is completely determined by its $k$-cells for all $k<n$.
For each triple $(x, x', \tau)$ where $(x,x')$ is a parallel pair of $(n{-}1)$-cells of $\bs{G_n}$ satisfying $g(x) = g(x') = {\partial}\tau$, there exists a \textit{unique} $n$-cell $\chi:x \longrightarrow x'$ of $\bs{G_n}$ satisfying $g(\chi) = \tau$. Similarly, given any $n$-cell $\big((\Lambda_1,...,\Lambda_m), \Lambda\big)$ of $\bs{G_n} \circ \bs{G_n}$ there is a unique choice of composite $n$-cell $\Lambda \circ (\Lambda_1,...,\Lambda_m)$.
\end{lemma}

\begin{example}
The identity $n$-cell $\text{id}_n$ of a contractible $n$-globular operad is the $n$-cell corresponding to the triple $(\text{id}_{n-1},\text{id}_{n-1},\iota)$, where $\iota$ is the simple $n$-pasting diagram in $\mathbf{1}$; see Definition \ref{trivlal pd}.
\end{example}

\begin{lemma}\label{contractible n-morphisms} Let $\bs{H_n}$ be an $n$-globular operad satisfying the second condition in Definition \ref{contractible n-glob ops}. A morphism $f:\bs{G_n} \longrightarrow \bs{H_n}$ of $n$-globular operads is completely determined by the value of $f$ on the $k$-cells of $\bs{G_n}$ for all $k<n$.
\end{lemma}

It is clear that any ($n$-)globular operad defining some sensible theory of higher category should be contractible, as contractibility is precisely the property needed to guarantee the required conditions on composition and coherence in its algebras. There are, however, contractible ($n$-)globular operads whose algebras are equipped more data than higher categories have. Take for example a contractible globular operad $\bs{G}$ with more than one 0-cell. The identity 0-cell $\text{id}_0$ provides the `do nothing' way of composing 0-cells in the algebras of $\bs{G}$. The existence of another 0-cell in $\bs{G}$ would mean that there exists some other non-trivial operation on the 0-cells of its algebras, which is not the case for higher categories.

\section{Presentations for globular operads}\label{presenations section}

Presentations for algebraic structures can be defined in terms of free-forgetful adjunctions and coequalisers. For example, a group presentation is a set $J$ whose elements we call \textit{generators}, a set $R$ whose elements we call \textit{relations} and a pair of functions 

\vspace{1mm}

\begin{center}
\begin{tikzpicture}[node distance=2.5cm, auto]

\node (A) {$UF(J)$};
\node (B) [left of=A] {$R$};

\draw[transform canvas={yshift=0.5ex},->] (B) to node {$e$} (A);
\draw[transform canvas={yshift=-0.5ex},->] (B) to node [swap] {$q$} (A);

\end{tikzpicture}
\end{center}
\noindent where $F$ and $U$ are the adjoint free-forgetful functors $F \dashv U:\mathbf{Grp} \longrightarrow \mathbf{Set}$.
We say that $(J,R,e,q)$ is a presentation for a group $G$ if the coequaliser of the diagram

\vspace{1mm}

\begin{center}
\begin{tikzpicture}[node distance=2.5cm, auto]

\node (A) {$F(J)$};
\node (B) [left of=A] {$F(R)$};

\draw[transform canvas={yshift=0.5ex},->] (B) to node {$e$} (A);
\draw[transform canvas={yshift=-0.5ex},->] (B) to node [swap] {$q$} (A);

\end{tikzpicture}
\end{center}

\noindent is isomorphic to $G$. For instance, $P = \big(\{x \}, \{r \}, e, q\big)$ where $e(r) = \text{id}$ and $q(r) = x^n$ is a presentation for the cyclic group $C_n$ of order $n$.
Presentations for other kinds of algebraic structures, such as abelian groups, monoids and rings are defined similarly by replacing the free group adjunction with the appropriate analogous adjunction. In this section we define presentations for ($n$-)globular operads. These presentations are more complex than those for structures whose underlying data is just a single set. To construct presentations for ($n$-)globular operads, we need to specify generators and relations in each dimension, beginning in dimension 0 and building the presentation inductively.  
Before we begin, we provide a preliminary motivating example for the 1-globular operad for ordinary categories. 

\begin{example}\label{pexample-presentations} The 1-globular operad $\bs{T_1}$ for ordinary (1-)categories (Example \ref{GlobOp T for strict n cats}) is the 1-globular operad with

\begin{list}{$\bullet$}{}

\item a single 0-cell, the identity $\text{id}_0$; and

\item 1-cells consisting of the operadic composites of 1-cells $i_1$ and $h_1$ whose images under the underlying collection map are as follows,

\begin{center}
\begin{tikzpicture}[node distance=2cm, auto]

\node (A) {$\text{id}_0$};
\node (B) [right of =A] {$\text{id}_0$};

\draw[->] (A) to node {$i_1$} (B);

\node (P) [node distance=3cm, right of =B] {$\cdot$};

\node (b) [node distance=0.5cm, right of=B] {};
\node (p) [node distance=0.5cm, left of=P] {};
\draw[|->, dashed] (b) to node {} (p);

\node (A') [node distance=1cm, below of=A] {$\text{id}_0$};
\node (B') [right of =A'] {$\text{id}_0$};

\draw[->] (A') to node {$h_1$} (B');

\node (P') [node distance=3cm, right of =B'] {$\cdot$};
\node (Q') [right of =P'] {$\cdot$};
\node (R') [right of =Q'] {$\cdot$};

\draw[->] (P') to node {} (Q');
\draw[->] (Q') to node {} (R');

\node (b') [node distance=0.5cm, right of=B'] {};
\node (p') [node distance=0.5cm, left of=P'] {};
\draw[|->, dashed] (b') to node {} (p');

\end{tikzpicture}
\end{center}

\noindent subject to the equalities below.
\begin{enumerate}[i)]
\item $h_1 \circ (i_1, \text{id}_1) = \text{id}_1$
\item $h_1 \circ (\text{id}_1, i_1) = \text{id}_1$ 
\item $h_1 \circ (\text{id}_1, h_1) = h_1 \circ (h_1, \text{id}_1)$ 
\end{enumerate}
\end{list}

\end{example}

The example above describes a presentation for $\bs{T_1}$ wherein the 1-cells $i_1$ and $h_1$ are the \textit{1-cell generators}, and the equations they satisfy correspond to the \textit{1-cell relations}. The only 0-cell of $\bs{T_1}$ is the identity 0-cell $\text{id}_0$, so there is no need for any 0-cell generators or relations in the presentation.

An algebra $\theta:\bs{A}_{\bs{T_1}} \longrightarrow \bs{A}$ for $\bs{T_1}$ is precisely an ordinary category with underlying 1-globular set $\bs{A}$.
The generators $i_1$ and $h_1$ provide 1-cell identities and binary composition of 1-cells in $\bs{A}$, respectively: given 1-pasting diagrams
\vspace{-1mm}
\begin{center}
\begin{tikzpicture}[node distance=2cm, auto]

\node (A) {$A$};
\node (B) [right of=A] {$A$};
\node (C) [right of=B] {$B$};
\node (D) [right of=C] {$C$};

\draw[->] (B) to node {$a$} (C);
\draw[->] (C) to node {$b$} (D);

\end{tikzpicture}
\end{center}

\noindent in $\bs{A}$ we write $\theta_1((A), i_1) = i_1(A) = 1_A$ and $\theta_1((a,b), h_1) = h_1(a,b) = ab$. 

\begin{center}
\begin{tikzpicture}[node distance=2.2cm, auto]

\node (A) {$A$};
\node (B) [right of=A] {$A$};
\node (C) [right of=B] {$A$};
\node (D) [right of=C] {$C$};

\draw[->] (A) to node {$1_A$} (B);

\draw[->] (C) to node {$ab$} (D);

\end{tikzpicture}
\end{center}

\noindent The relations ensure that the two identity axioms and the associativity axiom, respectively, are satisfied: given 1-pasting diagrams

\begin{center}
\begin{tikzpicture}[node distance=2cm, auto]

\node (A) {$A$};
\node (B) [right of=A] {$B$};
\node (C) [right of=B] {$A$};
\node (D) [right of=C] {$B$};
\node (E) [right of=D] {$C$};
\node (F) [right of=E] {$D$};

\draw[->] (A) to node {$a$} (B);

\draw[->] (C) to node {$a$} (D);
\draw[->] (D) to node {$b$} (E);
\draw[->] (E) to node {$c$} (F);

\end{tikzpicture}
\end{center}

\noindent in $\bs{A}$, equations i), ii) and iii) yield the following equalities.
$$1_Aa = h_1 \big(i_1(A), \hspace{0.5mm} \text{id}_1(a)\big) = \big(h_1 \circ (i_1,  \text{id}_1)\big) (a) = \text{id}_1(a) = a$$
$$a1_B = h_1 \big(\text{id}_1(a), \hspace{0.5mm} i_1(B)\big) = \big(h_1 \circ (\text{id}_1, i_1 )\big)(a) = \text{id}_1(a) = a$$
$$a(bc) = h_1 \big(\text{id}_1(a), \ h_1(b,c)\big) = \big( h_1 \circ (\text{id}_1, h_1) \big) (a,b,c) =  \big( h_1 \circ (h_1, \text{id}_1) \big) (a,b,c) =  h_1 \big(h_1(a,b), \ \text{id}_1(c)\big) = (ab)c$$

Observe that the generators in the presentation for $\bs{T_1}$ above correspond to the basic composition operations in an ordinary (1-)category, and the relations correspond to the axioms. Once we have formally defined presentations for ($n$-)globular operads we will be able to construct the globular operad for any theory of higher category in an analogous way. 

Before we proceed, it will be helpful to understand the process of freely adjoining cells. Recall that the category $\mathbf{GColl}_{\bs{n}} \cong \mathbf{GSet}_{\bs{n}} / \mathbf{1}^*$ of $n$-globular collections is a slice of a presheaf category (see Definition \ref{GSet_n}), so $\mathbf{GColl}_{\bs{n}}$  is itself a presheaf category \cite[Lemma 1.4.12]{MKPS}, and is therefore complete and cocomplete. In particular, $\mathbf{GColl}_{\bs{n}}$ has pushouts, allowing for the the following construction.

\begin{definition} Let $n$ be a natural number. 
\begin{enumerate}
\item For $k \leqslant n$ the $k$-\textit{ball} $\bs{B_k}$ is the $n$-globular set
\begin{center}
\begin{tikzpicture}[node distance=2.3cm, auto]

\node (A) {$\{0,1\}$};
\node (C) [left of=A] {$. . .$};
\node (D) [left of=C] {$\{0,1\}$};
\node (E) [left of=D] {$ \{ \filledstar \}$};
\node (F) [node distance=2.1cm, left of=E] {$\emptyset$};
\node (G) [node distance=2cm, left of=F] {$. . .$};
\node (H) [node distance=2cm, left of=G] {$\emptyset$};

\draw[transform canvas={yshift=0.5ex},->] (C) to node {\small ${0}$} (A);
\draw[transform canvas={yshift=-0.5ex},->] (C) to node [swap] {\small ${1}$} (A);

\draw[transform canvas={yshift=0.5ex},->] (D) to node {\small ${0}$} (C);
\draw[transform canvas={yshift=-0.5ex},->] (D) to node [swap] {\small ${1}$} (C);

\draw[transform canvas={yshift=0.5ex},->] (E) to node {\small ${0}$} (D);
\draw[transform canvas={yshift=-0.5ex},->] (E) to node [swap] {\small ${1}$} (D);

\draw[transform canvas={yshift=0.5ex},->] (F) to node {} (E);
\draw[transform canvas={yshift=-0.5ex},->] (F) to node [swap] {} (E);

\draw[transform canvas={yshift=0.5ex},->] (G) to node {} (F);
\draw[transform canvas={yshift=-0.5ex},->] (G) to node [swap] {} (F);

\draw[transform canvas={yshift=0.5ex},->] (H) to node {} (G);
\draw[transform canvas={yshift=-0.5ex},->] (H) to node [swap] {} (G);

\end{tikzpicture}
\end{center}
\noindent consisting of a single $k$-cell. Here the arrows labelled $0$ and $1$ represent the constant functions. 

\item For $-1 \leqslant k \leqslant n$ the $k$-\textit{sphere} $\bs{S_{k}}$ is the $n$-globular set
\begin{center}
\begin{tikzpicture}[node distance=2.3cm, auto]

\node (A) {$\{0,1\}$};
\node (C) [left of=A] {$. . .$};
\node (D) [left of=C] {$\{0,1\}$};
\node (E) [left of=D] {$\{0,1\}$};
\node (F) [node distance=2.1cm, left of=E] {$\emptyset$};
\node (G) [node distance=2cm, left of=F] {$. . .$};
\node (H) [node distance=2cm, left of=G] {$\emptyset$};

\draw[transform canvas={yshift=0.5ex},->] (C) to node {\small ${0}$} (A);
\draw[transform canvas={yshift=-0.5ex},->] (C) to node [swap] {\small ${1}$} (A);

\draw[transform canvas={yshift=0.5ex},->] (D) to node {\small ${0}$} (C);
\draw[transform canvas={yshift=-0.5ex},->] (D) to node [swap] {\small ${1}$} (C);

\draw[transform canvas={yshift=0.5ex},->] (E) to node {\small ${0}$} (D);
\draw[transform canvas={yshift=-0.5ex},->] (E) to node [swap] {\small ${1}$} (D);

\draw[transform canvas={yshift=0.5ex},->] (F) to node {} (E);
\draw[transform canvas={yshift=-0.5ex},->] (F) to node [swap] {} (E);

\draw[transform canvas={yshift=0.5ex},->] (G) to node {} (F);
\draw[transform canvas={yshift=-0.5ex},->] (G) to node [swap] {} (F);

\draw[transform canvas={yshift=0.5ex},->] (H) to node {} (G);
\draw[transform canvas={yshift=-0.5ex},->] (H) to node [swap] {} (G);

\end{tikzpicture}
\end{center}
\noindent consisting of a parallel pair of $k$-cells. Here it understood that the (-1)-sphere is the empty $n$-globular set.
\end{enumerate}
\end{definition}

Let $J_k$ be a set equipped with a function $f:J_k \longrightarrow \mathbf{GSet}_{\bs{n}} (\bs{B_k}, \bs{1}^*)$. Such a function corresponds to a map $\bar{f}:J_k \cdot \bs{B_k} \longrightarrow \bs{1}^*$ of $n$-globular sets, where $J_k \cdot \bs{B_k}$ denotes the coproduct of $|J_k|$ copies of  $\bs{B_k}$, under the adjunction
\vspace{-1mm}
\begin{center}
\begin{tikzpicture}[node distance=3cm, auto]

\node (A) {$\mathbf{Set}$};
\node (B) [right of=A] {$\mathbf{GSet}_{\bs{n}}$.};

\draw[->, bend left=38] (A) to node {$- \cdot \bs{B_k}$} (B);
\draw[->, bend left=38] (B) to node {$\mathbf{GSet}_{\bs{n}}(\bs{B_k}, -)$} (A);

\node (W) [node distance=1.5cm, right of=A] {$\perp$};

\end{tikzpicture}
\end{center}
Given an $n$-globular collection $g:\bs{G_n} \longrightarrow \bs{1}^*$ together with a function $\partial : J_k \longrightarrow \mathbf{GSet}_{\bs{n}}( \bs{S_{k-1}}, \bs{G_n})$ satisfying the commuativity of
\vspace{-1mm}
\begin{center}
\begin{tikzpicture}[node distance=5cm, auto]

\node (A) {$J_k$};
\node (B) [right of=A] {$\mathbf{GSet}_{\bs{n}} (\bs{B_k}, \bs{1}^*)$};
\node (X) [node distance=1.9cm, below of=A] {$\mathbf{GSet}_{\bs{n}}( \bs{S_{k-1}}, \bs{G_n})$};
\node (Y) [right of=X] {$\mathbf{GSet}_{\bs{n}}( \bs{S_{k-1}}, \bs{1}^*)$};

\draw[->] (A) to node {} (B);
\draw[->] (A) to node [swap] {$\partial$} (X);
\draw[->] (X) to node [swap] {$g \cdot -$} (Y);
\draw[->] (B) to node {$- \cdot i_k$} (Y);

\end{tikzpicture}
\end{center} 
where $i_k : \bs{S_{k-1}} \longrightarrow \bs{B_k}$ denotes the inclusion map, we define $\bs{J_k}$ via the pushout
\vspace{1mm}
\begin{center}
\begin{tikzpicture}[node distance=3.5cm, auto]

\node (A) {$J_k \cdot \bs{S_{k-1}}$};
\node (B) [right of=A] {$J_k \cdot \bs{B_k}$};
\node (X) [node distance=1.9cm, below of=A] {$\bs{G_n}$};
\node (Y) [right of=X] {$\bs{J_k}$};

\draw[->] (A) to node {$J_k \cdot i_k$} (B);
\draw[->] (A) to node [swap] {$\bar{\partial}$} (X);
\draw[->] (X) to node [swap] {} (Y);
\draw[->] (B) to node {} (Y);

\end{tikzpicture}
\end{center} 

\noindent in $\mathbf{GColl}_{\bs{n}}$. The commuativity of the first square above says that $\partial : J_k \cdot \bs{S_{k-1}} \longrightarrow \bs{G_n}$ is a map of $n$-globular collections, rather than just a map of $n$-globular sets. We think of $\bs{J_k}$ as the $n$-globular collection obtained by adjoining a set $J_k$ of $k$-cells to $\bs{G_n}$ whose sources and targets are specified by $\partial$. 
This method of freely adjoining cells can be extended to $n$-globular operads using the following proposition.

\begin{proposition}\cite[Chapter 9.3]{BaWe}
If $T=(T, \mu, \eta)$ is a finitary monad on a complete and cocomplete category then the category $T \mhyphen \mathbf{Alg}$ of algebras for $T$ is also complete and cocomplete.
\end{proposition}

It shown in \cite[Chapter 6.5]{TL} that there is a free functor $F:\mathbf{GColl}_{\bs{n}} \longrightarrow \mathbf{GOp}_{\bs{n}}$ left adjoint to canonical forgetful functor, the adjunction  $F \dashv U:\mathbf{GOp}_{\bs{n}} \longrightarrow \mathbf{GColl}_{\bs{n}}$ is monadic, and that the monad induced on $\mathbf{GColl}_{\bs{n}}$ is finitary. It follows that $\mathbf{GOp}_{\bs{n}}$ is complete and cocomplete, in particular, it has pushouts. We can now define, for any $n$-globular operad $\bs{G_n}$ together with a morphism $\bar{\partial} : J_k \cdot \bs{S_{k-1}} \longrightarrow U(\bs{G_n})$ of $n$-globular collections, an $n$-globular operad $\bs{J_k}$ via the pushout in $\mathbf{GOp}_{\bs{n}}$ below.

\begin{center}
\begin{tikzpicture}[node distance=2cm, auto]

\node (A) {$J_k \cdot F(\bs{S_{k-1}})$};
\node (B) [node distance=4.5cm, right of=A] {$J_k \cdot F(\bs{B_k})$};
\node (X) [below of=A] {$\bs{G_n}$};
\node (Y) [node distance=4.5cm, right of=X] {$\bs{J_k}$};

\draw[->] (A) to node {$J_k \cdot F(i_k)$} (B);
\draw[->] (A) to node [swap] {$\hat{\partial}$} (X);
\draw[->] (X) to node [swap] {} (Y);
\draw[->] (B) to node {} (Y);

\end{tikzpicture}
\end{center}

We will also make use of the following truncation functors.

\begin{definition}\label{truncationfunctor} Let $n$ be a natural number and let $k$ be an integer with $-1 \leqslant k \leqslant n$. The $k^{th}$ \textit{truncation functor} $Tr_k:\mathbf{GOp}_{\bs{n}} \longrightarrow \mathbf{GOp}_{\bs{n}}$
is the functor sending an $n$-globular operad $\bs{G_n}$ to the $n$-globular operad whose $j$-cells are those of $\bs{G_n}$ for all $j \leqslant k$ and whose only $m$-cell is the identity $m$-cell $\text{id}_m$ for all $m > k$.
\end{definition}

\begin{notation}\label{counit} Let $\epsilon_k:Tr_k \Rightarrow 1$ denote the natural transformation whose components are the inclusions.
\end{notation}

We are now ready to define presentations for $n$-globular operads. 
For each integer $k$, $-1 \leqslant k \leqslant n$, we define a category $\bs{k} \mhyphen \bf{Pres}$ of $k$-presentations for $n$-globular operads together with an adjunction
\begin{center}
\begin{tikzpicture}[node distance=3cm, auto]

\node (A) {$\bs{k} \mhyphen \mathbf{Pres}$};
\node (B) [right of=A] {$\mathbf{GOp}_{\bs{n}}$};

\draw[->, bend left=40] (A) to node {$F_k$} (B);
\draw[->, bend left=40] (B) to node {$V_k$} (A);

\node (W) [node distance=1.5cm, right of=A] {$\perp$};

\end{tikzpicture}
\end{center}
satisfying $F_kV_k=Tr_k$ and with counit $\epsilon_k$. These definitions are recursive; for $k=-1$, we define $\bs{k} \mhyphen \mathbf{Pres}$ to be the terminal category and $F_{-1}$ to be the functor picking out the initial $n$-globular operad $\bs{1}$ consisting only of the identity cells. For $k \geqslant 0$, we construct $\bs{k} \mhyphen \bf{Pres}$ and the accompanying adjunction below.

\begin{definition}
A \textit{$k$-presentation for an $n$-globular operad} is a tuple $P_k = (P_{k-1}, J_k, \partial_k, R_k, r_k)$ where 

\begin{list}{$\bullet$}{}

\item $P_{k-1}$ is a $(k{-}1)$-presentation;

\item $J_k$ is a set of \textit{$k$-cell generators}, equipped with a function $J_k \longrightarrow \mathbf{GSet}_{\bs{n}}(\bs{B_k}, \bs{1}^*)$;

\item $\partial_k : J_k \longrightarrow \mathbf{GSet}_{\bs{n}}(\bs{S_{k-1}}, UF_{k-1}(P_{k-1}))$ is a function for which the square 
\vspace{1mm}
\begin{center}
\begin{tikzpicture}[node distance=2cm, auto]

\node (A) {$J_k$};
\node (B) [node distance=5cm,right of=A] {$\mathbf{GSet}_{\bs{n}}(\bs{B_k}, \bs{1}^*)$};
\node (X) [below of=A] {$\mathbf{GSet}_{\bs{n}}(\bs{S_{k-1}}, UF_{k-1}(P_{k-1}))$};
\node (Y) [node distance=5cm, right of=X] {$\mathbf{GSet}_{\bs{n}}(\bs{S_{k-1}}, \bs{1}^*)$};

\draw[->] (A) to node {} (B);
\draw[->] (A) to node [swap] {$\partial_k$} (X);
\draw[->] (X) to node [swap] {} (Y);
\draw[->] (B) to node {$- \cdot i_k$} (Y);

\end{tikzpicture}
\end{center} 
commutes, where the bottom map is induced by the underlying collection map $UF_{k-1}(P_{k-1}) \longrightarrow \bs{1}^*$ of the $n$-globular operad $F_{k-1}(P_{k-1})$;

\item $R_k$ is a set of \textit{$k$-cell relations}; and 

\item $r_k : R_k \longrightarrow \mathbf{GSet}_{\bs{n}}(\bs{S_k}, U(\bs{J_k}))$ is a function, where $\bs{J_k}$ is defined via the pushout
\begin{center}
\begin{tikzpicture}[node distance=2cm, auto]

\node (A) {$J_k \cdot F(\bs{S_{k-1}})$};
\node (B) [node distance=4.5cm,right of=A] {$J_k \cdot F(\bs{B_k})$};
\node (X) [below of=A] {$F_{k-1}(P_{k{-}1})$};
\node (Y) [node distance=4.5cm, right of=X] {$ \ \bs{J_k} $,};

\draw[->] (A) to node {$J_k \cdot F(i_k)$} (B);
\draw[->] (A) to node [swap] {$\hat{\partial}_k$} (X);
\draw[->] (X) to node [swap] {$w_k$} (Y);
\draw[->] (B) to node {$w'_k$} (Y);

\end{tikzpicture}
\end{center} 
in $\mathbf{GOp}_{\bs{n}}$, for which there exists a (necessarily unique) function $R_k \longrightarrow \mathbf{GSet}_{\bs{n}}(\bs{B_k}, \bs{1}^*)$ such the square below commutes. 
\begin{center}
\begin{tikzpicture}[node distance=2cm, auto]

\node (A) {$R_k$};
\node (B) [node distance=4.5cm,right of=A] {$\mathbf{GSet}_{\bs{n}}(\bs{B_k}, \bs{1}^*)$};
\node (X) [below of=A] {$\mathbf{GSet}_{\bs{n}}(\bs{S_k}, U(\bs{J_k}))$};
\node (Y) [node distance=4.5cm, right of=X] {$\mathbf{GSet}_{\bs{n}}(\bs{S_k}, \bs{1}^*)$};

\draw[->] (A) to node {} (B);
\draw[->] (A) to node [swap] {$r_k$} (X);
\draw[->] (X) to node [swap] {} (Y);
\draw[->] (B) to node {$- \cdot \Delta_k$} (Y);

\end{tikzpicture}
\end{center} 
Here $\Delta_k : \bs{S_k} \longrightarrow \bs{B_k}$ denotes the map sending both $k$-cells of $\bs{S_k}$ to the unique $k$-cell of $\bs{B_k}$. 
\end{list}

\end{definition}

\begin{definition}\label{Morphisms of $k$-presentations.}  
A \textit{morphism $P_k \longrightarrow P'_k$ of $k$-presentations} is a triple $(\pi,\varrho,\rho)$ where 
$\pi:P_{k-1} \longrightarrow P'_{k-1}$ is a morphism of $(k{-}1)$-presentations and
$\varrho:J_k \longrightarrow J'_k$ and $\rho:R_k \longrightarrow R'_k $ are functions, satisfying the commutativity of the following diagrams.
\vspace{-3mm}
\begin{center}
\begin{tikzpicture}[node distance=4cm, auto]

\node (A) {$J_k$};
\node (B) [node distance=3cm, right of=A] {$J'_k $};
\node (P) [node distance=1.6cm, below of=A] {};
\node (Q) [node distance=1.5cm, right of=P] {$\mathbf{GSet}_{\bs{n}}(\bs{B_k}, \bs{1}^*)$};

\draw[->] (A) to node {$\varrho$} (B);
\draw[->] (A) to node [swap] {} (Q);
\draw[->] (B) to node {} (Q);

\node (X') [node distance=4.5cm, right of=Q] {};

\node (X'') [node distance=2mm, below of=X'] {$\mathbf{GSet}_{\bs{n}} (\bs{S_{k-1}}, UF_{k-1}(P_{k-1}))$};
\node (A') [node distance=2cm, above of=X''] {$J_k$};
\node (B') [node distance=7.5cm, right of=A'] {$J'_k$};
\node (Y') [node distance=7.5cm, right of=X''] {$\mathbf{GSet}_{\bs{n}} (\bs{S_{k-1}}, UF_{k-1}(P'_{k-1}))$};

\draw[->] (A') to node {$\varrho$} (B');
\draw[->] (A') to node [swap] {$\partial_k$} (X'');
\draw[->] (X'') to node [swap] {$UF_{k-1}(\pi) \cdot -$} (Y');
\draw[->] (B') to node {$\partial'_k$} (Y');

\end{tikzpicture}
\end{center}

\begin{center}
\begin{tikzpicture}[node distance=4cm, auto]

\node (A) {$R_k$};
\node (B) [node distance=3cm, right of=A] {$R'_k $};
\node (P) [node distance=1.6cm, below of=A] {};
\node (Q) [node distance=1.5cm, right of=P] {$\mathbf{GSet}_{\bs{n}}(\bs{B_k}, \bs{1}^*)$};

\draw[->] (A) to node {$\rho$} (B);
\draw[->] (A) to node [swap] {} (Q);
\draw[->] (B) to node {} (Q);

\node (A') [node distance=2.5cm, right of=B] {};

\node (A'') [node distance=2mm, above of=A'] {$R_k$};
\node (B') [node distance=5.5cm, right of=A''] {$R'_k$};
\node (X') [node distance=2cm, below of=A''] {$\mathbf{GSet}_{\bs{n}} (\bs{S_k}, U(\bs{J_k}) )$};
\node (Y') [node distance=5.5cm, right of=X'] {$\mathbf{GSet}_{\bs{n}} (\bs{S_k}, U(\bs{J}'_{\bs{k}}))$};

\draw[->] (A'') to node {$\rho$} (B');
\draw[->] (A'') to node [swap] {$r_k$} (X');
\draw[->] (X') to node [swap] {$U(\bs{\varrho}) \cdot -$} (Y');
\draw[->] (B') to node {$r'_k$} (Y');

\end{tikzpicture}
\end{center}
The map $\bs{\varrho}:\bs{J_k} \longrightarrow \bs{J}'_{\bs{k}}$ above is defined using the universal property of pushouts.

\end{definition}

\begin{definition}\label{The free pres functor.} 
The \textit{free functor} $F_k: \bs{k} \mhyphen \mathbf{Pres} \longrightarrow \mathbf{GOp}_{\bs{n}}$ sends a $k$-presentation $P_k = (P_{k-1}, J_k, \partial_k, R_k, r_k)$ to the coequaliser
\begin{center}
\begin{tikzpicture}[node distance=2.5cm, auto]

\node (A) {$F_{k}(P_k)$};
\node (B) [left of=A] {$\bs{J_k}$};
\node (C) [node distance=3cm, left of=B] {$R_k \cdot F(\bs{B_k})$};

\draw[transform canvas={yshift=0.5ex},->] (C) to node {$\hat{e}_k$} (B);
\draw[transform canvas={yshift=-0.5ex},->] (C) to node [swap] {$\hat{q}_k$} (B);

\draw[->] (B) to node {} (A);

\end{tikzpicture}
\end{center}
where $e_k, q_k : R_k \longrightarrow \mathbf{GSet}_{\bs{n}}(\bs{B_k}, U(\bs{J_k}))$ are the pair of maps induced by the function $r_k : R_k \longrightarrow \mathbf{GSet}_{\bs{n}}(\bs{S_k}, U(\bs{J_k}))$.
The value of $F_k$ on morphisms is defined using the universal property of coequalisers.
\end{definition}

\begin{definition} The \textit{forgetful functor} $V_k: \mathbf{GOp}_{\bs{n}} \longrightarrow \bs{k} \mhyphen \mathbf{Pres}$ sends an $n$-globular operad $\bs{G_n}$ to the $k$-presentation $V_k(\bs{G_n})=\left(V_{k-1}(\bs{G_n}), \, J_{(\bs{G_n},k)}, \, \partial_{(\bs{G_n},k)}, \, R_{(\bs{G_n},k)}, \,  r_{(\bs{G_n},k)} \right)$
defined as follows.

\begin{list}{$\bullet$}{}

\item $J_{(\bs{G_n},k)} = \mathbf{GSet}_{\bs{n}}(\bs{B_k}, U(\bs{G_n}))$ is the set of $k$-cells of $\bs{G_n}$.

\item $\partial_{(\bs{G_n},k)}$ is the function $- \cdot i_k : \mathbf{GSet}_{\bs{n}}(\bs{B_k}, U(\bs{G_n})) \longrightarrow \mathbf{GSet}_{\bs{n}}(\bs{S_{k-1}}, UTr_{k-1}(\bs{G_n}))$; here we are using the fact that $F_{k-1}V_{k-1} = Tr_{k-1}$ is the $k^{th}$ truncation functor (see Definition \ref{truncationfunctor}).

\item $R_{(\bs{G_n},k)}$ is the pullback object
\begin{center}
\begin{tikzpicture}[node distance=2.4cm, auto]

\node (P) {$R_{(\bs{G_n},k)}$};
\node (Q) [node distance=7cm,right of=P] {$\mathbf{GSet}_{\bs{n}}\big(\bs{B_k}, U\big(\bs{J}_{(\bs{G_n, k})}\big)\big)$};
\node (U) [below of=P] {$\mathbf{GSet}_{\bs{n}}\big(\bs{B_k}, U\big(\bs{J}_{(\bs{G_n, k})}\big)\big)$};
\node (V) [node distance=7cm, right of=U] {$\mathbf{GSet}_{\bs{n}}(\bs{B_k}, U(\bs{G_n}))$};

\draw[->] (P) to node {$e_{(\bs{G_n},k)}$} (Q);
\draw[->] (P) to node [swap] {$q_{(\bs{G_n},k)}$} (U);
\draw[->] (U) to node [swap] {$U(\bs{\epsilon_k}) \cdot -$} (V);
\draw[->] (Q) to node {$U(\bs{\epsilon_k}) \cdot -$} (V);

\end{tikzpicture}
\end{center}
where $\bs{\epsilon_k}$ is the unique morphism satisfying commutativity of the diagram below.
\vspace{1mm}
\begin{center}
\begin{tikzpicture}[node distance=2.2cm, auto]

\node (P) {$J_{(\bs{G_n},k)} \cdot F(\bs{S_{k-1}})$};
\node (Q) [node distance=6cm,right of=P] {$J_{(\bs{G_n},k)} \cdot F(\bs{B_k})$};
\node (U) [below of=P] {$F_{k-1}V_{k-1}(\bs{G_n})$};
\node () [node distance=2.2cm, left of=U] {$Tr_{k-1}(\bs{G_n}) = $};
\node (V) [node distance=6cm, right of=U] {$\bs{J}_{(\bs{G_n, k})}$};
\node (v) [node distance=2.2cm, right of=V] {};
\node (D) [node distance=1.5cm, below of=v] {$\bs{G_n}$};

\draw[->] (P) to node {$J_{(\bs{G_n},k)} \cdot F(i_k)$} (Q);
\draw[->] (P) to node [swap] {$\hat{\partial}_{(\bs{G_n},k)}$} (U);
\draw[->] (U) to node [swap] {$w_k$} (V);
\draw[->] (Q) to node {$w'_k$} (V);
\draw[->, bend right=16] (U) to node [swap] {$\epsilon_{k-1}$} (D);
\draw[->, bend left] (Q) to node {} (D);
\draw[->, dashed] (V) to node [swap] {$\bs{\epsilon_k}$} (D);

\end{tikzpicture}
\end{center}
\vspace{1mm}
The outer arrows of this diagram commute since the components of the natural transformation $\epsilon_{k-1}$ are the inclusion morphisms; see Notation \ref{counit}. 

\item $r_{(\bs{G_n},k)}$ is the function $R_{(\bs{G_n},k)} \longrightarrow \mathbf{GSet}_{\bs{n}}\big(\bs{S_k}, U\big(\bs{J}_{(\bs{G_n, k})}\big)\big)$ induced by the pair of functions $e_{(\bs{G_n},k)}$ and $q_{(\bs{G_n},k)}$ above.

\end{list}

\noindent The value of $V_k$ on morphisms of $n$-globular operads is defined similarly.

\end{definition} 

\begin{proposition} The composite $F_kV_k$ is isomorphic to the $k^{th}$ truncation functor $Tr_k$. 
\end{proposition}

\begin{proof}
For any $n$-globular operad $\bs{G_n}$ the diagram
\begin{center}
\begin{tikzpicture}[node distance=4cm, auto]

\node (B) {$\bs{J}_{(\bs{G_n,k})}$};
\node (C) [left of=B] {$R_{(\bs{G_n},k)} \cdot F(\bs{B_k})$};
\node (A) [node distance=3.2cm, right of=B] {$Tr_k(\bs{G_n})$};

\draw[transform canvas={yshift=0.75ex},->] (C) to node {$\hat{e}_{(\bs{G_n},k)}$} (B);
\draw[transform canvas={yshift=-0.25ex},->] (C) to node [swap] {$\hat{q}_{(\bs{G_n},k)}$} (B);
\draw[->] (B) to node {$\bs{\epsilon_k}$} (A);

\end{tikzpicture}
\end{center}
is a coequaliser in $\mathbf{GOp}_{\bs{n}}$ by construction, so it follows from Definition \ref{The free pres functor.} that $Tr_k(\bs{G_n})$ is isomorphic to $F_kV_k(\bs{G_n})$.
\end{proof}

\begin{proposition} The free functor $F_k$ is left adjoint to the forgetful functor $V_k$, and the counit of the adjunction is $\epsilon_k:Tr_k  \Rightarrow 1$.
\end{proposition}

\begin{proof} We prove the proposition by constructing a unit $\eta_k:1 \longrightarrow V_kF_k$ such that $\eta_k$ and $\epsilon_k$ satisfy the triangle identities. Let $P_k = (P_{k-1}, J_k, \partial_k, R_k, r_k)$ be a $k$-presentation. By Definition \ref{Morphisms of $k$-presentations.}, a morphism $\eta_k:P_k \longrightarrow V_kF_k(P_k)$ of $k$-presentations consists of three compatible maps
\vspace{2mm}
\begin{enumerate}
\item $\eta^1_k:P_{k-1} \longrightarrow V_{k-1}F_k(P_k)$; 
\item $\eta^2_k:J_k \longrightarrow J_{(F_k(P_k),k)}$; and 
\item $\eta^3_k:R_k \longrightarrow R_{(F_k(P_k),k)} $.
\end{enumerate}
\vspace{2mm}
Define $\eta_k^1$ to be the morphism corresponding under adjunction to the composite

\begin{center}
\begin{tikzpicture}[node distance=2.3cm, auto]

\node (A) {$F_{k-1}(P_{k-1})$};
\node (B) [node distance=2.7cm, right of=A] {$ \bs{J_k}$};
\node (C) [right of=B] {$F_k(P_k)$.};

\draw[->] (A) to node {$w_k$} (B);
\draw[->] (B) to node {} (C);

\end{tikzpicture}
\end{center}

\noindent Next, observe that since $J_{(F_k(P_k),k)}=\mathbf{GSet}_{\bs{n}}(\bs{B_k}, UF_k(P_k))$ the composite 

\begin{center}
\begin{tikzpicture}[node distance=2.6cm, auto]

\node (A) {$J_k \cdot \bs{B_k}$};
\node (B) [right of=A] {$U(\bs{J_k})$};
\node (C) [right of=B] {$UF_k(P_k)$};

\draw[->] (A) to node {$w'_k$} (B);
\draw[->] (B) to node {} (C);

\end{tikzpicture}
\end{center}

\noindent is equivalently a function $\eta_k^2:J_k \longrightarrow J_{(F_k(P_k),k)}$ of sets. Finally, let $\bs{\eta_k}$ be the unique morphism satisfying the commutativity of the diagram

\begin{center}
\resizebox{4.5in}{!}{
\begin{tikzpicture}[node distance=2.5cm, auto]

\node (P) {$J_k \cdot F(\bs{S_{k-1}})$};
\node (Q) [node distance=4.5cm,right of=P] {$J_k \cdot F(\bs{B_k})$};
\node (L) [node distance=1.9cm, below of=Q, right of=Q] {$J_{(F_k(P_k),k)} \cdot F(\bs{B_k})$};
\node (U) [below of=P] {$F_{k-1}(P_{k{-}1})$};
\node () [node distance=2.4cm, left of=U] {$Tr_{k-1}F_{k-1}(P_{k{-}1}) =$};
\node (A) [node distance=1.9cm, below of=U, right of=U] {$Tr_{k-1}F_k(P_k)$};
\node (V) [node distance=4.5cm, right of=U] {$\bs{J_k}$};
\node (D) [node distance=1.9cm, below of=V, right of=V] {$\bs{J}_{(F_k(P_k),\bs{k})}$};

\node () [node distance=1cm, below of=U] {$Tr_{k-1}(\eta_k^1) \ $};

\draw[->] (P) to node {$J_k \cdot F(i_k)$} (Q);
\draw[->] (P) to node [swap] {$\hat{\partial}_k$} (U);
\draw[->] (U) to node [swap] {$w_k$} (V);
\draw[->] (Q) to node [swap] {$w'_k$} (V);
\draw[->] (Q) to node {$\eta_k^2 \cdot F(1)$} (L);
\draw[->] (U) to node [swap] {} (A);
\draw[->, dashed] (V) to node [swap] {$\bs{\eta_k}$} (D);
\draw[->] (L) to node {} (D);
\draw[->] (A) to node [swap] {} (D);

\end{tikzpicture}
}
\end{center}

\noindent and define $\eta^3_k$ to be the unique morphism making the diagram below commute.

\begin{center}
\resizebox{5.6in}{!}{
\begin{tikzpicture}[node distance=2.5cm, auto]

\node (A) {$R_{(F_k(P_k),k)}$};
\node (B) [node distance=7cm, right of=A] {$\mathbf{GSet}_{\bs{n}} \big( \bs{B_k}, U\big(\bs{J}_{(F_k(P_k),\bs{k})}\big) \big)$};
\node (X) [below of=A] {$\mathbf{GSet}_{\bs{n}} \big( \bs{B_k}, U\big(\bs{J}_{(F_k(P_k),\bs{k})}\big) \big)$};
\node (Y) [below of=B] {$\mathbf{GSet}_{\bs{n}} ( \bs{B_k}, UF_k(P_k) )$};

\node (P) [node distance=1.8cm, left of=A, above of=A] {$R_k$};
\node (C) [node distance=1.8cm, left of=B, above of=B] {$\mathbf{GSet}_{\bs{n}} ( \bs{B_k}, U(\bs{J_k}) )$};
\node (Z) [node distance=1.9cm, left of=X, above of=X] {};
\node (Z') [node distance=0.3cm, above of=Z] {};
\node (Z'') [node distance=0.2cm, below of=Z, right of=Z] {};
\node () [node distance=0.9cm, left of=Z] {$\mathbf{GSet}_{\bs{n}} ( \bs{B_k}, U(\bs{J_k}) )$};

\draw[->] (A) to node {$q_{(F_k(P_k),k)}$} (B);
\draw[->] (A) to node {$e_{(F_k(P_k),k)}$} (X);
\draw[->] (X) to node [swap] {$U(\bs{\epsilon_k}) \cdot -$} (Y);
\draw[->] (B) to node {$U(\bs{\epsilon_k}) \cdot -$} (Y);

\draw[->] (Z) to node [swap] {$U(\bs{\eta_k}) \cdot -$} (X);
\draw[->] (C) to node {} (B);
\draw[->, dashed] (P) to node {$\eta^3_k$} (A);
\draw[->] (P) to node [swap] {$e_k$} (Z');
\draw[->] (P) to node {$q_k$} (C);

\node () [node distance=1cm, above of=B]  {$U(\bs{\eta_k}) \cdot -$};

\end{tikzpicture}
}
\end{center}

\noindent It is now a straighforward diagram chase to check that these morphisms are the components of a natural transformation $\eta_k:1 \longrightarrow V_kF_k$. The triangle identities can be verified by observing that $\epsilon_kF_k$, $V_k\epsilon_k$, $F_k\eta_k$ and $\eta_kV_k$ are all identity natural transformations.
\end{proof}

\begin{definition} 
A \textit{presentation} for an $n$-globular operad $\bs{G_n}$ is an $n$-presentation $P_n$ together with an isomorphism $F_n(P_n) \longrightarrow \bs{G_n}$.
\end{definition}

\begin{lemma}\label{map is determined at k-cell generators} Let $P_n$ be a presentation for an $n$-globular operad  $\bs{G_n}$. By the universal properties of coequalisers and pushouts, a morphism $\bs{G_n} \longrightarrow \bs{H_n}$ of $n$-globular operads is completely determined by its value on the $k$-cell generators for all $0 \leqslant k \leqslant n$.
\end{lemma}

\begin{definitions}
We can also define a category $\mathbf{Pres}$ of presentations for globular operads together with a free-forgetful adjunction

\begin{center}
\begin{tikzpicture}[node distance=3cm, auto]

\node (A) {$\mathbf{Pres}$};
\node (B) [right of=A] {$\mathbf{GOp}$};

\draw[->, bend left=40] (A) to node {$F$} (B);
\draw[->, bend left=40] (B) to node {$V$} (A);

\node (W) [node distance=1.5cm, right of=A] {$\perp$};

\end{tikzpicture}
\end{center}

\noindent satisfying $FV=1$ and whose counit is the identity natural transformation.
The \textit{category $\mathbf{Pres}$ of presentations for globular operads} is the limit of the diagram 

\begin{center}
\begin{tikzpicture}[node distance=3cm, auto]

\node (A) {$...$};
\node (B) [node distance=2.5cm, right of=A] {$\bs{1} \mhyphen \mathbf{Pres}$};
\node (C) [right of=B] {$\bs{0} \mhyphen \mathbf{Pres}$};
\node (D) [node distance=3.2cm, right of=C] {$(-\bs{1})\mhyphen \mathbf{Pres}$};

\draw[->] (A) to node {} (B);
\draw[->] (B) to node {$U_1$} (C);
\draw[->] (C) to node {$U_0$} (D);

\end{tikzpicture}
\end{center}

\noindent where the $U_k$s are the canonical forgetful functors.
For each $k \in \mathbb{N}$ we have $U_{k-1}V_k = V_{k-1}:\mathbf{GOp} \longrightarrow \bs{k} \mhyphen \mathbf{Pres}$, so the $V_k$s form a cone over the diagram, defining a \textit{forgetful functor} $V:\mathbf{GOp} \longrightarrow \mathbf{Pres}$. 
The value of the \textit{free functor} $F:\mathbf{Pres} \longrightarrow \mathbf{GOp}$ at a presentation $P$ is the colimit of the diagram 

\vspace{1mm}

\begin{center}
\begin{tikzpicture}[node distance=2.5cm, auto]

\node (A) {$F_{-1}(P_{-1})$};
\node (B) [node distance=2.7cm, right of=A] {$F_0(P_0)$};
\node (C) [right of=B] {$F_1(P_1)$};
\node (D) [node distance=2cm, right of=C] {$...$};

\draw[->] (A) to node {} (B);
\draw[->] (B) to node {} (C);
\draw[->] (C) to node {} (D);

\end{tikzpicture}
\end{center}

\noindent where each $F_{k-1}(P_{k-1}) \longrightarrow F_k(P_k)$ is given by the composite below. 

\begin{center}
\begin{tikzpicture}[node distance=2.5cm, auto]

\node (A) {$F_{k-1}(P_{k-1})$};
\node (B) [node distance=2.7cm, right of=A] {$\bs{J_k}$};
\node (C) [right of=B] {$F_k(P_k)$};

\draw[->] (A) to node {$w_k$} (B);
\draw[->] (B) to node {} (C);

\end{tikzpicture}
\end{center}
It is now a straightforward exercise using the universal properties of sequential limits and colimits to show that $FV = 1$ and that $F$ is left adjoint to $V$.
\end{definitions}

\begin{definition}
A \textit{presentation} for a globular operad $\bs{G}$ is a presentation $P$ together with an isomorphism $F(P) \longrightarrow \bs{G}$.
\end{definition}

Every $n$-globular operad $\bs{G_n}$ has at least one presentation, namely $V_n(\bs{G_n})$. However, it is often possible to find a simpler presentation. Below is a detailed example of a presentation for the $2$-globular operad for strict $2$-categories; see Example \ref{GlobOp T for strict n cats}. 

\begin{example}\label{biased presentation for T2} To construct a presentation for the $2$-globular operad $\bs{T_2}$ for strict 2-categories, we first define $P_0 \coloneq (\star, \emptyset, !, \emptyset, !)$, where $\star$ is the unique object of the category $(-1) \mhyphen \mathbf{Pres}$ of (-1)-presentations. Then $F_0(P_0) \cong \bs{1}$ is the initial $2$-globular operad consisting only of identity cells.

Next, we construct a 1-presentation $P_1 = (P_0, J_1, \partial_1, R_1, r_1)$. Define $J_1 \coloneq \{ i_1, h_1 \}$ and let the function $J_1 \longrightarrow \mathbf{GSet}_{\bs{2}}(\bs{B_1}, \bs{1}^*)$ be the one illustrated below.
\vspace{1mm}
\begin{center}
\begin{tikzpicture}[node distance=2cm, auto]

\node (B) {$i_1$};
\node (P) [node distance=3cm, right of =B] {$\cdot$};

\node (b) [node distance=0.5cm, right of=B] {};
\node (p) [node distance=0.5cm, left of=P] {};
\draw[|->, dashed] (b) to node {} (p);

\node (B') [node distance=8mm, below of=B] {$h_1$};
\node (P') [node distance=3cm, right of =B'] {$\cdot$};
\node (Q') [right of =P'] {$\cdot$};
\node (R') [right of =Q'] {$\cdot$};

\draw[->] (P') to node {} (Q');
\draw[->] (Q') to node {} (R');

\node (b') [node distance=0.5cm, right of=B'] {};
\node (p') [node distance=0.5cm, left of=P'] {};
\draw[|->, dashed] (b') to node {} (p');

\end{tikzpicture}
\end{center}
\vspace{1mm}
Since $\bs{S_0}$ is empty in dimensions $\geqslant 0$ and $F_0(P_0)$ has exactly one 0-cell, there exists a unique function $\partial_1:J_1 \longrightarrow \mathbf{GSet}_{\bs{2}} (\bs{S_0}, UF_0(P_0))$. It follows that $\bs{J_1}$ is the 2-globular operad whose only 0 and 2-cells are the identity cells $\text{id}_0$ and $\text{id}_2$, and whose 1-cells are the free operadic composites of $i_1$ and $h_1$: the 1-cells of $\bs{J_1}$ are $\text{id}_1, \ i_1, \ h_1, \ h_1 \circ (\text{id}_1,i_1), \ h_1 \circ (h_1,\text{id}_1), \ h_1 \circ (h_1,h_1), \, ...$ and so on. 
Define $R_1 = \{ u,v,a \}$ and let $r_1 : R_1 \longrightarrow \mathbf{GSet}_{\bs{2}}(\bs{S_2}, U(\bs{J_1}))$ be the function induced by the pair of functions
\vspace{1mm}
\begin{center}
\begin{tikzpicture}[node distance=3.5cm, auto]

\node (B) {$\mathbf{GSet}_{\bs{2}}(\bs{B_2}, U(\bs{J_1}))$};
\node (C) [left of=B] {$R_1$};

\draw[transform canvas={yshift=0.75ex},->] (C) to node {$e_1$} (B);
\draw[transform canvas={yshift=-0.25ex},->] (C) to node [swap] {$q_1$} (B);

\end{tikzpicture}
\end{center}
\vspace{-1mm}
defined below.
\begin{enumerate}[i)]

\item 
$e_1(u) = h_1 \circ (i_1,\text{id}_1)$\\
$q_1(u) = \text{id}_1$

\vspace{2mm}

\item
$e_1(v)=h_1 \circ (\text{id}_1,i_1)$\\
$q_1(v)=\text{id}_1$

\vspace{2mm}

\item
$e_1(a) = h_1 \circ (h_1, \text{id}_1)$\\ 
$q_1(a) = h_1 \circ (\text{id}_1,h_1)$

\end{enumerate}

\vspace{1mm}

$F_1(P_1)$ is then the 2-globular operad whose only 0 and 2-cells are the identity cells $\text{id}_0$ and $\text{id}_2$, and whose 1-cells are the operadic composites of 1-cells $i_1$ and $h_1$ with the following images under the underlying collection map,
\vspace{-1mm}
\begin{center}
\begin{tikzpicture}[node distance=2cm, auto]

\node (A) {$\text{id}_0$};
\node (B) [node distance=2.4cm, right of =A] {$\text{id}_0$};

\draw[->] (A) to node {$i_1$} (B);

\node (P) [node distance=3cm, right of =B] {$\cdot$};

\node (b) [node distance=0.5cm, right of=B] {};
\node (p) [node distance=0.5cm, left of=P] {};
\draw[|->, dashed] (b) to node {} (p);

\node (A') [node distance=9mm, below of=A] {$\text{id}_0$};
\node (B') [node distance=2.4cm, right of =A'] {$\text{id}_0$};

\draw[->] (A') to node {$h_1$} (B');

\node (P') [node distance=3cm, right of =B'] {$\cdot$};
\node (Q') [right of =P'] {$\cdot$};
\node (R') [right of =Q'] {$\cdot$};

\draw[->] (P') to node {} (Q');
\draw[->] (Q') to node {} (R');

\node (b') [node distance=0.5cm, right of=B'] {};
\node (p') [node distance=0.5cm, left of=P'] {};
\draw[|->, dashed] (b') to node {} (p');

\end{tikzpicture}
\end{center}
subject to the equalities below.
\begin{enumerate}[i)]
\item $h_1 \circ (i_1,\text{id}_1) = \text{id}_1$
\item $h_1 \circ (\text{id}_1,i_1) = \text{id}_1$
\item $h_1 \circ (h_1, \text{id}_1) = h_1 \circ (\text{id}_1,h_1)$
\end{enumerate} 
Observe that $F_1(P_1) \cong Tr_1(\bs{T_2})$, so $F_1(P_1)$ is isomorphic to $\bs{T_2}$ in dimensions $\leqslant 1$.

Finally, we construct a presentation $P_2 = (P_1, J_2, \partial_2, R_2, r_2)$ for $\bs{T_2}$. Define $J_2 = \{ i_2, h_2, v_2 \}$, let $J_2 \longrightarrow \mathbf{GSet}_{\bs{2}}(\bs{B_2}, \bs{1}^*)$ be the function
\begin{center}
\begin{tikzpicture}[node distance=2cm, auto]

\node (B) {$i_2$};

\node (P) [node distance=3cm, right of =B] {$\cdot$};
\node (Q) [right of =P] {$\cdot$};

\draw[->] (P) to node {} (Q);

\node (b) [node distance=0.5cm, right of=B] {};
\node (p) [node distance=0.5cm, left of=P] {};
\draw[|->, dashed] (b) to node {} (p);

\node (B') [node distance=1.2cm, below of=B] {$h_2$};

\node (P') [node distance=3cm, right of =B'] {$\cdot$};
\node (Q') [right of =P'] {$\cdot$};
\node (R') [right of =Q'] {$\cdot$};

\draw[->, bend left=40] (P') to node {} (Q');
\draw[->, bend right=40] (P') to node {} (Q');
\draw[->, bend left=40] (Q') to node {} (R');
\draw[->, bend right=40] (Q') to node {} (R');

\node () [node distance=1cm, right of =P'] {$\Downarrow$};
\node () [node distance=1cm, right of =Q'] {$\Downarrow$};

\node (b') [node distance=0.5cm, right of=B'] {};
\node (p') [node distance=0.5cm, left of=P'] {};
\draw[|->, dashed] (b') to node {} (p');

\node (B'') [node distance=1.5cm, below of=B'] {$v_2$};

\node (P'') [node distance=3cm, right of =B''] {$\cdot$};
\node (Q'') [right of =P''] {$\cdot$};

\draw[->, bend left=60] (P'') to node {} (Q'');
\draw[->] (P'') to node {} (Q'');
\draw[->, bend right=60] (P'') to node {} (Q'');

\node (x) [node distance=1cm, right of =P''] {};
\node () [node distance=0.3cm, above of =x] {$\Downarrow$};
\node () [node distance=0.35cm, below of =x] {$\Downarrow$};

\node (b'') [node distance=0.5cm, right of=B''] {};
\node (p'') [node distance=0.5cm, left of=P''] {};
\draw[|->, dashed] (b'') to node {} (p'');

\end{tikzpicture}
\end{center}
and define $\partial_2:J_2 \longrightarrow \mathbf{GSet}_{\bs{2}}(\bs{S_1}, UF_1(P_1))$ to be function given by,
\begin{center}
\begin{tikzpicture}[node distance=2cm, auto]

\node (B) {$i_2, v_2$};

\node (P) [node distance=3cm, right of =B] {$\text{id}_0$};
\node (Q) [node distance=2.2cm, right of =P] {$\text{id}_0$};

\draw[->, bend left=35] (P) to node {$\text{id}_1$} (Q);
\draw[->, bend right=35] (P) to node [swap] {$\text{id}_1$} (Q);

\node (b) [node distance=0.5cm, right of=B] {};
\node (p) [node distance=0.5cm, left of=P] {};
\draw[|->, dashed] (b) to node {$\partial_2$} (p);

\node (B') [node distance=2.2cm, below of=B] {$h_2$};

\node (P') [node distance=3cm, right of =B'] {$\text{id}_0$};
\node (Q') [node distance=2.2cm, right of =P'] {$\text{id}_0$};
\node (x) [node distance=3mm, right of =Q'] {};
\node () [node distance=1mm, below of=x] {$.$};

\draw[->, bend left=35] (P') to node {$h_1$} (Q');
\draw[->, bend right=35] (P') to node [swap] {$h_1$} (Q');

\node (b') [node distance=0.5cm, right of=B'] {};
\node (p') [node distance=0.5cm, left of=P'] {};
\draw[|->, dashed] (b') to node {$\partial_2$} (p');

\end{tikzpicture}
\end{center}
Then $\bs{J_2}$ is the 2-globular operad isomorphic to $\bs{T_2}$ in dimensions $\leqslant 1$, and whose 2-cells are the operadic composites of $i_2$, $h_2$ and $v_2$: the 2-cells of $\bs{J_2}$ are $\text{id}_2, \ i_2, \ h_2, \ v_2, \ v_2 \cdot (h_2, h_2 \circ (\text{id}_2, i_2)), \ h_2 \circ (i_2 \circ (i_1), h_2 ), \, ...$ and so on.
Define $R_2 = \{ u,v,p,q,a,b,n,m\}$ and let $r_2 : R_2 \longrightarrow \mathbf{GSet}_{\bs{2}}(\bs{S_2}, U(\bs{J_2}))$ be the function induced by the pair of functions
\begin{center}
\begin{tikzpicture}[node distance=3.5cm, auto]

\node (B) {$\mathbf{GSet}_{\bs{2}}(\bs{B_2}, U(\bs{J_2}))$};
\node (C) [left of=B] {$R_2$};

\draw[transform canvas={yshift=0.75ex},->] (C) to node {$e_2$} (B);
\draw[transform canvas={yshift=-0.25ex},->] (C) to node [swap] {$q_2$} (B);

\end{tikzpicture}
\end{center}
defined below.
\begin{enumerate}[i)]

\item
$e_2(u) = h_2 \circ \big(i_2 \circ (i_1), \text{id}_2\big)$\\
$q_2(u) = \text{id}_2$ 
\vspace{2mm}

\item
$e_2(v) = h_2 \circ \big(\text{id}_2, i_2 \circ (i_1)\big)$\\
$q_2(v) = \text{id}_2$ 
\vspace{2mm}

\item
$e_2(p) = v_2 \circ (i_2, \text{id}_2)$\\
$q_2(p) = \text{id}_2$
\vspace{2mm}

\item
$e_2(q) =v_2 \circ (\text{id}_2, i_2)$\\
$q_2(q) = \text{id}_2$
\vspace{2mm}

\item
$e_2(a) = h_2 \circ (\text{id}_2, h_2)$ \\
$q_2(a) = h_2 \circ (h_2, \text{id}_2)$
\vspace{2mm}

\item
$e_2(b) = v_2 \circ (\text{id}_2, v_2)$\\
$q_2(b) = v_2 \circ (v_2, \text{id}_2)$
\vspace{2mm}

\item
$e_2(n) = h_2 \circ (v_2, v_2)$\\
$q_2(n) = v_2 \circ (h_2, h_2)$
\vspace{2mm}

\item
$e_2(m) = h_2 \circ (i_2, i_2)$\\
$q_2(m) = i_2 \circ (h_1)$
\vspace{2mm}

\end{enumerate}

The 2-globular operad $F_2(P_2)$ is isomorphic to $F_1(P_1)$ in dimensions $<2$ and has 2-cells consisting of the operadic composites of 2-cells $i_2$, $h_2$ and $v_2$ whose images under the underlying collection map are as follows,

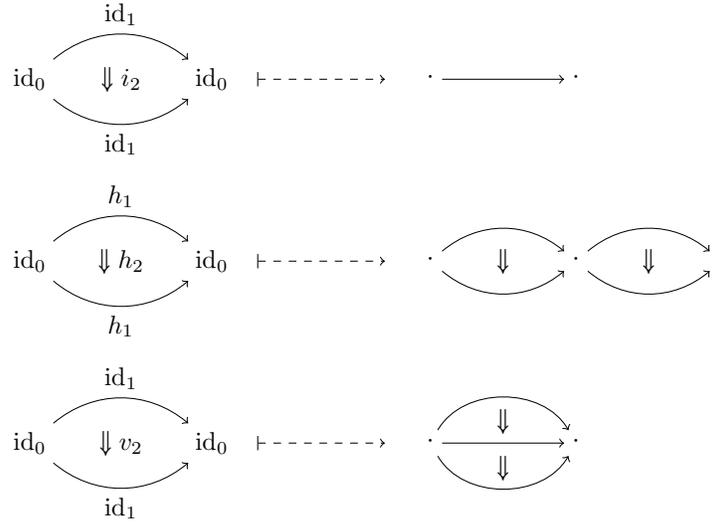
\begin{center}
\resizebox{3.9in}{!}{
\begin{tikzpicture}[node distance=2cm, auto]

\node (A) {$\text{id}_0$};
\node (B) [node distance=2.5cm, right of =A] {$\text{id}_0$};

\draw[->, bend left=40] (A) to node {$\text{id}_1$} (B);
\draw[->, bend right=40] (A) to node [swap] {$\text{id}_1$} (B);

\node () [node distance=1.25cm, right of =A] {$\Downarrow i_2$};

\node (P) [node distance=3cm, right of =B] {$\cdot$};
\node (Q) [right of =P] {$\cdot$};

\draw[->] (P) to node {} (Q);

\node (b) [node distance=0.5cm, right of=B] {};
\node (p) [node distance=0.5cm, left of=P] {};
\draw[|->, dashed] (b) to node {} (p);

\node (A') [node distance=2.5cm, below of=A] {$\text{id}_0$};
\node (B') [node distance=2.5cm, right of =A'] {$\text{id}_0$};

\draw[->, bend left=40] (A') to node {$h_1$} (B');
\draw[->, bend right=40] (A') to node [swap] {$h_1$} (B');

\node () [node distance=1.25cm, right of =A'] {$\Downarrow h_2$};

\node (P') [node distance=3cm, right of =B'] {$\cdot$};
\node (Q') [right of =P'] {$\cdot$};
\node (R') [right of =Q'] {$\cdot$};

\draw[->, bend left=40] (P') to node {} (Q');
\draw[->, bend right=40] (P') to node {} (Q');
\draw[->, bend left=40] (Q') to node {} (R');
\draw[->, bend right=40] (Q') to node {} (R');

\node () [node distance=1cm, right of =P'] {$\Downarrow$};
\node () [node distance=1cm, right of =Q'] {$\Downarrow$};

\node (b') [node distance=0.5cm, right of=B'] {};
\node (p') [node distance=0.5cm, left of=P'] {};
\draw[|->, dashed] (b') to node {} (p');

\node (A'') [node distance=2.5cm, below of=A'] {$\text{id}_0$};
\node (B'') [node distance=2.5cm, right of =A''] {$\text{id}_0$};

\draw[->, bend left=40] (A'') to node {$\text{id}_1$} (B'');
\draw[->, bend right=40] (A'') to node [swap] {$\text{id}_1$} (B'');

\node () [node distance=1.25cm, right of =A''] {$\Downarrow v_2$};

\node (P'') [node distance=3cm, right of =B''] {$\cdot$};
\node (Q'') [right of =P''] {$\cdot$};

\draw[->, bend left=60] (P'') to node {} (Q'');
\draw[->] (P'') to node {} (Q'');
\draw[->, bend right=60] (P'') to node {} (Q'');

\node (x) [node distance=1cm, right of =P''] {};
\node () [node distance=0.3cm, above of =x] {$\Downarrow$};
\node () [node distance=0.35cm, below of =x] {$\Downarrow$};

\node (b'') [node distance=0.5cm, right of=B''] {};
\node (p'') [node distance=0.5cm, left of=P''] {};
\draw[|->, dashed] (b'') to node {} (p'');

\end{tikzpicture}
}
\end{center}

\noindent subject to the equalities listed below.
\begin{enumerate}[i)]
\item $h_2 \circ \big(i_2 \circ (i_1), \text{id}_2\big) = \text{id}_2$
\item $h_2 \circ \big(\text{id}_2, i_2 \circ (i_1)\big) = \text{id}_2$
\item $v_2 \circ (i_2, \text{id}_2) = \text{id}_2$
\item $v_2 \circ (\text{id}_2, i_2) = \text{id}_2$
\item $h_2 \circ (\text{id}_2, h_2) = h_2 \circ (h_2, \text{id}_2)$
\item $v_2 \circ (\text{id}_2, v_2) = v_2 \circ (v_2, \text{id}_2)$
\item $h_2 \circ (v_2,v_2) = v_2 \circ (h_2,h_2)$
\item $h_2 \circ (i_2,i_2) = i_2 \circ (h_1)$
\end{enumerate}

\noindent These equalities mean that $F_2(P_2)$ contains exactly one 2-cell for each 2-cell in $\bs{1}^*$, so $F_2(P_2) \cong \bs{T_2}$.

\end{example} 

An algebra for $F_2(P_2)$ on a 2-globular set $\bs{A}$ is precisely a strict 2-category with underlying 2-globular set $\bs{A}$. The (1-)category structure imposed on the 0 and 1-cells of $\bs{A}$ is the same as in Example \ref{pexample-presentations}.
The 2-cells $i_2$, $h_2$ and $v_2$ of $F_2(P_2)$ provide 2-cell identites, horizontal composition of 2-cells and binary composition of 2-cells in $\bs{A}$, respectively. Equalities i) - iv) above yield the four unit axioms, equalities v) and vi) yield the two associativity axioms, equality vii) yields the interchange law, and equality viii) yields the axiom stating that the horizontal composite of two identity cells is another identity cell. 

From this point onwards, we will define $n$-globular operads by describing the generators and relations of a presentation, as in Example \ref{pexample-presentations}. Additionally, since we are only interested in those $n$-globular operads which are equivalent to some theory of $n$-category, and there are no non-trivial operations on 0-cells in higher categories, our globular operads will always have a single 0-cell - the identity $\text{id}_0$. This means that the presentations will contain no 0-cell generators, and therefore no 0-cell relations. Since there is no ambiguity, we will leave the single identity 0-cell of these operads unlabelled. For example, we would represent a 1-cell $x$ of an $n$-globular operad for some theory of $n$-category by 
\vspace{1mm}
\begin{center}
\begin{tikzpicture}[node distance=2.4cm, auto]

\node (A) {$\cdot$};
\node (B) [right of =A] {$\cdot$};

\draw[->] (A) to node {$x$} (B);

\node (x) [right of=B] {rather than};

\node (A') [right of=x] {$\text{id}_0$};
\node (B') [right of =A'] {$\text{id}_0$.};

\draw[->] (A') to node {$x$} (B');

\end{tikzpicture}
\end{center}
The next example defines the 2-globular operad for strict 2-categories using the presentation for $\bs{T_2}$ given in Example \ref{biased presentation for T2} above.

\begin{example}\label{pres for T_2}
The 2-globular operad $\bs{T_2}$ for strict 2-categories is the 2-globular operad with

\begin{list}{$\bullet$}{}

\item a single 0-cell, the identity $\text{id}_0$; 

\item 1-cells consisting of the operadic composites of 1-cells $i_1$ and $h_1$ whose images under the underlying collection map are as follows

\begin{center}
\begin{tikzpicture}[node distance=2cm, auto]

\node (A) {$\cdot$};
\node (B) [right of =A] {$\cdot$};

\draw[->] (A) to node {$i_1$} (B);

\node (P) [node distance=3cm, right of =B] {$\cdot$};

\node (b) [node distance=0.5cm, right of=B] {};
\node (p) [node distance=0.5cm, left of=P] {};
\draw[|->, dashed] (b) to node {} (p);

\node (A') [node distance=9mm, below of=A] {$\cdot$};
\node (B') [right of =A'] {$\cdot$};

\draw[->] (A') to node {$h_1$} (B');

\node (P') [node distance=3cm, right of =B'] {$\cdot$};
\node (Q') [right of =P'] {$\cdot$};
\node (R') [right of =Q'] {$\cdot$};

\draw[->] (P') to node {} (Q');
\draw[->] (Q') to node {} (R');

\node (b') [node distance=0.5cm, right of=B'] {};
\node (p') [node distance=0.5cm, left of=P'] {};
\draw[|->, dashed] (b') to node {} (p');

\end{tikzpicture}
\end{center}
\vspace{1mm}
subject to the following equalities
\begin{enumerate}[i)]
\item $h_1 \circ (i_1, \text{id}_1) = \text{id}_1$
\item $h_1 \circ (\text{id}_1, i_1) = \text{id}_1$ 
\item $h_1 \circ (\text{id}_1, h_1) = h_1 \circ (h_1, \text{id}_1)$;
\end{enumerate}
and

\item 2-cells consisting of the operadic composites of 2-cells $i_2$, $h_2$ and $v_2$ whose images under the underlying collection map are as follows

\begin{center}
\begin{tikzpicture}[node distance=2cm, auto]

\node (A) {$\cdot$};
\node (B) [node distance=2.5cm, right of =A] {$\cdot$};

\draw[->, bend left=40] (A) to node {$\text{id}_1$} (B);
\draw[->, bend right=40] (A) to node [swap] {$\text{id}_1$} (B);

\node () [node distance=1.25cm, right of =A] {$\Downarrow i_2$};

\node (P) [node distance=3cm, right of =B] {$\cdot$};
\node (Q) [right of =P] {$\cdot$};

\draw[->] (P) to node {} (Q);

\node (b) [node distance=0.5cm, right of=B] {};
\node (p) [node distance=0.5cm, left of=P] {};
\draw[|->, dashed] (b) to node {} (p);

\node (A') [node distance=2.4cm, below of=A] {$\cdot$};
\node (B') [node distance=2.5cm, right of =A'] {$\cdot$};

\draw[->, bend left=40] (A') to node {$h_1$} (B');
\draw[->, bend right=40] (A') to node [swap] {$h_1$} (B');

\node () [node distance=1.25cm, right of =A'] {$\Downarrow h_2$};

\node (P') [node distance=3cm, right of =B'] {$\cdot$};
\node (Q') [right of =P'] {$\cdot$};
\node (R') [right of =Q'] {$\cdot$};

\draw[->, bend left=40] (P') to node {} (Q');
\draw[->, bend right=40] (P') to node {} (Q');
\draw[->, bend left=40] (Q') to node {} (R');
\draw[->, bend right=40] (Q') to node {} (R');

\node () [node distance=1cm, right of =P'] {$\Downarrow$};
\node () [node distance=1cm, right of =Q'] {$\Downarrow$};

\node (b') [node distance=0.5cm, right of=B'] {};
\node (p') [node distance=0.5cm, left of=P'] {};
\draw[|->, dashed] (b') to node {} (p');

\node (A'') [node distance=2.4cm, below of=A'] {$\cdot$};
\node (B'') [node distance=2.5cm, right of =A''] {$\cdot$};

\draw[->, bend left=40] (A'') to node {$\text{id}_1$} (B'');
\draw[->, bend right=40] (A'') to node [swap] {$\text{id}_1$} (B'');

\node () [node distance=1.25cm, right of =A''] {$\Downarrow v_2$};

\node (P'') [node distance=3cm, right of =B''] {$\cdot$};
\node (Q'') [right of =P''] {$\cdot$};

\draw[->, bend left=60] (P'') to node {} (Q'');
\draw[->] (P'') to node {} (Q'');
\draw[->, bend right=60] (P'') to node {} (Q'');

\node (x) [node distance=1cm, right of =P''] {};
\node () [node distance=0.3cm, above of =x] {$\Downarrow$};
\node () [node distance=0.35cm, below of =x] {$\Downarrow$};

\node (b'') [node distance=0.5cm, right of=B''] {};
\node (p'') [node distance=0.5cm, left of=P''] {};
\draw[|->, dashed] (b'') to node {} (p'');

\end{tikzpicture}
\end{center}

\noindent subject to the equations below.
\begin{enumerate}[i)]
\item $h_2 \circ \big(i_2 \circ (i_1), \text{id}_2\big) = \text{id}_2$
\item $h_2 \circ \big(\text{id}_2, i_2 \circ (i_1)\big) = \text{id}_2$
\item $v_2 \circ (i_2, \text{id}_2) = \text{id}_2$
\item $v_2 \circ (\text{id}_2, i_2) = \text{id}_2$
\item $h_2 \circ (\text{id}_2, h_2) = h_2 \circ (h_2, \text{id}_2)$
\item $v_2 \circ (\text{id}_2, v_2) = v_2 \circ (v_2, \text{id}_2)$
\item $h_2 \circ (v_2,v_2) = v_2 \circ (h_2,h_2)$
\item $h_2 \circ (i_2,i_2) = i_2 \circ (h_1)$
\end{enumerate}
\end{list}

In this example $i_1$ and $h_1$ are the 1-cell generators and the equations they satisfy correspond to the 1-cell relations. Similarly, $i_2$, $h_2$ and $v_2$ are the 2-cell generators and the equations they satisfy correspond to the 2-cell relations.
\end{example}

If an $n$-globular operad $\bs{G_n}$ is contractible then there is no need to specify any $n$-cell generators or relations in a presentation for $\bs{G_n}$; by Lemma \ref{contractible n-cells} a contractible $n$-globular operad is completely determined by the $k$-cells for all $k<n$.
In light of this, we have the following equivalent definition of $\bs{T_2}$.

\begin{example}\label{example of a contractible presentation}
The 2-globular operad $\bs{T_2}$ is the \textit{contractible} 2-globular operad whose 0 and 1-cells are as in Example \ref{pres for T_2}.
\end{example}

\section{The globular operads for weak unbiased higher categories}\label{weak unbiased section}

Many classical definitions of higher category have two types of composition - binary and nullary. Binary composition takes two composable cells and produces a single composite cell, while nullary composition picks out identities. All other compositions are derived from these two generating types. For instance, given three composable 1-cells

\begin{center}
\begin{tikzpicture}[node distance=2cm, auto]

\node (A) {$A$};
\node (B) [right of=A] {$B$};
\node (C) [right of=B] {$C$};
\node (D) [right of=C] {$D$};

\draw[->] (A) to node {$a$} (B);
\draw[->] (B) to node {$b$} (C);
\draw[->] (C) to node {$c$} (D);

\end{tikzpicture}
\end{center}

\noindent in a higher category we have composites like $a(bc)$, $(ab)c$, $((a1_B)b)c$, and so on, but there is no specified ternary composite $abc$. Likewise, for a 2-pasting diagram 

\begin{center}
\begin{tikzpicture}[node distance=2.5cm, auto]

\node (A) {$A$};
\node (B) [right of=A] {$B$};
\node (C) [right of=B] {$C$};

\draw[->, bend left=60] (A) to node {$a$} (B);
\draw[->, bend left=60] (B) to node {$b$} (C);

\draw[->] (A) to node {} (B);
\draw[->] (B) to node {} (C);

\draw[->, bend right=60] (A) to node [swap] {$a''$} (B);
\draw[->, bend right=60] (B) to node [swap] {$b''$} (C);

\node (W) [node distance=1.25cm, right of=A] {};
\node (U) [node distance=1.25cm, right of=B] {};

\node (W') [node distance=0.35cm, above of=W] {$\Downarrow  \alpha$};
\node (U') [node distance=0.35cm, above of=U] {$\Downarrow  \beta$};

\node (W'') [node distance=0.4cm, below of=W] {$\Downarrow  \alpha'$};
\node (U'') [node distance=0.4cm, below of=U] {$\Downarrow  \beta'$};

\end{tikzpicture}
\end{center} 

\noindent we have composites like $(\alpha \cdot \alpha') \ast (\beta \cdot \beta')$, $(\alpha \ast \beta) \cdot (\alpha' \ast \beta')$ and $(\alpha \cdot \alpha') \ast (\beta \cdot (1_b' \cdot \beta') )$, but no specified operation directly composing a pasting diagram like this into a 2-cell $\alpha \alpha' \beta \beta'$.
We say that these definitions of higher category are \textit{biased} towards binary and nullary composition. An \textit{unbiased} higher category is one which takes into account all types of composition, rather than just the binary and nullary composition operations present in traditional definitions of higher category. 

\begin{definition}
An \textit{unbiased higher category} is a higher category for which, given any $k$-pasting diagram together with a way of composing its $(k-1)$-dimensional boundary, there is a specified operation composing it directly into a single $k$-cell. 
\end{definition}

Leinster defines weak $\omega$-categories as algebras for the initial globular operad with contraction (Notation \ref{the category CG-Op}) and weak $n$-categories as algebras for the initial $n$-globular operad with contraction (Notation \ref{the category CG-Opn}) \cite[Chapter 9]{TL}. The algebras for these operads are weak unbiased higher categories. In this section we provide explicit definitions of these operads by constructing presentations for them expressed in the style of Example \ref{example of a contractible presentation}; that is, omitting the unnecessary $n$-dimensional data. We show that for $n = 0, 1$ and $2$, the algebras for the initial $n$-globular operad with contraction are precisely sets, categories and unbiased bicategories, respectively. See \cite[Definition 1.2.1]{TL3} for a hands-on definition of an unbiased bicategory.

\begin{notation} We will denote by $i:\bs{I_n} \longrightarrow \bs{1}^*$ the underlying collection map of the $n$-globular operad $\bs{I_n}$ for weak unbiased $n$-categories
\end{notation}

\begin{definition}\label{presentation for I_n} The $n$-globular operad $\bs{I_n}$ for weak unbiased $n$-categories is the \textit{contractible} $n$-globular operad with

\begin{list}{•}{}

\item a single 0-cell, the identity $\text{id}_0$; and

\item for $0<k<n$, $k$-cells consisting of the operadic composites of a specified $k$-cell $\chi:x \longrightarrow x'$ satisfying $i(\chi)=\tau$ for each triple $(x,x',\tau)$ where $(x,x')$ is a parallel pair of $(k{-}1)$-cells of $\bs{I_n}$ satisfying $i(x)=i(x')={\partial}\tau$.

\end{list}
\end{definition}

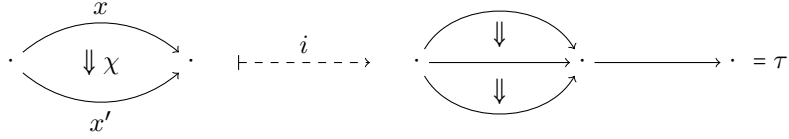
\begin{center}
\begin{tikzpicture}[node distance=2.4cm, auto]

\node (B) {$\cdot$};
\node (C) [node distance=2.2cm, right of=B] {$\cdot$};
\node (D) [node distance=2cm, right of=C] {$\cdot$};
\node (D') [node distance=0.5cm, right of=D] {$= \tau$}; 

\draw[->] (C) to node {} (D);

\draw[->, bend left=60] (B) to node {} (C);
\draw[->] (B) to node {} (C);
\draw[->, bend right=60] (B) to node [swap] {} (C);

\node (W) [node distance=1.1cm, right of=B] {};
\node (YA) [node distance=0.37cm, above of=W] {$\Downarrow $};
\node (YC) [node distance=0.38cm, below of=W] {$\Downarrow $};

\node (X) [node distance=3cm, left of=B] {$\cdot$};
\node (Y) [left of=X] {$\cdot$};
\draw[->, bend left=40] (Y) to node {$x$} (X);
\draw[->, bend right=40] (Y) to node [swap] {$x'$} (X);
\node (Z) [node distance=1.2cm, right of=Y] {$\Downarrow \chi$};

\node (P) [node distance=0.5cm, right of=X] {};
\node (Q) [node distance=0.5cm, left of=B] {};
\draw[|->, dashed] (P) to node {$i$} (Q);

\end{tikzpicture}
\end{center}

Note that our presentation for $\bs{I_n}$ has no $k$-cell relations, only $k$-cell generators. This is what we would expect, as the $k$-cell relations in a presentation for the ($n$-)globular operad for some theory of higher category should correspond to the axioms on $k$-cell composition in those higher categories. In order to satisfy the Coherence Theorem \ref{the coherence theorem}, the axioms for a fully weak higher category should impose no restraints on $k$-cell composition for any $k<n$. We also observe that there is a canonical contraction on the underlying collection of $\bs{I_n}$, defined by taking the $k$-cell generators to be the contraction cells.

\begin{lemma} When equipped with its canonical contraction, the $n$-globular operad $\bs{I_n}$ for weak unbiased $n$-categories is initial in the category $\mathbf{C \mhyphen GOp}_{\bs{n}}$ of $n$-globular operads with contraction. 
\end{lemma}

\begin{proof}
Morphisms of $n$-globular operads preserve identities and composition, and the morphisms in $\mathbf{C \mhyphen GOp}_{\bs{n}}$ also preserve contractions.
It follows immediately from the definition above that when equipped with its canonical contraction, every $k$-cell of $\bs{I_n}$ for $k<n$ is a unique composite of identity cells and contraction cells. Moreover, by Lemma \ref{contractible n-morphisms}  every morphism of contractible $n$-globular operads is determined by its value on $k$-cells for all $k<n$. It follows that for any $n$-globular operad with contraction $\bs{G_n}$ there exists a unique morphism $\bs{I_n} \longrightarrow \bs{G_n}$ of $n$-globular operads with contraction.
\end{proof}

An algebra $\theta : \bs{A_{I_n}} \longrightarrow \bs{A}$ for $\bs{I_n}$ is a weak unbiased $n$-category with underlying $n$-globular set $\bs{A}$. For $k<n$, the generating $k$-cell $\chi:x \longrightarrow x'$ of $\bs{I_n}$ corresponding to the triple $(x,x',\tau)$ provides a specified operation directly composing $k$-pasting diagrams of shape $\tau$ in $\bs{A}$ into single $k$-cells of $\bs{A}$ in a way that is consistent with the operations composing their boundaries provided by $x$ and $x'$. The $k$-cells of $\bs{I_n}$ obtained by operadic composition of the generating cells provide composition operations on pasting diagrams in $\bs{A}$ derived from these basic ones. Since there are no $k$-cell relations in the presentation for $\bs{I_n}$, this compositions of $k$-pasting diagrams in $\bs{A}$ satisfies no axioms in dimensions $k<n$. In dimension $n$, the contractibility of $\bs{I_n}$ means that there exists a \textit{unique} $n$-cell $\Lambda:\chi \longrightarrow \chi'$ of $\bs{I_n}$ satisfying $i(\Lambda)=\pi$ for each triple $(\chi,\chi',\pi)$ where $\chi$ and $\chi'$ are parallel $(n-1)$-cells of $\bs{I_n}$ satisfying $i(\chi)=i(\chi')=\partial(\pi)$. As a result, there is a \text{unique} operation composing $n$-pasting diagrams of shape $\pi$ in $\bs{A}$ with respect to any composition of the $(n-1)$-dimensional boundary, so the Coherence Theorem \ref{the coherence theorem} is satisfied.

\begin{definition} 
A \textit{weak unbiased $n$-category} is an algebra for $\bs{I_n}$.
\end{definition}

Recall that a morphism $F:(\bs{A}, \theta) \longrightarrow (\bs{B}, \sigma)$ of algebras for an $n$-globular operad $\bs{G_n}$ is a morphism $F:\bs{A} \longrightarrow \bs{B}$ of the underlying $n$-globular sets \textit{strictly} preserving composition of pasting diagrams; see section \ref{algebras section}. 

\begin{definition} 
A \textit{strict $n$-functor} of weak unbiased $n$-categories is a morphism of algebras for $\bs{I_n}$.
\end{definition}

\begin{notation} We denote by $\mathbf{WkU}\  n \mhyphen \mathbf{Cat}_{str}$ the category of algebras for $\bs{I_n}$.
\end{notation}

To understand these definitions it is instructive to unpack them for low values of $n$. The first two cases, $n=0$ and $n=1$, are straightforward. The $0$-globular operad $\bs{I_0}$ consists of a single 0-cell, the identity $\text{id}_0$, so $\bs{I_0}$ is isomorphic to $\bs{T_0}$. The monad induced on $\mathbf{GSet_0}$ by $\bs{I_0}$ is isomorphic to the identity monad, and so an algebra for $\bs{I_0}$ is just a 0-globular set, or equivalently, a set.
Similarly, the $1$-globular operad $\bs{I_1}$ is isomorphic to $\bs{T_1}$, so the monad induced on $\mathbf{GSet_1}$ by $\bs{I_1}$ is isomorphic to the free category monad $(-)^*$, and an algebra for $\bs{I_1}$ is just an ordinary category. 
The lowest dimensional non-trivial case is the 2-globular operad $\bs{I_2}$ for weak unbiased 2-categories, or unbiased bicatgeories. 

\begin{example}\label{I_2}
The 2-globular operad $\bs{I_2}$ for unbiased bicategories is the \textit{contractible} 2-globular operad with

\begin{list}{•}{}

\item a single 0-cell, the identity $\text{id}_0$; and

\item 1-cells consisting of the operadic composites of 1-cells $c^0, c^1, c^2, c^3, ...$ whose images under the underlying collection map are as follows.

\begin{center}
\begin{tikzpicture}[node distance=2cm, auto]

\node (A) {$\cdot$};
\node (B) [right of =A] {$\cdot$};

\draw[->] (A) to node {$c^0$} (B);

\node (P) [node distance=3cm, right of =B] {$\cdot$};

\node (b) [node distance=0.5cm, right of=B] {};
\node (p) [node distance=0.5cm, left of=P] {};
\draw[|->, dashed] (b) to node {$i$} (p);

\node (A') [node distance=9mm, below of=A] {$\cdot$};
\node (B') [right of =A'] {$\cdot$};

\draw[->] (A') to node {$c^1$} (B');

\node (P') [node distance=3cm, right of =B'] {$\cdot$};
\node (Q') [right of =P'] {$\cdot$};

\draw[->] (P') to node {} (Q');

\node (b') [node distance=0.5cm, right of=B'] {};
\node (p') [node distance=0.5cm, left of=P'] {};
\draw[|->, dashed] (b') to node {$i$} (p');

\node (A'') [node distance=9mm, below of=A'] {$\cdot$};
\node (B'') [right of =A''] {$\cdot$};

\draw[->] (A'') to node {$c^2$} (B'');

\node (P'') [node distance=3cm, right of =B''] {$\cdot$};
\node (Q'') [right of =P''] {$\cdot$};
\node (R'') [right of =Q''] {$\cdot$};

\draw[->] (P'') to node {} (Q'');
\draw[->] (Q'') to node {} (R'');

\node (b'') [node distance=0.5cm, right of=B''] {};
\node (p'') [node distance=0.5cm, left of=P''] {};
\draw[|->, dashed] (b'') to node {$i$} (p'');

\node (A''') [node distance=9mm, below of=A''] {$\cdot$};
\node (B''') [right of =A'''] {$\cdot$};

\draw[->] (A''') to node {$c^3$} (B''');

\node (P''') [node distance=3cm, right of =B'''] {$\cdot$};
\node (Q''') [right of =P'''] {$\cdot$};
\node (R''') [right of =Q'''] {$\cdot$};
\node (S''') [right of =R'''] {$\cdot$};

\draw[->] (P''') to node {} (Q''');
\draw[->] (Q''') to node {} (R''');
\draw[->] (R''') to node {} (S''');

\node (b''') [node distance=0.5cm, right of=B'''] {};
\node (p''') [node distance=0.5cm, left of=P'''] {};
\draw[|->, dashed] (b''') to node {$i$} (p''');

\node (X) [node distance=6mm, below of=b'''] {};
\node () [node distance=0.9cm, right of=X] {$\vdots$};

\end{tikzpicture}
\end{center}

\end{list}
\end{example}

An algebra $\theta:\bs{A_{I_2}} \longrightarrow \bs{A}$ for $\bs{I_2}$ is precisely an unbiased bicategory with underlying 2-globular set $\bs{A}$. For each $n \in \mathbb{N}$, $c^n$ provides an operation directly composing 1-pasting diagrams in $\bs{A}$ made up of $n$ 1-cells into single 1-cells, with $c^0$ providing 1-cell identites. For example, given 1-pasting diagrams

\begin{center}
\begin{tikzpicture}[node distance=2cm, auto]

\node (A) {$A$};
\node (B) [right of=A] {$A$};
\node (C) [right of=B] {$B$};
\node (D) [right of=C] {$C$};
\node (E) [right of=D] {$D$};

\draw[->] (B) to node {$a$} (C);
\draw[->] (C) to node {$b$} (D);
\draw[->] (D) to node {$c$} (E);

\end{tikzpicture}
\end{center}

\noindent in $\bs{A}$ we write $c^0(A) = 1_A$ and $c^3(a,b,c) = (abc)$,

\begin{center}
\begin{tikzpicture}[node distance=2.2cm, auto]

\node (A) {$A$};
\node (B) [right of=A] {$A$};
\node (C) [right of=B] {$A$};
\node (D) [node distance=2.5cm, right of=C] {$D$.};

\draw[->] (A) to node {$1_A$} (B);

\draw[->] (C) to node {$(abc)$} (D);

\end{tikzpicture}
\end{center}
Note that $c^1$ provides an operation composing simple 1-pasting diagrams in $\bs{A}$ (Definition \ref{trivlal pd}) that is \textit{distinct} from the `do nothing' operation on simple 1-pasting diagrams provided by the identity 1-cell $\text{id}_1$. The 1-cells of $\bs{I_2}$ obtained by operadic composition of the $c^n$s provide operations composing 1-pasting diagrams in $\bs{A}$ that are derived from these basic operations. For instance, given the 1-pasting diagram $(a,b,c)$ above we have $\big(c^3 \circ (c^2, \text{id}_1, c^0)\big)(a,b,c) = \big( (ab)c1_D \big)$ and $\big(c^3 \circ (c^2, c^1, c^0)\big)(a,b,c) = \big( (ab)(c)1_D \big)$. The contractibility of $\bs{I_2}$ means that for any 2-pasting diagram in $\bs{A}$ together with a composition of the 1-dimensional boundary, there exists a unique composite 2-cell that is consistent with this boundary composition.

\begin{notation} We denote by $\mathbf{UBicat}_{str}$ the category $\bs{I_2} \mhyphen \mathbf{Alg}$ of algebras for $\bs{I_2}$.
\end{notation}

There is a presentation for the globular operad $\bs{I}$ for weak unbiased $\omega$-categories analogous to the one given for $\bs{I_n}$ in Definition \ref{presentation for I_n}. Like the $n$-dimensional case, it follows immediately from the definition that $\bs{I}$ is contractible, and that the lemma below holds.

\begin{lemma} When equipped with its canonical contraction the globular operad $\bs{I}$ for weak unbiased $\omega$-categories is initial in the category $\mathbf{C \mhyphen GOp}$ of globular operads with contraction.
\end{lemma}

\begin{definition} 
A \textit{weak unbiased $\omega$-category} is an algebra for $\bs{I}$.
\end{definition}


\begin{notation} We denote by $\mathbf{WkU} \ \omega \mhyphen \mathbf{Cat}_{str}$ the category of algebras for $\bs{I}$.
\end{notation}

\section{Globular operads for higher categories}\label{glob ops for higher cats section}

Recall that the globular operad $\bs{T}$ for strict $\omega$-categories is given by equipping the terminal globular collection $1:\bs{1}^* \longrightarrow \bs{1}^*$ with its unique operad structure; see Example \ref{GlobOp T for strict w cats}. This means that $\bs{T}$ contains exactly one $n$-cell for each $n$-cell of $\bs{1}^*$, and therefore provides exactly one way to compose an $n$-pasting diagram of any given shape. Meanwhile, the globular operad $\bs{I}$ for weak unbiased $\omega$-categories has a unique $n$-cell for every possible way to compose an $n$-pasting diagram in an $\omega$-category. Intuitively, any globular operad that lies somewhere between $\bs{I}$ and $\bs{T}$ should define a sensible theory of $\omega$-category. A similar statement is true for $n$-globular operads and $n$-categories. In this section we construct a category $\mathbf{Cat \mhyphen GOp}$ of globular operads for $\omega$-categories, and a similar category $\mathbf{Cat \mhyphen GOp}_{\bs{n}}$ of $n$-globular operads for $n$-categories.

\begin{definition} We say that a morphism $f:\bs{G} \longrightarrow \bs{H}$ of globular sets is 
\begin{enumerate}[i)]
\item \textit{surjective on 0-cells} if the function $f_0:G_0 \longrightarrow H_0$ is a surjection;
\item \textit{$k$-full} (respectively, \textit{$k$-faithful}) for $k > 0$ if the restriction $G_k(x,x') \longrightarrow H_k(f_{k-1}(x),f_{k-1}(x'))$ of $f_k:G_k \longrightarrow H_k$ is a surjection (respectively, an injection) for each parallel pair $(x,x')$ of $(k {-} 1)$-cells of $\bs{G}$; and 
\item \textit{full} (respectively \textit{faithful}) if it is $k$-full (respectively, $k$-faithful) for all $k > 0$.
\end{enumerate}
\end{definition}

\begin{lemma}\label{contractible0Full} A globular operad $\bs{G}$ with underlying collection $g:\bs{G} \longrightarrow \bs{1}^*$ is contractible if and only if $g$ is full.
\end{lemma}

\begin{lemma}\label{full+surj=split epi}
If a morphism of globular sets is surjective on 0-cells and full, then it is a split epimorphism. 
\end{lemma}

\begin{proof} Let $f:\bs{G} \longrightarrow \bs{H}$ be surjective on 0-cells and full. 
We show by induction on $n$ that there exists a morphism $f':\bs{H} \longrightarrow \bs{G}$ of globular sets such that $f_nf'_n=1_{H_n}$ for all $n \in \mathbb{N}$. For the base case, $f_0: G_0 \longrightarrow H_0$ is a surjection of sets by assumption, so we can choose an injective function $f'_0: H_0 \longrightarrow G_0$ such that $f_0f'_0 = 1_{H_0}$. Next, assume that for all $k$, $0 \leqslant k \leqslant n$, there exists a function $f'_k: H_k \longrightarrow G_k$ for which $f_kf'_k = 1_{H_k}$ and such that the sources and targets are preserved, i.e., the diagrams below commute.

\begin{center}
\begin{tikzpicture}[node distance=2cm, auto]

\node (X) {$G_k$};
\node (Y) [below of=X] {$G_{k-1}$};
\node (U) [node distance=2.5cm, left of=X] {$H_k$};
\node (V) [below of=U] {$H_{k-1}$};

\draw[->] (X) to node{$s$} (Y);
\draw[->] (U) to node [swap] {$s$} (V);
\draw[->] (U) to node {$f'_k$} (X);
\draw[->] (V) to node [swap] {$f'_{k-1}$} (Y);

\node (X') [node distance=5cm, right of =X] {$G_k$};
\node (Y') [below of=X'] {$G_{k-1}$.};
\node (U') [node distance=2.5cm, left of=X'] {$H_k$};
\node (V') [below of=U'] {$H_{k-1}$};

\draw[->] (X') to node{$t$} (Y');
\draw[->] (U') to node [swap] {$t$} (V');
\draw[->] (U') to node {$f'_k$} (X');
\draw[->] (V') to node [swap] {$f'_{k-1}$} (Y');

\end{tikzpicture}
\end{center}

\noindent Here we define $G_{-1} = H_{-1}$ to be the terminal set consisting of a single element. We need to construct a function $f'_{n+1}:H_{n+1} \longrightarrow G_{n+1}$ satisfying $f_{n+1}f'_{n+1} = 1_{H_{n+1}}$ and preserving the sources and targets. Let $\psi:y \longrightarrow y'$ be an $(n{+}1)$-cell of $\bs{H}$. Since $f$ is full, the restriction 

\begin{center}
\begin{tikzpicture}[node distance=6.5cm, auto]
\node (X) {$G_{n+1}(f'_n(y), \, f'_n(y'))$};
\node (Y) [right of=X] {$H_{n+1}(f_nf'_n(y), \, f_nf'_n(y')) = H_{n+1}(y,y')$};
\draw[->] (X) to node {$f_{n+1}$} (Y);
\end{tikzpicture}
\end{center}
\noindent of $f_{n+1}$ is a surjection, so we can choose an $(n{+}1)$-cell $\chi:f'_n(y) \longrightarrow f_n'(y')$ of $\bs{G}$ such that $f_{n+1}(\chi) = \psi$. Define $f'_{n+1}(\psi)=\chi$, then $f_{n+1}f'_{n+1}(\psi)=\psi$ and the source and target diagrams for $f_{n+1}$ commute at $\psi$ by construction.
\end{proof}

\begin{definition}\label{def of A_f} Let $f:\bs{G} \longrightarrow \bs{H}$ be a morphism of globular collections and let $\bs{A}$ be a globular set. The morphism $\bs{A}_f:\bs{A}_{\bs{G}} \longrightarrow \bs{A}_{\bs{H}}$ of globular sets is the unique morphism making the diagram below commute.
\begin{center}
\begin{tikzpicture}[node distance=1.4cm, auto]

\node (X) {$\bs{A}_{\bs{H}}$};
\node (A) [left of=X, below of=X] {$\bs{A}^*$};
\node (B) [right of=X, below of=X] {$\bs{H}$};
\node (C) [node distance=2.8cm, below of=X] {$\bs{1}^*$};

\draw[->] (X) to node [swap] {} (A);
\draw[->] (X) to node {} (B);
\draw[->] (A) to node [swap] {$!^*$} (C);
\draw[->] (B) to node {$h$} (C);

\node (Y) [node distance=2.2cm, above of=X] {$\bs{A}_{\bs{G}}$};
\node (Q) [right of=Y, below of=Y] {$\bs{G}$};

\draw[->] (Y) to node {} (Q);
\draw[->, bend right=30] (Y) to node [swap] {} (A);
\draw[->] (Q) to node {$f$} (B);

\draw[->, dashed] (Y) to node [swap] {$\bs{A}_f$} (X); 

\end{tikzpicture}
\end{center}
\end{definition}

\begin{lemma}\label{induced functor between alg cats} A morphism $f:\bs{G} \longrightarrow \bs{H}$ of globular operads induces a faithful functor $f:\bs{H} \mhyphen \mathbf{Alg} \longrightarrow \bs{G} \mhyphen \mathbf{Alg}$ sending an algebra $(\bs{A}, \theta)$ for $\bs{H}$ to the algebra $(\bs{A}, \hspace{0.5mm} \theta \cdot \bs{A}_f)$ for $\bs{G}$.
\end{lemma}

\begin{proof} The fact that $\theta \cdot \bs{A}_f$ satisfies the unit and multiplication axioms for $\bs{G}$-algebras can be shown using the universal property of pullbacks. The same property is used to verify that any morphism $F: (\bs{A}, \theta) \longrightarrow (\bs{B}, \sigma)$ of $\bs{H}$-algebras is also a morphism $F:(\bs{A}, \, \theta \cdot \bs{A}_f) \longrightarrow (\bs{B}, \, \sigma \cdot \bs{B}_f)$ of $\bs{G}$-algebras, so $f$ is well-defined on morphisms and faithful.
\end{proof}

\begin{proposition}\label{split epi prop} If a morphism $\bs{G} \longrightarrow \bs{H}$ of globular operads is a split epimorphism on the underlying globular sets then the induced functor $\bs{H} \mhyphen \mathbf{Alg} \longrightarrow \bs{G} \mhyphen \mathbf{Alg}$ is injective on objects and full.
\end{proposition}

\begin{proof} Let $f:\bs{G} \longrightarrow \bs{H}$ be such a morphism. By assumption there exists a morphism $f':\bs{H} \longrightarrow \bs{G}$ of globular sets such that $ff'=1_{\bs{H}}$. For each globular set $\bs{A}$ this induces a morphism $\bs{A}_{f'}:\bs{A}_{\bs{H}} \longrightarrow \bs{A}_{\bs{G}}$ of globular sets satisfying $\bs{A}_{f} \bs{A}_{f'} = 1_{\bs{A_H}}$. Given $\bs{H}$-algebra structures $\theta$ and $\phi$ on $\bs{A}$ for which the composites $\theta \cdot \bs{A}_f$ and $\phi \cdot \bs{A}_f$ are equal
we can precompose with $\bs{A}_{f'}$ to get $\theta = \phi$, so the functor $f:\bs{H} \mhyphen \mathbf{Alg} \longrightarrow \bs{G} \mhyphen \mathbf{Alg}$ is injective on objects. The fact that $f$ is full can be shown similarly.
\end{proof}

Recall that $\bs{I}$ denotes the globular operad for weak unbiased $\omega$-categories; see Section \ref{weak unbiased section}.

\begin{lemma}\label{GlobOps for higher cats are contractible} Let $\bs{G}$ be a globular operad. If there exists a morphism $u:\bs{I} \longrightarrow \bs{G}$ of globular operads which is surjective on 0-cells and full on the underlying globular sets, then $\bs{G}$ is contractible.
\end{lemma}

\begin{proof} By Lemma \ref{contractible0Full} it is enough to show that the underlying collection map $g:\bs{G} \longrightarrow \bs{1}^*$ is full. By assumption, there exists morphism $u:\bs{I} \longrightarrow \bs{G}$ of globular sets which is surjective on 0-cells and full such that the following diagram commutes.
\vspace{-1mm}
\begin{center}
\begin{tikzpicture}[node distance=2.8cm, auto]

\node (X) {$\bs{I}$};
\node (Y) [right of=X] {$\bs{G}$};
\node (Z) [node distance=1.4cm, below of=X, right of=X] {$\bs{1}^*$};

\draw[->] (X) to node {$u$} (Y);
\draw[->] (Y) to node {$g$} (Z);
\draw[->] (X) to node [swap] {$i$} (Z);

\end{tikzpicture}
\end{center}
Additionally, by Lemma \ref{full+surj=split epi} there exists a morphism $u':\bs{G} \longrightarrow \bs{I}$ of globular sets such that $uu'=1_{\bs{G}}$. This means that for any parallel pair $(x,x')$ of $n$-cells $\bs{G}$ we have a parallel pair $(u'_n(x), u'_n(x')) = (y,y')$ of $n$-cells of $\bs{I}$ such that $(x,x')=(u_nu'_n(x), u_nu'_n(x'))=(u_n(y), u_n(y'))$. The restriction
\begin{center}
\begin{tikzpicture}[node distance=4.5cm, auto]

\node (X) {$G_{n+1}(x,x')$};
\node (Y) [right of=X] {$1^*_{n+1}(g_n(x), \, g_n(x'))$};

\draw[->] (X) to node {$g_{n+1}$} (Y);

\end{tikzpicture}
\end{center}
of $g_{n+1}$ may then be rewritten as 
\begin{center}
\begin{tikzpicture}[node distance=7.2cm, auto]

\node (X) {$G_{n+1}(u_n(y), \, u_n(y'))$};
\node (Y) [right of=X] {$1^*_{n+1}(g_nu_n(y), \, g_nu_n(y')) = 1^*_{n+1}(i_n(y), \, i_n(y'))$.};

\draw[->] (X) to node {$g_{n+1}$} (Y);

\end{tikzpicture}
\end{center}
Since $\bs{I}$ is contractible, $i$ must be full, so for each $\tau \in 1^*_{n+1}(i_n(y),i_n(y'))$ there exists an element $\psi \in I_{n+1}(y,y')$ such that $i_{n+1}(\psi)=\tau$. Then $g_{n+1}u_{n+1}(\psi) = i_{n+1}(\psi) = \tau$, so $g_{n+1}:G_{n+1}(x,x') \longrightarrow 1^*_{n+1}(g_n(x), \, g_n(x'))$ is surjective, implying that $g$ is full.
\end{proof}

\begin{remark}\label{One 0-cell}
Since $\bs{I}$ has exactly one 0-cell, the existence of a morphism $u:\bs{I} \longrightarrow \bs{G}$  of globular operads which is surjective on 0-cells implies that $\bs{G}$ must also contain exactly one 0-cell, the identity $\text{id}_0$. 
\end{remark}

Given a globular operad $\bs{G}$ for which there exists a morphism $u:\bs{I} \longrightarrow \bs{G}$ of globular operads which is surjective on 0-cells and full on the underlying globular sets, Lemma \ref{GlobOps for higher cats are contractible} and Remark \ref{One 0-cell} tell us that the algebras for $\bs{G}$ share many similarities with higher categories. The fact that $\bs{G}$ is contractible means that its algebras satisfy the required conditions on composition and coherence for $\omega$-categories; see Section \ref{contractibility section}. The fact that $\bs{G}$ contains exactly one 0-cell means that like higher categories, there are no non-trivial operations on the 0-cells of its algebras. On closer inspection, we see that that the algebras for $\bs{G}$ are \textit{precisely} $\omega$-categories:
the fact that $u$ is full means that the restriction
\vspace{1mm}
\begin{center}
\begin{tikzpicture}[node distance=4.2cm, auto]

\node (X) {$I_n(x,x')$};
\node (Y) [right of=X] {$G_n(u_{n-1}(x), \, u_{n-1}(x'))$};

\draw[->] (X) to node {$u_n$} (Y);

\end{tikzpicture}
\end{center}
of $u_n$ is a surjection for each $n>0$ and parallel pair $(x,x')$ of $(n - 1)$-cells of $\bs{I}$, so we may view the $n$-cells of $\bs{G}$ as a quotient of the $n$-cells of $\bs{I}$. We think of each $n$-cell $\chi:x \longrightarrow x'$ of $\bs{I}$ as an operation composing $n$-pasting diagrams of shape $i(\chi)$ with respect to the composition of the boundary given by $x$ and $x'$. This means that in $\bs{G}$ some of these composition operations are equal, which is to say that the algebras for $\bs{G}$ are a stricter variety of the $\omega$-categories defined by $\bs{I}$. Furthermore, it follows from Lemmas \ref{full+surj=split epi} and \ref{induced functor between alg cats} and Proposition \ref{split epi prop} that the induced functor
\vspace{1mm}
\begin{center}
\begin{tikzpicture}[node distance=4cm, auto]

\node (X) {$\bs{G} \mhyphen \mathbf{Alg}$};
\node (Y) [right of=X] {$\bs{I} \mhyphen \mathbf{Alg} = \mathbf{WkU \ \omega \mhyphen Cat}_{str}$};

\draw[->] (X) to node {$u$} (Y);

\end{tikzpicture}
\end{center}
is injective on objects, full and faithful, i.e., a full inclusion functor. So $\bs{G} \mhyphen \mathbf{Alg}$ is a full subcategory of the category $\mathbf{WkU} \ \omega \mhyphen \mathbf{Cat}_{str}$ of weak unbiased $\omega$-categories.

\begin{definition}
We say that a globular operad $\bs{G}$ is a \textit{globular operad for $\omega$-categories} if the only 0-cell of $\bs{G}$ is the identity $\text{id}_0$ and there exists a morphism $u:\bs{I} \longrightarrow \bs{G}$ of globular operads which is full on the underlying globular sets.
\end{definition}

\begin{remark} It should be noted that there are reasonable theories of $\omega$-category that do not fit the definition above. For example, $\omega$-categories with multiple ways to \textit{directly} compose a pair of 1-cells would be omitted. We have chosen to consider only those higher categories with at most one direct  composition operation for each shape of $n$-pasting diagram together with a composition of the boundary.
\end{remark}

Let $\bs{G}$ and $\bs{H}$ be globular operads for $\omega$-categories and let $f:\bs{G} \longrightarrow \bs{H}$ be a morphism of globular operads which is full on the underlying globular sets. Then by the same reasoning as above the $\omega$-categories defined by $\bs{H}$ must be a stricter variety of those defined by $\bs{G}$, and the induced functor $f:\bs{H} \mhyphen \mathbf{Alg} \longrightarrow \bs{G} \mhyphen \mathbf{Alg}$
is a full inclusion functor, making $\bs{H} \mhyphen \mathbf{Alg}$ a full subcategory of $\bs{G} \mhyphen \mathbf{Alg}$.

\begin{definition}
A \textit{morphism $f:\bs{G} \longrightarrow \bs{H}$ of globular operads for $\omega$-categories} is a morphism of globular operads which is full on the underlying gobular sets. 
\end{definition}

\begin{notation}We denote by $\mathbf{Cat \mhyphen GOp}$ the category of globular operads for $\omega$-categories. 
\end{notation}

The globular operad $\bs{I}$ for weak unbiased $\omega$-categories is weakly initial in $\mathbf{Cat \mhyphen GOp}$ by definition and the globular operad $\bs{T}$ for strict $\omega$-categories is terminal, so every operad in this category can be thought of as lying somewhere between $\bs{I}$ and $\bs{T}$. This, together with the fact that the existence of a morphism $f:\bs{G} \longrightarrow \bs{H}$ in $\mathbf{Cat \mhyphen GOp}$ implies that the $\omega$-categories defined by $\bs{H}$ are a stricter variety of those defined by $\bs{G}$, means that the category $\mathbf{Cat \mhyphen GOp}$ organises theories of algebraic $\omega$-category according to their relative weakness.

We now define $n$-globular operads for $n$-categories. First, recall that there is an extra requirement $n$-globular operads must meet in order to be contractible; see Definition \ref{contractible n-glob ops}. With this in mind, we give the following $n$-dimensional analog of Lemma \ref{contractible0Full}.

\begin{lemma} An $n$-globular operad $\bs{G_n}$ with underlying collection $g:\bs{G_n} \longrightarrow \bs{1}^*$ is contractible if and only if $g$ is full and $n$-faithful.
\end{lemma}

Next, recall the definition of the $n$-globular operad $\bs{I_n}$ for weak unbiased $n$-categories; Definition \ref{presentation for I_n}.

\begin{definition}
We say that an $n$-globular operad $\bs{G_n}$ is an \textit{$n$-globular operad for $n$-categories} if the only 0-cell of $\bs{G_n}$ is the identity $\text{id}_0$ and there exists a morphism $u:\bs{I_n} \longrightarrow \bs{G_n}$ of $n$-globular operads which full on the underlying $n$-globular sets.
\end{definition}

\begin{lemma}\label{n-G-Ops for n-cats are contractible} Every $n$-globular operad for $n$-categories is contractible.
\end{lemma}

\begin{definition}
A \textit{morphism $f:\bs{G_n} \longrightarrow \bs{H_n}$ of $n$-globular operads for $n$-categories} is a morphism of $n$-globular operads which is full on the underlying $n$-gobular sets. 
\end{definition}

\begin{notation}We denote by $\mathbf{Cat \mhyphen GOp}_{\bs{n}}$ the category of globular operads for $n$-categories. 
\end{notation}

To close this section, we provide some results that will be useful in Section \ref{examples section}.

\begin{lemma}\label{contractible=n-full and n-faithful} Let $\bs{G_n}$ and $\bs{H_n}$ be $n$-globular operads satisfying the second condition in the definition of contractibility for $n$-globular operads; Definition \ref{contractible n-glob ops}. Then every morphism $\bs{G_n} \longrightarrow \bs{H_n}$ of the underlying $n$-globular sets is $n$-full and $n$-faithful.
\end{lemma}








\begin{definition}\label{induced nat trans} Given a morphism $f:\bs{G_n} \longrightarrow \bs{H_n}$ of $n$-globular operads, $(\mhyphen)_f$ is the natural transformation 
\begin{center}
\begin{tikzpicture}[node distance=3.6cm, auto]

\node (X) {$\mathbf{GSet}_{\bs{n}}$};
\node (Y) [right of=X] {$\bs{G_n} \mhyphen \mathbf{Alg}$};
\node (A) [node distance=1.8cm, right of=X] {$\Downarrow (\mhyphen)_f$};
\node (Z) [node distance=0.9cm, below of=A] {$\bs{H_n} \mhyphen \mathbf{Alg}$};

\draw[->, bend left=40] (X) to node {} (Y);
\draw[->, bend right=15] (X) to node [swap] {} (Z);
\draw[->, bend right=15] (Z) to node [swap] {$f$} (Y);

\end{tikzpicture}
\end{center}
whose component at an $n$-globular set $\bs{A}$ is $\bs{A}_f$, defined by truncating Definition \ref{def of A_f} to $n$-dimensions by replacing globular operads with $n$-globular operads.
\end{definition} 

\begin{remark}
The natural transformation in Definition \ref{induced nat trans} is the mate of the identity natural transformation 

\begin{center}
\begin{tikzpicture}[node distance=3.6cm, auto]

\node (X) {$\bs{H_n} \mhyphen \mathbf{Alg}$};
\node (Y) [right of=X] {$\mathbf{GSet}_{\bs{n}}$.};
\node (A) [node distance=1.8cm, right of=X] {$\Downarrow 1$};
\node (Z) [node distance=0.9cm, below of=A] {$\bs{G_n} \mhyphen \mathbf{Alg}$};

\draw[->, bend left=40] (X) to node {} (Y);
\draw[->, bend right=15] (X) to node [swap] {$f$} (Z);
\draw[->, bend right=15] (Z) to node [swap] {} (Y);

\end{tikzpicture}
\end{center}
\end{remark}



\section{Examples of $n$-globular operads for $n$-categories}\label{examples section}

As mentioned in the introduction, a preprint of Michael Batanin \cite{batanin} conjectures that it should possible to take `slices' of globular operads. The $k^{th}$ slice was described as the symmetric operad determined by the $k$-dimensional data of a globular operad. Thus, given a globular operad for higher categories, the slices would isolate the algebraic structure of the associated higher categories in each dimension. As a first application of presentations, we will show in the follow up paper \cite{RG} that given a presentation $P$ for a globular operad $\bs{G}$, there exists a symmetric operad determined by the $k$-dimensional data of $P$; this symmetric operad is the $k^{th}$ slice of $\bs{G}$.

Batanin also hypothesised that slices could tell us when one theory of higher category is equivalent to another, and in particular, when a semi-strict notion of higher category is equivalent to the fully weak variety. In this section we construct presentations for the $n$-globular operads for two semi-strict theories of $n$-category in dimensions $n \leqslant 4$; $n$-categories with weak units in low dimensions, and $n$-categories with weak interchange laws. Using the language of presentations and slices, we show in \cite{RG} that both of our notions of semi-strict 4-category are equivalent to fully weak 4-categories.

Since the $n$-globular operads appearing in this section are all operads for some theory of $n$-category, they must all be contractible. We therefore give the corresponding presentations in the style of Example \ref{example of a contractible presentation}. In particular, when constructing a presentation for some $n$-globular operad $\bs{G_n}$ we will declare $\bs{G_n}$ to be contractible, so that the $n$-cells of $\bs{G_n}$ are determined by the lower dimensional cells and there is no need to specify any $n$-cell generators or relations; see Lemma \ref{contractible n-cells}. Following this, we will need to show that $\bs{G_n}$ is actually contractible, i.e., that the $\bs{G_n}$ also satisfies the definiton of contractibilty in dimensions $<n$; see Definition \ref{contractible n-glob ops}.

\subsection{Weak $n$-categories}\label{weaksection}

In this section we construct presentations for the 2, 3, and 4-globular operads for bicategories, tricategories and (biased) weak 4-categories, respectively. 

\begin{notation} We will denote by $w:\bs{W_n} \longrightarrow \bs{1}^*$ the underlying collection map of the $n$-globular operad $\bs{W_n}$ for weak $n$-categories
\end{notation}

\begin{definition}\label{W_2}
The 2-globular operad $\bs{W_2}$ for bicategories is the \textit{contractible} 2-globular operad with

\begin{list}{•}{}

\item a single 0-cell, the identity $\text{id}_0$; and

\item 1-cells consisting of the operadic composites of 1-cells $i_1$ and $h_1$ whose images under the underlying collection map are as follows.

\begin{center}
\begin{tikzpicture}[node distance=2cm, auto]

\node (A) {$\cdot$};
\node (B) [right of =A] {$\cdot$};

\draw[->] (A) to node {$i_1$} (B);

\node (P) [node distance=3cm, right of =B] {$\cdot$};

\node (b) [node distance=0.5cm, right of=B] {};
\node (p) [node distance=0.5cm, left of=P] {};
\draw[|->, dashed] (b) to node {} (p);

\node (A') [node distance=1cm, below of=A] {$\cdot$};
\node (B') [right of =A'] {$\cdot$};

\draw[->] (A') to node {$h_1$} (B');

\node (P') [node distance=3cm, right of =B'] {$\cdot$};
\node (Q') [right of =P'] {$\cdot$};
\node (R') [right of =Q'] {$\cdot$};

\draw[->] (P') to node {} (Q');
\draw[->] (Q') to node {} (R');

\node (b') [node distance=0.5cm, right of=B'] {};
\node (p') [node distance=0.5cm, left of=P'] {};
\draw[|->, dashed] (b') to node {} (p');

\end{tikzpicture}
\end{center}

\end{list}
\end{definition}

Just as an algebra for the 2-globular operad $\bs{I_2}$ is an unbiased bicategory (Example \ref{I_2}), an algebra $\theta:\bs{A_{W_2}} \longrightarrow \bs{A}$
for $\bs{W_2}$ is precisely a (biased) bicategory with underlying 2-globular set $\bs{A}$. 
A morphism of algebras for $\bs{W_2}$ is strict 2-functor between bicategories.

\begin{notation}
We denote by $\mathbf{Bicat}_{str}$ the category $\bs{W_2} \mhyphen \mathbf{Alg}$ of algebras for $\bs{W_2}$. 
\end{notation}

The 2-globular operad $\bs{W_2}$ satisfies the definition of contractibility in dimension 2 by construction. However, in order for Definition \ref{W_2} to make sense we need to show that that $\bs{W_2}$ is actually contractible. By Lemma \ref{n-G-Ops for n-cats are contractible} it suffices to show that $\bs{W_2}$ is an object in the category $\mathbf{Cat \mhyphen Op_2}$ of 2-globular operads for 2-categories, meaning that there exists a morphism $f:\bs{I_2} \longrightarrow \bs{W_2}$ of 2-globular operads which is full on the underlying 2-globular sets. Furthermore, since our presentation for $\bs{I_2}$ consisted of a 1-cell generator $c^n$ for each $n \in \mathbb{N}$ (see Example \ref{I_2}), Lemmas \ref{contractible n-morphisms} and \ref{map is determined at k-cell generators} tell us that a morphism $f:\bs{I_2} \longrightarrow \bs{W_2}$ of 2-globular operads is completely determined by its value on the the $c^n$s. This gives several choices for $f$, for example we could choose,

\begin{align*}
f_1(c^0) &= i_1\\
f_1(c^1) &= \text{id}_1\\
f_1(c^2) &= h_1\\
f_1(c^3) &= h_1 \circ (h_1, \text{id}_1)\\
\vdots \\
f_1(c^n) &= h_1 \circ \big( ... \ h_1 \circ \big(h_1 \circ (h_1,\text{id}_1),\text{id}_1\big)...,\text{id}_1\big).
\end{align*}
For this choice of $f$ both 1-cell generators $i_1$ and $h_1$ of $\bs{W_2}$ are in the image of $f_1:I_1 = I_1(\text{id}_0, \text{id}_0) \longrightarrow W_1(\text{id}_0, \text{id}_0) = W_1$.\label{contructing a full map from I_2 to W_2} 
Since the remaining 1-cells of $\bs{W_2}$ are operadic composites of $i_1$ and $h_1$, and morphisms of $n$-globular operads preserve operadic composition, $f_1$ must be surjective, so $f$ is $1$-full. It now follows from Lemma \ref{contractible=n-full and n-faithful} that $f$ is full, so $\bs{W_2}$ is indeed contractible. The induced functor $f:\bs{W_2} \mhyphen \mathbf{Alg} \longrightarrow \bs{I_2} \mhyphen \mathbf{Alg}$ is just the inclusion functor $\mathbf{Bicat}_{str} \longrightarrow \mathbf{UBicat}_{str}$ corresponding to our choice of $f$. Explicitly, it is the functor sending a bicategory $\mathcal{B}$ to the unbiased bicategory whose underlying bicategory is $\mathcal{B}$ and for which the $n$-ary composite of $n$ composable 1-cells $a_1,...,a_n$ is given by the composite $\big(\big(...\big((a_1a_2)a_3\big)...\big)a_{n-1}\big)a_n$. 

\begin{definition}\label{W_3} 
The 3-globular operad $\bs{W_3}$ for tricategories is the  \textit{contractible} 3-globular operad with

\begin{list}{•}{}

\item the same 0 and 1-cells as $\bs{W_2}$; and

\item 2-cells consisting of the operadic composites of 2-cells $i_2$, $h_2$, $v_2$, $l_2$, $l'_2$, $r_2$, $r'_2$, $a_2$ and $a'_2$ whose images under the underlying collection map are as follows.

\begin{center}
\begin{tikzpicture}[node distance=2.5cm, auto]

\node (A) {$\cdot$};
\node (B) [node distance=2.5cm, right of =A] {$\cdot$};

\draw[->, bend left=40] (A) to node {$\text{id}_1$} (B);
\draw[->, bend right=40] (A) to node [swap] {$\text{id}_1$} (B);

\node () [node distance=1.25cm, right of =A] {$\Downarrow i_2$};

\node (P) [node distance=3cm, right of =B] {$\cdot$};
\node (Q) [node distance=2cm, right of =P] {$\cdot$};

\draw[->] (P) to node {} (Q);

\node (b) [node distance=0.5cm, right of=B] {};
\node (p) [node distance=0.5cm, left of=P] {};
\draw[|->, dashed] (b) to node {} (p);

\node (A') [node distance=2.3cm, below of=A] {$\cdot$};
\node (B') [node distance=2.5cm, right of =A'] {$\cdot$};

\draw[->, bend left=40] (A') to node {$h_1$} (B');
\draw[->, bend right=40] (A') to node [swap] {$h_1$} (B');

\node () [node distance=1.25cm, right of =A'] {$\Downarrow h_2$};

\node (P') [node distance=3cm, right of =B'] {$\cdot$};
\node (Q') [node distance=2.2cm, right of =P'] {$\cdot$};
\node (R') [node distance=2.2cm, right of =Q'] {$\cdot$};

\draw[->, bend left=40] (P') to node {} (Q');
\draw[->, bend right=40] (P') to node {} (Q');
\draw[->, bend left=40] (Q') to node {} (R');
\draw[->, bend right=40] (Q') to node {} (R');

\node () [node distance=1.1cm, right of =P'] {$\Downarrow$};
\node () [node distance=1.1cm, right of =Q'] {$\Downarrow$};

\node (b') [node distance=0.5cm, right of=B'] {};
\node (p') [node distance=0.5cm, left of=P'] {};
\draw[|->, dashed] (b') to node {} (p');

\end{tikzpicture}
\end{center}


\begin{center}
\begin{tikzpicture}[node distance=2.5cm, auto]

\node (A'') {$\cdot$};
\node (B'') [node distance=2.5cm, right of =A''] {$\cdot$};

\draw[->, bend left=40] (A'') to node {$\text{id}_1$} (B'');
\draw[->, bend right=40] (A'') to node [swap] {$\text{id}_1$} (B'');

\node () [node distance=1.25cm, right of =A''] {$\Downarrow v_2$};

\node (P'') [node distance=3cm, right of =B''] {$\cdot$};
\node (Q'') [node distance=2.2cm, right of =P''] {$\cdot$};

\draw[->, bend left=60] (P'') to node {} (Q'');
\draw[->] (P'') to node {} (Q'');
\draw[->, bend right=60] (P'') to node {} (Q'');

\node (x) [node distance=1.1cm, right of =P''] {};
\node () [node distance=0.3cm, above of =x] {$\Downarrow$};
\node () [node distance=0.35cm, below of =x] {$\Downarrow$};

\node (b'') [node distance=0.5cm, right of=B''] {};
\node (p'') [node distance=0.5cm, left of=P''] {};
\draw[|->, dashed] (b'') to node {} (p'');

\node (X) [node distance=2.3cm, below of=A''] {$\cdot$};
\node (Y) [node distance=2.5cm, right of =X] {$\cdot$};

\draw[->, bend left=40] (X) to node {$h_1 \circ (\text{id}_1, i_1)$} (Y);
\draw[->, bend right=40] (X) to node [swap] {$\text{id}_1$} (Y);

\node () [node distance=1.25cm, right of =X] {$\Downarrow l_2$};

\node (U) [node distance=3cm, right of =Y] {$\cdot$};
\node (V) [node distance=2cm, right of =U] {$\cdot$};

\draw[->] (U) to node {} (V);

\node (y) [node distance=0.5cm, right of=Y] {};
\node (u) [node distance=0.5cm, left of=U] {};
\draw[|->, dashed] (y) to node {} (u);

\node (X') [node distance=2.3cm, below of=X] {$\cdot$};
\node (Y') [node distance=2.5cm, right of =X'] {$\cdot$};

\draw[->, bend left=40] (X') to node {$\text{id}_1$} (Y');
\draw[->, bend right=40] (X') to node [swap] {$h_1 \circ (\text{id}_1, i_1)$} (Y');

\node () [node distance=1.25cm, right of =X'] {$\Downarrow l'_2$};

\node (U') [node distance=3cm, right of =Y'] {$\cdot$};
\node (V') [node distance=2cm, right of =U'] {$\cdot$};

\draw[->] (U') to node {} (V');

\node (y') [node distance=0.5cm, right of=Y'] {};
\node (u') [node distance=0.5cm, left of=U'] {};
\draw[|->, dashed] (y') to node {} (u');

\node (X'') [node distance=2.3cm, below of=X'] {$\cdot$};
\node (Y'') [node distance=2.5cm, right of =X''] {$\cdot$};

\draw[->, bend left=40] (X'') to node {$h_1 \circ (i_1, \text{id}_1)$} (Y'');
\draw[->, bend right=40] (X'') to node [swap] {$\text{id}_1$} (Y'');

\node () [node distance=1.25cm, right of =X''] {$\Downarrow r_2$};

\node (U'') [node distance=3cm, right of =Y''] {$\cdot$};
\node (V'') [node distance=2cm, right of =U''] {$\cdot$};

\draw[->] (U'') to node {} (V'');

\node (y'') [node distance=0.5cm, right of=Y''] {};
\node (u'') [node distance=0.5cm, left of=U''] {};
\draw[|->, dashed] (y'') to node {} (u'');

\node (X''') [node distance=2.3cm, below of=X''] {$\cdot$};
\node (Y''') [node distance=2.5cm, right of =X'''] {$\cdot$};

\draw[->, bend left=40] (X''') to node {$\text{id}_1$} (Y''');
\draw[->, bend right=40] (X''') to node [swap] {$h_1 \circ (i_1, \text{id}_1)$} (Y''');

\node () [node distance=1.25cm, right of =X'''] {$\Downarrow r'_2$};

\node (U''') [node distance=3cm, right of =Y'''] {$\cdot$};
\node (V''') [node distance=2cm, right of =U'''] {$\cdot$};

\draw[->] (U''') to node {} (V''');

\node (y''') [node distance=0.5cm, right of=Y'''] {};
\node (u''') [node distance=0.5cm, left of=U'''] {};
\draw[|->, dashed] (y''') to node {} (u''');

\node (N) [node distance=2.4cm, below of=X'''] {$\cdot$};
\node (M) [node distance=2.5cm, right of =N] {$\cdot$};

\draw[->, bend left=40] (N) to node {$h_1 \circ (\text{id}_1, h_1)$} (M);
\draw[->, bend right=40] (N) to node [swap] {$h_1 \circ (h_1, \text{id}_1)$} (M);

\node () [node distance=1.25cm, right of =N] {$\Downarrow a_2$};

\node (H) [node distance=3cm, right of =M] {$\cdot$};
\node (I) [node distance=2cm, right of =H] {$\cdot$};
\node (J) [node distance=2cm, right of =I] {$\cdot$};
\node (K) [node distance=2cm, right of =J] {$\cdot$};

\draw[->] (H) to node {} (I);
\draw[->] (I) to node {} (J);
\draw[->] (J) to node {} (K);

\node (m) [node distance=0.5cm, right of=M] {};
\node (h) [node distance=0.5cm, left of=H] {};
\draw[|->, dashed] (m) to node {} (h);

\node (N') [node distance=2.4cm, below of=N] {$\cdot$};
\node (M') [node distance=2.5cm, right of =N'] {$\cdot$};

\draw[->, bend left=40] (N') to node {$h_1 \circ (h_1, \text{id}_1)$} (M');
\draw[->, bend right=40] (N') to node [swap] {$h_1 \circ (\text{id}_1, h_1)$} (M');

\node () [node distance=1.25cm, right of =N'] {$\Downarrow a'_2$};

\node (H') [node distance=3cm, right of =M'] {$\cdot$};
\node (I') [node distance=2cm, right of =H'] {$\cdot$};
\node (J') [node distance=2cm, right of =I'] {$\cdot$};
\node (K') [node distance=2cm, right of =J'] {$\cdot$};

\draw[->] (H') to node {} (I');
\draw[->] (I') to node {} (J');
\draw[->] (J') to node {} (K');

\node (m') [node distance=0.5cm, right of=M'] {};
\node (h') [node distance=0.5cm, left of=H'] {};
\draw[|->, dashed] (m') to node {} (h');

\end{tikzpicture}
\end{center}
\end{list}
\end{definition}

An algebra $\theta:\bs{A_{W_3}} \longrightarrow \bs{A}$ for $\bs{W_3}$ is precisely a (biased) tricategory in the sense of \cite{NG} with underlying 3-globular set $\bs{A}$. In dimensions $\leqslant 1$ an algebra for $\bs{W_3}$ is the same as an algebra for $\bs{W_2}$, since the 0 and 1-cells of these operads are the same. The 2-cell generators $i_2$, $h_2$ and $v_2$ of $\bs{W_3}$ provide 2-cell identites, binary horizontal composition of 2-cells and binary vertical composition of 2-cells in $\bs{A}$, respectively. The remaining 2-cell generators for $\bs{W_3}$, $l_2$, $r_2$, $a_2$,  $l'_2$, $r'_2$ and $a'_2$, provide the left unit coherence 2-cells, right unit coherence 2-cells, associativity coherence 2-cells, and their (weak) inverses in $\bs{A}$, respectively. 

\begin{notation}
We denote by $\mathbf{Tricat}_{str}$ the category $\bs{W_3} \mhyphen \mathbf{Alg}$ of algebras for $\bs{W_2}$. 
\end{notation}

We may define a morphism $f:\bs{I_3} \longrightarrow \bs{W_3}$ of 3-globular operads which is full on the underlying 3-globular sets similarly to how we defined the morphism $\bs{I_2} \longrightarrow \bs{W_2}$ of 2-globular operads on page \pageref{contructing a full map from I_2 to W_2}. It follows that $\bs{W_3}$ is an object in the category $\mathbf{Cat \mhyphen Op_3}$ of 3-globular operads for 3-categories, and is therefore indeed contractible. The induced functor $f:\bs{W_3} \mhyphen \mathbf{Alg} \longrightarrow \bs{I_3} \mhyphen \mathbf{Alg}$ is the inclusion functor $\mathbf{Tricat}_{str} \longrightarrow \mathbf{UTricat}_{str}$ corresponding to the choice of $f$.

In dimensions $\leqslant 1$ bicategories and tricategories are identical, so the $n$-globular operads $\bs{W_2}$ and $\bs{W_3}$ for bicategories and tricategories, respectively, have the same 0 and 1-cells. More generally, weak $n$-categories and weak $(n{+}1)$-categories should be identical in dimensions $\leqslant n-1$, so the operads $\bs{W_n}$ and $\bs{W_{n+1}}$ for weak $n$-categories and weak $(n{+}1)$-categories, respectively, should have the same $k$-cells for all $k \leqslant n-1$. Additionally, weak $n$-categories and weak $(n{+}1)$-categories share the same types of binary composition operations on $n$-cells and the same coherence $n$-cells, but only weak $n$-categories possess axioms for $n$-cell composition. For example, ordinary categories and bicategories both have a single binary composition operation on 1-cells and the same coherence 1-cells, namely the identity 1-cells. However, 1-cell composition only satisfies axioms in ordinary categories. In bicategories, the corresponding axioms are pushed to dimension 2, where they are replaced by the invertibility, naturality and compatibility axioms satisfied by the coherence 2-cells. Analogously, bicategories and tricategories both have binary horizontal and vertical composition operations on 2-cells, and the same kinds of coherence 2-cells. However, only in bicategories does 2-cell composition satisfy any axioms. In tricategories, the axioms are pushed to dimension 3. With this in mind, we give the following presentation for the 4-globular operad $\bs{W_4}$ for weak 4-categories.

\begin{definition}\label{W_4}
The 4-globular operad $\bs{W_4}$ for weak 4-categories is the \textit{contractible} 4-globular operad with

\begin{list}{•}{}

\item the same 0, 1 and 2-cells as $\bs{W_3}$; and

\item 3-cells consisting of the operadic composites of forty-six 3-cells, including 3-cells $i_3$, $h_3$, $v_3$ and $c_3$ whose images under the underlying collection map are as follows.

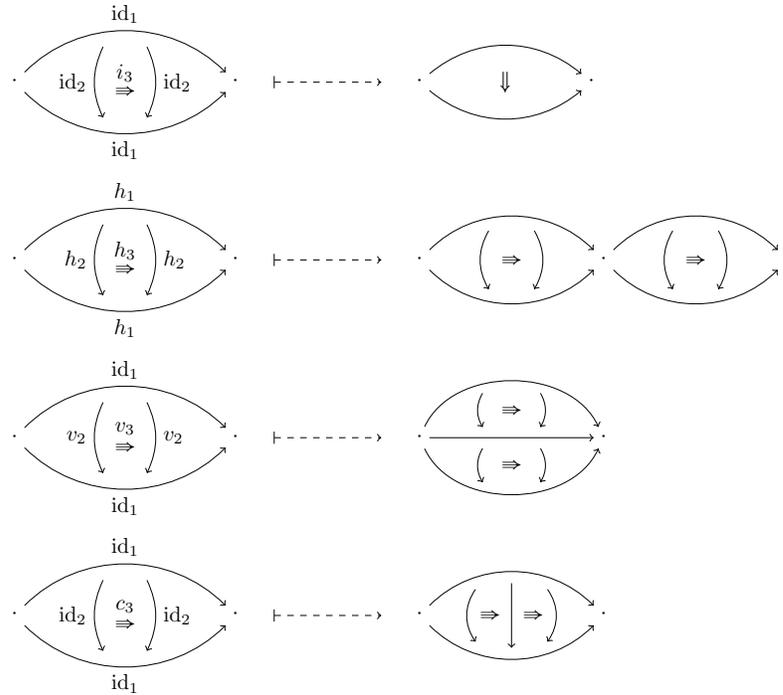
\begin{center}
\resizebox{0.67\textwidth}{!}{
\begin{tikzpicture}[node distance=2cm, auto]

\node (A) {$\cdot$};
\node (B) [node distance=3.6cm, right of=A] {$\cdot$};
\node (a) [node distance=1.8cm, right of=A] {};
\node () [node distance=0.15cm, above of=a] {$i_3$};
\node () [node distance=0.15cm, below of=a] {$\Rrightarrow$};

\node (c) [node distance=1.5cm, right of=A] {};
\node (e) [node distance=0.7cm, above of=c] {};
\node (f) [node distance=0.7cm, below of=c] {};
\node (d) [node distance=2.1cm, right of=A] {};
\node (g) [node distance=0.7cm, above of=d] {};
\node (h) [node distance=0.7cm, below of=d] {};

\draw[->, bend left=45] (A) to node {$\text{id}_1$} (B);
\draw[->, bend right=45] (A) to node [swap] {$\text{id}_1$} (B);
\draw[->, bend right=25] (e) to node [swap] {$\text{id}_2$} (f);
\draw[->, bend left=25] (g) to node {$\text{id}_2$} (h);

\node (P) [node distance=3cm, right of =B] {$\cdot$};
\node (Q) [node distance=2.8cm, right of =P] {$\cdot$};
\node () [node distance=1.4cm, right of=P] {$\Downarrow$};

\draw[->, bend left=40] (P) to node {} (Q);
\draw[->, bend right=40] (P) to node {} (Q);

\node (x) [node distance=0.5cm, right of=B] {};
\node (y) [node distance=0.5cm, left of=P] {};
\draw[|->, dashed] (x) to node {} (y);

\node (A') [node distance=2.9cm, below of=A] {$\cdot$};
\node (B') [node distance=3.6cm, right of=A'] {$\cdot$};
\node (a') [node distance=1.8cm, right of=A'] {};
\node () [node distance=0.15cm, above of=a'] {$h_3$};
\node () [node distance=0.15cm, below of=a'] {$\Rrightarrow$};

\node (c') [node distance=1.5cm, right of=A'] {};
\node (e') [node distance=0.7cm, above of=c'] {};
\node (f') [node distance=0.7cm, below of=c'] {};
\node (d') [node distance=2.1cm, right of=A'] {};
\node (g') [node distance=0.7cm, above of=d'] {};
\node (h') [node distance=0.7cm, below of=d'] {};

\draw[->, bend left=45] (A') to node {$h_1$} (B');
\draw[->, bend right=45] (A') to node [swap] {$h_1$} (B');
\draw[->, bend right=25] (e') to node [swap] {$h_2$} (f');
\draw[->, bend left=25] (g') to node {$h_2$} (h');

\node (P') [node distance=3cm, right of=B'] {$\cdot$};
\node (Q') [node distance=3cm, right of=P'] {$\cdot$};
\node (R') [node distance=3cm, right of=Q'] {$\cdot$};
\node () [node distance=1.5cm, right of=P'] {$\Rrightarrow$};
\node () [node distance=1.5cm, right of=Q'] {$\Rrightarrow$};

\node (u') [node distance=1.2cm, right of=P'] {};
\node (i') [node distance=0.6cm, above of=u'] {};
\node (j') [node distance=0.6cm, below of=u'] {};
\node (v') [node distance=1.8cm, right of=P'] {};
\node (n') [node distance=0.6cm, above of=v'] {};
\node (m') [node distance=0.6cm, below of=v'] {};

\node (w') [node distance=1.2cm, right of=Q'] {};
\node (k') [node distance=0.6cm, above of=w'] {};
\node (l') [node distance=0.6cm, below of=w'] {};
\node (z') [node distance=1.8cm, right of=Q'] {};
\node (s') [node distance=0.6cm, above of=z'] {};
\node (t') [node distance=0.6cm, below of=z'] {};

\draw[->, bend left=45] (P') to node {} (Q');
\draw[->, bend right=45] (P') to node [swap] {} (Q');
\draw[->, bend right=30] (i') to node [swap] {} (j');
\draw[->, bend left=30] (n') to node {} (m');

\draw[->, bend left=45] (Q') to node {} (R');
\draw[->, bend right=45] (Q') to node [swap] {} (R');
\draw[->, bend right=30] (k') to node [swap] {} (l');
\draw[->, bend left=30] (s') to node {} (t');

\node (x') [node distance=0.5cm, right of=B'] {};
\node (y') [node distance=0.5cm, left of=P'] {};
\draw[|->, dashed] (x') to node {} (y');

\node (A'') [node distance=2.9cm, below of=A'] {$\cdot$};
\node (B'') [node distance=3.6cm, right of=A''] {$\cdot$};
\node (a'') [node distance=1.8cm, right of=A''] {};
\node () [node distance=0.15cm, above of=a''] {$v_3$};
\node () [node distance=0.15cm, below of=a''] {$\Rrightarrow$};

\node (c'') [node distance=1.5cm, right of=A''] {};
\node (e'') [node distance=0.7cm, above of=c''] {};
\node (f'') [node distance=0.7cm, below of=c''] {};
\node (d'') [node distance=2.1cm, right of=A''] {};
\node (g'') [node distance=0.7cm, above of=d''] {};
\node (h'') [node distance=0.7cm, below of=d''] {};

\draw[->, bend left=45] (A'') to node {$\text{id}_1$} (B'');
\draw[->, bend right=45] (A'') to node [swap] {$\text{id}_1$} (B'');
\draw[->, bend right=25] (e'') to node [swap] {$v_2$} (f'');
\draw[->, bend left=25] (g'') to node {$v_2$} (h'');

\node (P'') [node distance=3cm, right of=B''] {$\cdot$};
\node (Q'') [node distance=3cm, right of=P''] {$\cdot$};

\draw[->, bend left=65] (P'') to node {} (Q'');
\draw[->] (P'') to node {} (Q'');
\draw[->, bend right=65] (P'') to node [swap] {} (Q'');

\node (u'') [node distance=1.5cm, right of=P''] {};
\node (i'') [node distance=0.45cm, above of=u''] {$\Rrightarrow$};
\node (j'') [node distance=0.45cm, below of=u''] {$\Rrightarrow$};

\node (n'') [node distance=0.4cm, left of=i'', above of=i''] {};
\node (m'') [node distance=0.4cm, left of=i'', below of=i''] {};
\node (k'') [node distance=0.4cm, right of=i'', above of=i''] {};
\node (l'') [node distance=0.4cm, right of=i'', below of=i''] {};
\node (w'') [node distance=0.4cm, left of=j'', above of=j''] {};
\node (z'') [node distance=0.4cm, left of=j'', below of=j''] {};
\node (s'') [node distance=0.4cm, right of=j'', above of=j''] {};
\node (t'') [node distance=0.4cm, right of=j'', below of=j''] {};

\draw[->, bend right=30] (n'') to node [swap] {} (m'');
\draw[->, bend left=30] (k'') to node {} (l'');
\draw[->, bend right=30] (w'') to node [swap] {} (z'');
\draw[->, bend left=30] (s'') to node {} (t'');

\node (x'') [node distance=0.5cm, right of=B''] {};
\node (y'') [node distance=0.5cm, left of=P''] {};
\draw[|->, dashed] (x'') to node {} (y'');

\node (A''') [node distance=2.9cm, below of=A''] {$\cdot$};
\node (B''') [node distance=3.6cm, right of=A'''] {$\cdot$};
\node (a''') [node distance=1.8cm, right of=A'''] {};
\node () [node distance=0.15cm, above of=a'''] {$c_3$};
\node () [node distance=0.15cm, below of=a'''] {$\Rrightarrow$};

\node (c''') [node distance=1.5cm, right of=A'''] {};
\node (e''') [node distance=0.7cm, above of=c'''] {};
\node (f''') [node distance=0.7cm, below of=c'''] {};
\node (d''') [node distance=2.1cm, right of=A'''] {};
\node (g''') [node distance=0.7cm, above of=d'''] {};
\node (h''') [node distance=0.7cm, below of=d'''] {};

\draw[->, bend left=45] (A''') to node {$\text{id}_1$} (B''');
\draw[->, bend right=45] (A''') to node [swap] {$\text{id}_1$} (B''');
\draw[->, bend right=25] (e''') to node [swap] {$\text{id}_2$} (f''');
\draw[->, bend left=25] (g''') to node {$\text{id}_2$} (h''');

\node (P''') [node distance=3cm, right of=B'''] {$\cdot$};
\node (Q''') [node distance=3cm, right of=P'''] {$\cdot$};

\node (u''') [node distance=1cm, right of=P'''] {};
\node (i''') [node distance=1.5cm, right of=P'''] {};
\node (j''') [node distance=2cm, right of=P'''] {};
\node (n''') [node distance=0.55cm, above of=u'''] {};
\node (m''') [node distance=0.55cm, below of=u'''] {};
\node (k''') [node distance=0.65cm, above of=i'''] {};
\node (l''') [node distance=0.65cm, below of=i'''] {};
\node (w''') [node distance=0.55cm, above of=j'''] {};
\node (z''') [node distance=0.55cm, below of=j'''] {};
\node () [node distance=1.15cm, right of=P'''] {$\Rrightarrow$};
\node () [node distance=1.85cm, right of=P'''] {$\Rrightarrow$};

\draw[->, bend left=45] (P''') to node {} (Q''');
\draw[->, bend right=45] (P''') to node [swap] {} (Q''');
\draw[->, bend right=35] (n''') to node [swap] {} (m''');
\draw[->] (k''') to node [swap] {} (l''');
\draw[->, bend left=35] (w''') to node {} (z''');

\node (x''') [node distance=0.5cm, right of=B'''] {};
\node (y''') [node distance=0.5cm, left of=P'''] {};
\draw[|->, dashed] (x''') to node {} (y''');

\end{tikzpicture}
}
\end{center}
\end{list}

An algebra for $\bs{W_4}$ is a weak 4-category. The 3-cell generators $i_3$, $h_3$, $v_3$ and $c_3$ provide 3-cell identities and binary composition of 3-cells along matching 0, 1, and 2-cell boundaries, respectively. The remaining forty-two 3-cell generators correspond to the kinds of coherence 3-cells in the tetracategories of \cite{AH}, which in turn match the coherence 3-cells present in tricategories \cite{NG}. The 3-cell generators for $\bs{W_4}$ can be found listed explicitly in \cite{RG1}.
\end{definition}

\begin{notation}
We denote by $\mathbf{Wk \ 4 \mhyphen Cat}_{str}$ the category $\bs{W_4} \mhyphen \mathbf{Alg}$ of algebras for $\bs{W_4}$. 
\end{notation}

In analogy with the lower dimensional cases, it is straighforward to construct a morphism $f:\bs{I_4} \longrightarrow \bs{W_4}$ of 4-globular operads which is full on the underlying 4-globular sets. It follows that $\bs{W_4}$ is indeed contractible. 

\begin{remark} It is unknown whether the tetracategories of \cite{AH} satisfy the Coherence Theorem \ref{the coherence theorem}. 
If they do, then an algebra for $\bs{W_4}$ is precisely one of these tetracategories.  If not, then the algebras for $\bs{W_4}$ must be a stricter variation.
\end{remark}

We can now define, for $n=2,3$ and $4$, a subcategory of  $\mathbf{Cat \mhyphen GOp}_{\bs{n}}$ whose objects are globular operads for \textit{biased} $n$-categories.

\begin{definition}\label{W_n -> means BCat} An \textit{$n$-globular operad for biased $n$-categories} is an $n$-globular operad $\bs{G_n}$ for which there exists a morphism $\bs{W_n} \longrightarrow \bs{G_n}$ of $n$-globular operads which is full on the underlying $n$-globular sets.
\end{definition}

\begin{notation}\label{BCat} We denote by $\mathbf{BCat \mhyphen GOp}_{\bs{n}}$ the full subcategory of $\mathbf{Cat \mhyphen GOp}_{\bs{n}}$ whose objects are the $n$-globular operad for biased $n$-categories.
\end{notation}

\begin{remark}\label{BCat op are contractible}
Note that since $\mathbf{BCat \mhyphen GOp}_{\bs{n}}$ is a subcategory $\mathbf{Cat \mhyphen GOp}_{\bs{n}}$ every $n$-globular operad for biased $n$-categories is contractible.
\end{remark}

The presentations for the $n$-globular operads $\bs{W_n}$ for weak $n$-categories given in this section consist of a $k$-cell generator for each binary composition operation on $k$-cells and each kind of coherence $k$-cell present in a weak $n$-category ($k<n$). Presentations for $n$-globular operads for stricter varieties of $n$-category will require fewer $k$-cell generators since there are fewer coherence cells, but will also require some $k$-cell relations, each of will yield an axiom for $k$-cell composition.

\subsection{Semi-strict $n$-categories}\label{semi-strict section}

In this section we construct presentations for the $n$-globular operads for two theories of semi-strict $n$-category in dimensions $n \leqslant 4$. The first are $n$-categories with weak identities in low dimensions (namely, dimensions $\leqslant n-2$), and the second are $n$-categories with weak interchange laws. Note that for $n=2$, both of these are precisely strict 2-categories. 

\begin{definition}\label{E_4}
The 4-globular operad $\bs{E_4}$ for 4-categories with weak units in low dimensions is the  \textit{contractible} 4-globular operad with

\begin{list}{•}{}

\item a single 0-cell, the identity $\text{id}_0$; 

\item 1-cells consisting of the operadic composites of 1-cells $i_1$ and $h_1$ whose images under the underlying collection map are as in the presentation for $\bs{W_2}$, subject to the equality $h_1 \circ (\text{id}_1, h_1) = h_1 \circ (h_1, \text{id}_1)$;

\item 2-cells consisting of the operadic composites of 2-cells $i_2$, $h_2$, $v_2$, $l_2$, $l'_2$, $r_2$ and $r'_2$ whose images under the underlying collection map are as in the presentation for $\bs{W_3}$, subject to the following equalities

\begin{enumerate}[i)]
\item $h_2 \circ (\text{id}_2, h_2) = h_2 \circ (h_2, \text{id}_2)$
\item $v_2 \circ (\text{id}_2, v_2) = v_2 \circ (v_2, \text{id}_2)$
\item $h_2 \circ (v_2,v_2) = v_2 \circ (h_2,h_2)$
\end{enumerate}
and;

\item 3-cells consisting of the operadic composites of thirty-two 3-cells, including 3-cells $i_3$, $h_3$, $v_3$ and $c_3$, whose images under the underlying collection map are as in the presentation for $\bs{W_4}$, subject to the equalities listed below.

\begin{enumerate}[i)]

\item $c_3 \circ (\text{id}_3,i_3) = \text{id}_3$

\item $c_3 \circ (i_3, \text{id}_3) = \text{id}_3$

\item $h_3 \circ (\text{id}_3, h_3) = h_3 \circ (h_3,\text{id}_3)$

\item $v_3 \circ (\text{id}_3, v_3) = v_3 \circ (v_3, \text{id}_3)$

\item $c_3 \circ (\text{id}_3, c_3) = c_3 \circ (c_3, \text{id}_3)$

\item $h_3 \circ (v_3,v_3) = v_3 \circ (h_3,h_3)$

\item $v_3 \circ (c_3,c_3) = c_3 \circ (v_3,v_3)$ 

\item $c_3 \circ (h_3,h_3) = h_3 \circ (c_3,c_3)$

\item $h_3 \circ (i_3,i_3) = i_3 \circ (h_2)$

\item $v_3 \circ (i_3,i_3) = i_3 \circ (v_2)$

\end{enumerate}

\end{list}

\end{definition}
 
An algebra $\theta:\bs{A_{E_4}} \longrightarrow \bs{A}$ for $\bs{E_4}$ is a 4-category with weak units in dimensions $\leqslant 2$ whose underlying 4-globular set is $\bs{A}$. The single 1-cell relation, i.e., the equality $h_1 \circ (\text{id}_1, h_1) = h_1 \circ (h_1, \text{id}_1)$, yields an associativity axiom for 1-cell composition; given a 1-pasting diagram $(a,b,c)$ in $\bs{A}$ we have 
$$a(bc) = \big(h_1 \circ (\text{id}_1, h_1)\big) (a,b,c) = \big(h_1 \circ (h_1, \text{id}_1)\big) (a,b,c)=(ab)c.$$
Thus, the 2-cell generators $a_2$ and $a_2'$ for associativity coherence 2-cells that appear in the presentation for $\bs{W_n}$ do not appear here, but every other 2-cell generator does. The 2-cell relations yield associativity axioms for horizontal and vertical composition of 2-cells, respectively, and an interchange law.
The twenty-eight 3-cell generators not listed above provide coherence 3-cells related to the identity 1- and 2-cells, and can be found listed explicitly in \cite{RG1}. The first two 3-cell relations yield left and right unit axioms, respectively, for composition of 3-cells along a 2-cell boundary. The next three yield associativity axioms (one for each binary composition operation on 3-cells). The three that follow yield interchange laws (one for each pair of binary composition operations on 3-cells). The final two 3-cell relations yield axioms stating that the composite of two 3-cell identities along a 0 or a 1-cell boundary, respectively, is an identity 2-cell.

\begin{definition} \label{E_3}
The 3-globular operad $\bs{E_3}$ for 3-categories with weak units in low dimensions is the \textit{contractible} 3-globular operad with

\begin{list}{$\bullet$}{}

\item the same 0 and 1-cells as $\bs{E_4}$; and

\item the same 2-cell generators as $\bs{E_4}$, subject to the same equalities, as well as the following additional equalities.

\begin{enumerate}[i)]

\item $v_2 \circ (\text{id}_2,i_2) = \text{id}_2$

\item $v_2 \circ (i_2, \text{id}_2) = \text{id}_2$

\item $h_2 \circ (i_2,i_2) = i_2 \circ (h_1)$

\end{enumerate}
\end{list}
\end{definition}

An algebra for $\bs{E_3}$ on a 3-globular set $\bs{A}$ is the same as an algebra for $\bs{E_4}$ in dimensions 0 and 1. However, the 2-cells of $\bs{E_3}$ satisfy three extra axioms imposed by the additional 2-cell relations. The first two yield left and right unit axioms, respectively, for vertical 2-cell composition. The third yields an axiom stating that the horizontal composite of two 2-cell identities is an identity 2-cell.

\begin{notation}
We denote by $\mathbf{WkUnit} \ n \mhyphen \mathbf{Cat}_{str}$ the category $\bs{E_n} \mhyphen \mathbf{Alg}$ of algebras for $\bs{E_n}$. 
\end{notation}

By Lemmas \ref{contractible n-morphisms} and \ref{map is determined at k-cell generators}, a morphism $f:\bs{W_n} \longrightarrow \bs{E_n}$ of $n$-globular operads is completely determined by its value on the $k$-cell generators of $\bs{W_n}$ for all $k<n$. Thus we can construct a canonical morphism $f:\bs{W_3} \longrightarrow \bs{E_3}$ of 3-globular operads by sending the $k$-cell generators in $\bs{W_3}$ with a corresponding 3-cell generator in $\bs{E_3}$ to their counterparts, and setting $f(a_2) = i_2 \circ (h_1 \circ (id_1,h_1)) =  i_2 \circ (h_1 \circ (h_1, id_1)) = f(a_2')$. There is an analogous morphism $\bs{W_4} \longrightarrow \bs{E_4}$ of 4-globular operads. These morphism are clearly full on the underlying $n$-globular sets by inspection so, for $n=3$ and $n=4$, $\bs{E_n}$ is an object in category $\mathbf{BCat \mhyphen GOp}_{\bs{n}}$ of $n$-globular operad for biased weak $n$-categories, and is therefore contractible; see Remark \ref{BCat op are contractible}. The induced functor $f:\bs{E_n} \mhyphen \mathbf{Alg} \longrightarrow \bs{W_n} \mhyphen \mathbf{Alg}$ is the canonical inclusion functor
\vspace{1.5mm}
\begin{center}
\begin{tikzpicture}[node distance=4.2cm, auto]

\node (A) {$\mathbf{WkUnit} \ n \mhyphen \mathbf{Cat}_{str}$};
\node (B) [right of=A] {$\mathbf{Wk} \ n \mhyphen \mathbf{Cat}_{str}$};

\draw[right hook->] (A) to node {} (B);

\end{tikzpicture}
\end{center}
and the corresponding natural transformation (Definition \ref{induced nat trans})
\begin{center}
\begin{tikzpicture}[node distance=3.6cm, auto]

\node (A) {$\mathbf{GSet}_{\bs{n}}$};
\node (B) [right of=A] {$\mathbf{Wk} \ n \mhyphen \mathbf{Cat}_{str}$};
\node (a) [node distance=1.8cm, right of=A] {$\Downarrow$};
\node (C) [node distance=0.9cm, below of=a] {$\mathbf{WkUnit} \ n \mhyphen \mathbf{Cat}_{str}$};

\draw[->, bend left=40] (A) to node {} (B);
\draw[->, bend right=10] (A) to node [swap] {} (C);
\draw[->, bend right=10] (C) to node [swap] {} (B);

\end{tikzpicture}
\end{center}
is the one whose component at an $n$-globular set $\bs{A}$ is the canonical strict 4-functor $\bs{A_{W_n}} \longrightarrow \bs{A_{E_n}}$ from the free weak $n$-category on $\bs{A}$ to the free $n$-category with weak units in low dimensions on $\bs{A}$.

We now construct presentations for the $n$-globular operads for $n$-categories with weak interchange laws. Note that 3-categories with weak interchange laws are not the same as Gray categories. Gray categories are strict 3-categories with no direct horizontal composite for 2-cells. On the other hand, 3-categories with weak interchange laws have a direct horizontal composite for 2-cells, but the 2-dimensional interchange law is weak. See \cite{RG1} for a presentation for the 3-globular operad for Gray categories, and a comparsion between this operad and the one for 3-categories with weak interchange laws. 

\begin{definition} The 4-globular operad $\bs{H_4}$ for 4-categories with weak interchange laws is the \textit{contractible}  4-globular operad with

\begin{list}{•}{}

\item a single 0-cell, the identity $\text{id}_0$; 

\item 1-cells consisting of the operadic composites of the 1-cells $i_1$ and $h_1$ whose images under the underlying collection map are as in the presentation for $\bs{W_2}$, subject to the following equalities;

\begin{enumerate}[i)]

\item $h_1 \circ (\text{id}_1, i_1) = \text{id}_1$

\item $h_1 \circ (i_1, \text{id}_1) = \text{id}_1$

\item $h_1 \circ (\text{id}_1, h_1) = h_1 \circ (h_1, \text{id}_1)$

\end{enumerate}

\item 2-cells consisting of the operadic composites of 2-cells $i_2$, $h_2$ and $v_2$ whose images under the underlying collection map are as in the presentation for $\bs{W_3}$, subject to the following equalities;
\begin{enumerate}[i)]

\item $h_2 \circ (\text{id}_2, i_2 \circ (i_1)) = \text{id}_2$
\item $h_2 \circ (i_2 \circ (i_1), \text{id}_2) = \text{id}_2$
\item $v_2 \circ (\text{id}_2,i_2) = \text{id}_2$
\item $v_2 \circ (i_2, \text{id}_2) = \text{id}_2$
\item $h_2 \circ (\text{id}_2, h_2) = h_2 \circ (h_2, \text{id}_2)$
\item $v_2 \circ (\text{id}_2, v_2) = v_2 \circ (v_2, \text{id}_2)$
\item $h_2 \circ (i_2,i_2) = i_2 \circ (h_1)$

\end{enumerate}
and

\item 3-cells consisting of the operadic composites of 3-cells $i_3$, $h_3$, $v_3$ and $c_3$ whose images under the underlying collection map are as in the presentation for $\bs{W_4}$, and 3-cells $s_3$ and $s'_3$ whose images under the underlying collection map are,
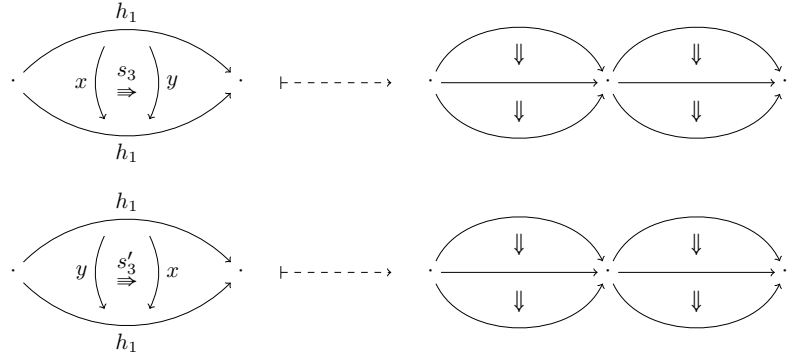
\begin{center}
\resizebox{0.67\textwidth}{!}{
\begin{tikzpicture}[node distance=2cm, auto]

\node (A*) {$\cdot$};
\node (B*) [node distance=3.6cm, right of=A*] {$\cdot$};
\node (a*) [node distance=1.8cm, right of=A*] {};
\node () [node distance=0.15cm, above of=a*] {$s_3$};
\node () [node distance=0.15cm, below of=a*] {$\Rrightarrow$};

\node (c*) [node distance=1.5cm, right of=A*] {};
\node (e*) [node distance=0.7cm, above of=c*] {};
\node (f*) [node distance=0.7cm, below of=c*] {};
\node (d*) [node distance=2.1cm, right of=A*] {};
\node (g*) [node distance=0.7cm, above of=d*] {};
\node (h*) [node distance=0.7cm, below of=d*] {};

\draw[->, bend left=45] (A*) to node {$h_1$} (B*);
\draw[->, bend right=45] (A*) to node [swap] {$h_1$} (B*);
\draw[->, bend right=25] (e*) to node [swap] {$x$} (f*);
\draw[->, bend left=25] (g*) to node {$y$} (h*);

\node (P*) [node distance=3cm, right of=B*] {$\cdot$};
\node (Q*) [node distance=2.8cm, right of=P*] {$\cdot$};
\node (R*) [node distance=2.8cm, right of=Q*] {$\cdot$};

\node (i*) [node distance=1.4cm, right of=P*] {};
\node () [node distance=0.45cm, above of=i*] {$\Downarrow$};
\node () [node distance=0.45cm, below of=i*] {$\Downarrow$};
\node (j*) [node distance=1.4cm, right of=Q*] {};
\node () [node distance=0.45cm, above of=j*] {$\Downarrow$};
\node () [node distance=0.45cm, below of=j*] {$\Downarrow$};

\draw[->, bend left=65] (P*) to node {} (Q*);
\draw[->] (P*) to node {} (Q*);
\draw[->, bend right=65] (P*) to node [swap] {} (Q*);
\draw[->, bend left=65] (Q*) to node {} (R*);
\draw[->] (Q*) to node {} (R*);
\draw[->, bend right=65] (Q*) to node [swap] {} (R*);

\node (x*) [node distance=0.5cm, right of=B*] {};
\node (y*) [node distance=0.5cm, left of=P*] {};
\draw[|->, dashed] (x*) to node {} (y*);

\node (A**) [node distance=3cm, below of=A*] {$\cdot$};
\node (B**) [node distance=3.6cm, right of=A**] {$\cdot$};
\node (a**) [node distance=1.8cm, right of=A**] {};
\node () [node distance=0.15cm, above of=a**] {$s'_3$};
\node () [node distance=0.15cm, below of=a**] {$\Rrightarrow$};

\node (c**) [node distance=1.5cm, right of=A**] {};
\node (e**) [node distance=0.7cm, above of=c**] {};
\node (f**) [node distance=0.7cm, below of=c**] {};
\node (d**) [node distance=2.1cm, right of=A**] {};
\node (g**) [node distance=0.7cm, above of=d**] {};
\node (h**) [node distance=0.7cm, below of=d**] {};

\draw[->, bend left=45] (A**) to node {$h_1$} (B**);
\draw[->, bend right=45] (A**) to node [swap] {$h_1$} (B**);
\draw[->, bend right=25] (e**) to node [swap] {$y$} (f**);
\draw[->, bend left=25] (g**) to node {$x$} (h**);

\node (P**) [node distance=3cm, right of=B**] {$\cdot$};
\node (Q**) [node distance=2.8cm, right of=P**] {$\cdot$};
\node (R**) [node distance=2.8cm, right of=Q**] {$\cdot$};

\node (i**) [node distance=1.4cm, right of=P**] {};
\node () [node distance=0.45cm, above of=i**] {$\Downarrow$};
\node () [node distance=0.45cm, below of=i**] {$\Downarrow$};
\node (j**) [node distance=1.4cm, right of=Q**] {};
\node () [node distance=0.45cm, above of=j**] {$\Downarrow$};
\node () [node distance=0.45cm, below of=j**] {$\Downarrow$};

\draw[->, bend left=65] (P**) to node {} (Q**);
\draw[->] (P**) to node {} (Q**);
\draw[->, bend right=65] (P**) to node [swap] {} (Q**);
\draw[->, bend left=65] (Q**) to node {} (R**);
\draw[->] (Q**) to node {} (R**);
\draw[->, bend right=65] (Q**) to node [swap] {} (R**);

\node (x**) [node distance=0.5cm, right of=B**] {};
\node (y**) [node distance=0.5cm, left of=P**] {};
\draw[|->, dashed] (x**) to node {} (y**);

\end{tikzpicture}
}
\end{center}
where $x=h_2 \circ (v_2,v_2)$ and $y=v_2 \circ (h_2,h_2)$, subject to the equalities below.

\vspace{2mm}

\begin{enumerate}[i)]

\item $h_3 \circ \big(\text{id}_3, i_3 \circ (i_2 \circ (i_1))\big) = \text{id}_3$
\item $h_3 \circ \big(i_3 \circ (i_2 \circ (i_1)), \text{id}_3\big) = \text{id}_3$

\item $v_3 \circ (\text{id}_3, i_3 \circ (i_2)) = \text{id}_3$
\item $v_3 \circ (i_3 \circ (i_2), \text{id}_3) = \text{id}_3$

\item $c_3 \circ (\text{id}_3, i_3) = \text{id}_3$
\item $c_3 \circ (i_3, \text{id}_3) = \text{id}_3$

\item $h_3 \circ (\text{id}_3, h_3) = h_3 \circ (h_3,\text{id}_3)$
\item $v_3 \circ (\text{id}_3, v_3) = v_3 \circ (v_3, \text{id}_3)$
\item $c_3 \circ (\text{id}_3, c_3) = c_3 \circ (c_3, \text{id}_3)$

\item $h_3 \circ (i_3, i_3) = i_3 \circ (h_2)$ 
\item $v_3 \circ (i_3, i_3) = i_3 \circ (v_2)$ 

\end{enumerate}

\end{list}
\end{definition}

An algebra for $\bs{A_{H_4}} \longrightarrow \bs{A}$ for $\bs{H_4}$ is a 4-category with weak interchange laws with underlying 4-globular set $\bs{A}$. Note that in dimensions $\leqslant 2$, our presentation for $\bs{H_4}$ is the same as the presentation for the 2-globular operad $\bs{T_2}$ for strict 2-categories given in Example \ref{pres for T_2}, with the exception of a single relation; $\bs{H_4}$ does not have the 2-cell relation yielding a strict interchange law. The 3-cell generators $s_3$ and $s_3$ provide interchange coherence 3-cells and their weak inverses, respectively. The 3-cell relations yield unit axioms for composition of 3-cells along matching 0, 1, and 2-cell boundaries, three associativity axioms, and axioms stating the composite of two identity 3-cells along a 0 or 1-cell is another identity 3-cell.

\begin{definition}\label{H_3}
The 3-globular operad $\bs{H_3}$ for 3-categories with weak interchange laws is the \textit{contractible} 3-globular operad with the same 0, 1 and 2-cells as $\bs{H_4}$.
\end{definition}

\begin{notation}
We denote by $\mathbf{WkInt} \ n \mhyphen \mathbf{Cat}_{str}$ the category $\bs{H_n} \mhyphen \mathbf{Alg}$ of algebras for $\bs{H_n}$. 
\end{notation}

For $n=3$ and $n=4$ there is a canonical morphism $\bs{W_n} \longrightarrow \bs{H_n}$ of $n$-globular operads given by sending the $k$-cell generators of $\bs{W_n}$ with corresponding $k$-cell generators in $\bs{H_n}$ to their counterparts. There is then a unique choice of image in $\bs{H_n}$ for the remaining $k$-cell generators of $\bs{W_n}$. For example, the 2-cell generator $a_2$ of $\bs{W_n}$ must go to the 2-cell $i_2 \circ ( h_1 \circ (\text{id}_1, h_1)) = i_2 \circ ( h_1 \circ (h_1, \text{id}_1))$ of $\bs{H_n}$. This morphism is clearly full on the underlying 3-globular sets by inspection, so $\bs{H_n}$ is contractible. The induced functor $\bs{H_n} \mhyphen \mathbf{Alg} \longrightarrow \bs{W_n} \mhyphen \mathbf{Alg}$ is the canonical inclusion functor
\begin{center}
\begin{tikzpicture}[node distance=4cm, auto]

\node (A) {$\mathbf{WkInt} \ n \mhyphen \mathbf{Cat}_{str}$};
\node (B) [right of=A] {$\mathbf{Wk} \ n \mhyphen \mathbf{Cat}_{str}$,};

\draw[right hook->] (A) to node {} (B);

\end{tikzpicture}
\end{center}
\vspace{-1mm}
and the corresponding natural transformation
\begin{center}
\begin{tikzpicture}[node distance=3.6cm, auto]

\node (A) {$\mathbf{GSet}_{\bs{n}}$};
\node (B) [right of=A] {$\mathbf{Wk} \ n \mhyphen \mathbf{Cat}_{str}$};
\node (a) [node distance=1.8cm, right of=A] {$\Downarrow$};
\node (C) [node distance=0.9cm, below of=a] {$\mathbf{WkInt} \ n \mhyphen \mathbf{Cat}_{str}$};

\draw[->, bend left=40] (A) to node {} (B);
\draw[->, bend right=10] (A) to node [swap] {} (C);
\draw[->, bend right=10] (C) to node [swap] {} (B);

\end{tikzpicture}
\end{center}
is the natural transformation whose component at an $n$-globular set $\bs{A}$ is the canonical strict $n$-functor $\bs{A_{W_n}} \longrightarrow \bs{A_{H_n}}$ from the free weak $n$-category on $\bs{A}$ to the free $n$-category with weak interchange laws on $\bs{A}$.

\end{document}